\documentclass[reqno]{amsart}
\usepackage{graphicx,graphics}
\usepackage{epstopdf,epsfig}
\usepackage[all]{xy}
\usepackage{textcomp,gensymb}
\usepackage{eucal}
\usepackage[T1]{fontenc}
\usepackage{enumerate}
\usepackage{caption}
\usepackage{subcaption}
\usepackage{amssymb, amsfonts, mathrsfs}
\usepackage[normalem]{ulem}
\usepackage[latin1]{inputenc}

\DeclareGraphicsExtensions{.eps}

\newtheorem*{theorem*}{Theorem}
\newtheorem{theorem}{Theorem}[section]
\newtheorem{corollary}[theorem]{Corollary}
\newtheorem{lemma}[theorem]{Lemma}
\newtheorem{proposition}[theorem]{Proposition}

\theoremstyle{definition}
\newtheorem{definition}[theorem]{Definition}
\theoremstyle{remark}
\newtheorem{example}[theorem]{\bf Example}
\newtheoremstyle{mystyle}{2mm}{0mm}{}{}{\bfseries}{}{ }{\thmnumber{#2}.\thmnote{ #3}}
\theoremstyle{mystyle}
\newtheorem{fact}[theorem]{}
\newtheoremstyle{myremark}{2mm}{0mm}{}{}{\bfseries}{}{ }{\thmname{#1} \thmnumber{#2}. \thmnote{ #3}}
\theoremstyle{myremark}
\newtheorem{remark}[theorem]{Remark}
\newtheorem{remarks}[theorem]{Remarks}

\def\Ga{{T_A}}
\def\cJ{\mc{J}}\def\cR{{R}}\def\cV{{A^1}}\def\cX{{X}}\def\cY{{Y}}\def\cZ{{Z}}\def\cK{{K}}\def\D{D}\def\A{A}
\newcommand\restr[1]{\raisebox{-.5ex}{$|$}_{#1}}

\newcommand\wh[1]{\widehat{#1}}
\newcommand\ol[1]{\mathscr{#1}}
\def\ot{\otimes}

\DeclareMathOperator{\gr}{gr}
\DeclareMathOperator{\Id}{id}
\DeclareMathOperator{\id}{id}

\DeclareMathOperator{\h}{H}

\DeclareMathOperator{\Hom}{Hom}
\DeclareMathOperator{\tor}{Tor}

\DeclareMathOperator{\ext}{Ext}
\DeclareMathOperator{\im}{im}

\DeclareMathOperator{\Tor}{Tor}

\newcommand{\LH}[1]{\mathrm{LH}(#1)}

\newcommand{\ig}[1]{\langle #1\rangle}
\newcommand{\mc}[1]{\mathcal{#1}}
\newcommand{\ev}[1]{\mathrm{ev}_{#1}}
\newcommand{\an}[1]{\mathrm{ann}\left(z^{#1}\right)}

\newcommand{\otp}[1]{ #1\otimes_{\mathfrak{t}_{#1}}S}
\def\M{{\Cl{M}}}

\def\N{\mathbb N}

\def\tz{T_z}

\def\bgr{\cup}\def\bgi{\cap}\def\li{}\def\ts{}

\def\TS{{}_\mK\mathscr{M}_\mK^S}\def\TA{{}_\mK\mathscr{M}_\mK^A}
\def\ST{{_\mK{^S\!\!\!\mathscr{M}}}_\mK}\def\AT{{_\mK{^A\!\!\!\mathscr{M}}}_\mK}
\def\SMS{{}_S\M_S}
\def\mK{\mathbb{K}}\def\mk{\Bbbk}\def\mt{\mathfrak{t}}\def\bts{A\ot_\mt S}\def\ott{\ot_\mt}\def\otK{\ot_\mK}\def\otk{\ot_\mk}
\email{}
\thanks{}

\DeclareGraphicsExtensions{.pdf}
\newcommand\Cl[1]{\mathscr{#1}}

\newcommand{\varfun}[3]{#1 \colon #2 \to #3}

\IfFileExists{tcilatex.tex}{\input{tcilatex}}{}
\newcommand{\CA}{\mathcal A }
\newcommand{\CD}{\mathcal D  }
\newcommand{\CK}{\mathcal K}
\newcommand{\CP}{\mathcal P }
\newcommand{\CR}{\mathcal R }
\newcommand{\CT}{\mathcal T}
\newcommand{\CU}{\mathcal U}

\newcommand{\ov}[1]{\overline{#1}}
\newcommand{\wt}[1]{\widetilde{#1}}
\newcommand{\fd}[2]{{}_{#2}{#1}}
\newcommand{\fu}[2]{{}^{#2}{#1}}

\def\CI{\mathcal I}
\def\CJ{\mathcal J}
\usepackage[dvipsnames]{xcolor}
\usepackage[colorlinks]{hyperref}
\begin{document}



\begin{abstract}
We prove in a very general framework several versions of the classical Poincar\'e-Birkhoff-Witt Theorem, which extend results from \cite{BeGi, BrGa, CS, HvOZ, WW}. Applications and examples are discussed in the last part of the paper.
\end{abstract}

\title{$PBW$-deformations of graded rings}
\author[A. Ardizzoni, P. Saracco]{Alessandro Ardizzoni, Paolo Saracco }
\address{\hspace*{-4.5mm}University of Torino, Department of Mathematics ``G. Peano", Via Carlo Alberto 10, I-10123 Torino, Italy}
\email{alessandro.ardizzoni@unito.it}
\urladdr{\url{sites.google.com/site/aleardizzonihome}}
\address{\hspace*{-4.5mm}University of Torino, Department of Mathematics ``G. Peano", Via Carlo Alberto 10, I-10123 Torino, Italy}
\email{paolo.saracco@gmail.com}
  \urladdr{\url{sites.google.com/site/paolosaracco}}
\author[D. \c Stefan]{Drago\c s \c Stefan}
\address{\hspace*{-4.5mm}University of Bucharest, Faculty of Mathematics and
Computer Science, 14 Academiei Street, Bucharest, Ro-010014, Romania}
\email{dragos.stefan@fmi.unibuc.ro}
\maketitle
\tableofcontents

\section*{Introduction.}

The Poincar\'e-Birkhoff-Witt Theorem states that, for every %finite dimensional 
Lie algebra $L$ over a field, the graded algebra associated to the canonical filtration on the universal enveloping algebra $U(L)$ is isomorphic to the symmetric algebra $S(L)$.

Many authors established theorems like this classical result, working with notions such as: $PBW$-deformations, or $PBW$-basis. We will compare here three of them, since they are related to our work. In order to that, we need some terminology and notation. We fix an associative and  unital ring $S$. Let   $T:=T_S(V)$ denote the tensor $S$-ring generated by some $S$-bimodule $V$. By definition, $T=\oplus_{n\in \N}T^n$ is $\N$-graded. The canonical filtration on $T$, defined by the subbimodules $T^{\leq n}:=\oplus_{p\leq n}T^p$, induces an increasing and exhausting  filtration $X^{\leq n}=X\cap T^{\leq n}$ on an arbitrary set $X$.

For an $S$-subbimodule $P\subseteq T$ let $U(P)$ denote the quotient $S$-ring of $T$ by the ideal $\ig{P}$ generated by $P$.  Let $R_P$ be the graded $S$-subbimodule of $T$ which is generated by the nonzero homogeneous components of highest degree of all elements in $P$. On the other hand, for a graded $S$-subbimodule $R$ of $T$, let $A(R):=T/\ig{R}$.  If  $R\subseteq R_P$, then there exists a natural morphism of graded $S$-rings $\Phi_{P,R}:A(R)\to \gr U(P)$.

By definition, a $PBW$-deformation of $A(R)$ is an $S$-ring $U(P)$, where $P$ is an $S$-subbimodule of $T$ such that  $R\subseteq R_P$ and  $\Phi_{P,R}$ is an isomorphism. Note that, in particular, one can consider the map $\Phi_P:=\Phi_{P,R_P}$. If $U(P)$ is a $PBW$-deformation of $A(R_P)$, then one says that $P$ is of $PBW$-type. For  detailed definitions, the reader is referred to the first section.

Given $R\subseteq T^2$, the $PBW$-deformations of $A(R)$ are investigated in \cite{BrGa,Po,PP}. More precisely, in these papers, the authors consider the case when $S$ is a field, $P$ is a subspace of $T^{\leq 2}$ so that $R:=p^2(P)$, where $p^2:T\to T^2$ denotes the canonical projection. Thus $R\subseteq R_P$ and, under the assumption that $A(R)$ is Koszul, using homological methods and the theory of deformation of (graded) associative algebras, they show that $U(P)$ is a $PBW$-deformation of $A(R)$ if and only if $R_P=R$ and $P$ satisfies the condition:
\[
(PT^{1}+T^{1}P)^{\leq 2}\subseteq P.
\]
When $R$ is the subspace of commutators $[v,w]=v\ot w-w\ot v$, with $v,w\in V$, then $A(R)$ is the symmetric algebra of $V$, and the above identity is equivalent to Jacobi identity. For this reason we shall refer to it as the \emph{Jacobi condition}.

A similar characterization of $PBW$-deformations  is obtained in \cite{BeGi} for $N$-homogeneous $S$-rings over a regular Von Neumann ring $S$, which is not necessarily commutative. Therefore, now, $R\subseteq T^N$ and a $PBW$-deformation corresponds to an $S$-bimodule  $P\subseteq T^{\leq N}$ such that $R:=p^N(P)$, where $p^N$ denotes the projection onto $T^N$. The Koszulity of $A(R)$ is replaced by the assumption that $\Tor_3^{A(R)}(S,S)$, as a graded left $S$-module, is concentrated in degree $N+1$. For $P$ and $R$ as above, in \cite[Theorem 3.4]{BeGi} one shows that $U(P)$ is a $PBW$-deformation of $A(R)$ if and only if $R_P=R$ and $P$ satisfies the relation:
\[
(PT^1 + T^1 P)^{\leq N}\subseteq P,
\]
which will be thought, again, as a Jacobi-like identity. Let us note that if $S$ is a  field and $A(R)$ is $N$-Koszul (see \cite{B2} for the definition of these algebras), then the above assumption on $A(R)$ is satisfied. Since $2$-Koszul algebras and usual Koszul algebras coincide, it follows that the above mentioned theorem extends the result from \cite{BrGa}.

In both $PBW$-like theorems that we discussed above, $R$ is a pure bimodule, that is $R$ is a subbimodule of some $T^n$. To our best knowledge,  in the non-pure case, a first variant of the $PBW$-theorem was proved in \cite{CS}. The method is entirely new. Its main ingredients are central extensions, the complexity of a graded algebra  and the Jacobi condition.

Briefly, these notions are defined as follows. Given a field $S$ and a finite dimensional linear space $V$, let $R\subseteq T$ denote a finite dimensional  graded vector space. Let $R':=\{r_1,\dots,r_m\}$ be  a minimal set of homogeneous relations for $A(R)$. Let $P$ be the linear space generated by $P':=\{r_1+l_1,\dots, r_m+l_m\}$, where each $l_i$ is an element in $T^{<\deg(r_i)}$.
By definition, the central extension  $D(P)$ of $A(R)$ is the quotient of the polynomial ring $T[z]$ by the ideal generated  by $h(P'):=\{(r_1+l_1)^*,\dots, (r_m+l_m)^*\}$, where $f^*\in T[z]$ denotes the external homogenization of $f\in T$, see \cite[\S II.11]{NvO1}.  By construction, the class of $z$ in $D^1(P)$, still denoted by $z$,  is a central element and $D(P)/zD(P)\simeq A(R)$. If $z$ is not a zero divisor in $D(P)$, then one says that $z$ is regular, or that $D(P)$ is regular.

On the other hand, if $A$ is a graded $S$-algebra, with $A^0=S$, then $\ext_n^A(S,S):=\oplus_{m\in\N}\ext^{n,m}_A(S,S)$ is a graded module, with respect to the internal grading induced by that one of $A$.  Thus, the complexity of $A$ is defined by the relation:
\[
c(A):=\sup\{n\mid\ext^{3,n}_A(S,S)\neq 0\}-1.
\]
By definition $P$ satisfies the Jacobi condition if and only if, for all $1\leq k\leq c(A(R))$,
\begin{equation}\label{eq:Jacobi}
P_{k+1}\cap T^{\leq k}\subseteq P_k
\end{equation}
where $\{P_k\}_{k\geq 1}$ is constructed inductively, setting  $P_1:=P^{\leq 1}$ and, for $k>0$, using the recursion formula:
\begin{equation}\label{eq:recursion}
P_{k+1}:=T^1P_k+P_kT^1+P^{\leq k+1}.
\end{equation}
By \cite[Theorem 4.2]{CS},  $U(P)$ is a $PBW$-deformation of $A(R)$  if and only if $P_1=0$ and $P$ satisfies the Jacobi condition. The proof of this version of the $PBW$-theorem is a consequence of the following three steps. First, one demonstrates that $U(P)$ is a $PBW$-deformation of $A(R)$ if and only if $z$ is regular.  Then, one shows that $z$ is regular if and only if it is $c(A(R))$-regular, that is $za\neq 0$, for all non-zero $a\in D^{\leq c(A(R))}$. Finally, one proves that $c(A(R))$-regularity is equivalent to the Jacobi condition.

The main goal of this paper is to extend these $PBW$-type theorems, by relaxing as much as possible the conditions imposed on the rings $S$ and on the $S$-bimodules $P$ and $R$. Our strategy is similar to the one used in \cite{CS}, namely, of showing that the following properties of an $S$-bimodule $P\subseteq T$ are equivalent: $U(P)$ is a $PBW$-deformation of $A(R)$;  $D(P)$ is regular; $P$ satisfies certain Jacobi relations. However, because of the level of generality that we want to achieve, we will relate these properties in a different way. Moreover, significant changes are required, not only in the proofs of our results, but also in the definitions of several basic notions, including the Jacobi condition and the complexity of $A(R)$.

We start by characterizing $PBW$-deformations in terms of appropriate Jacobi conditions. This will be done in the first two sections. For this part we do not need that $T$ to be  a tensor $S$-rings. All definitions and constructions that we work with still make sense for any associative unital graded $S$-ring $T:=\oplus_{n\geq 0}T^n$, which is connected and strongly graded, that is  $T^0=S$ and $T^nT^m=T^{n+m}$ for all $n,m$.

Assuming that $T$  is such an $S$-ring, we first study the basic properties of $PBW$-deformations and consider some examples. For instance, if $R$ is regarded as a filtered $S$-bimodule with respect to the filtration $\{R^{\leq n}\}_{n\in\N}$, then a filtered map $\alpha:R\to T$ yields us a deformation $P:=\alpha(R)$ of $R$, see Example \ref{ex:alpha}. In this case we have  $R_P=R$ and $U(P)$ is a $PBW$-deformation if and only if $P$ is of $PBW$-type. We will say that $P$ is associated to the map $\alpha$.

It is worth to note that, in view of Theorem \ref{te:alpha}, for a bimodule $P$ such that every $P^{\leq n}$ is a direct summand of $P^{\leq n+1}$ as a subbimodule, we have $R=R_P$ if and only if $P$ is associated to a certain  $\alpha$ as above. Consequently,  by Corollary \ref{co:BG}, it follows that any deformation of an  $N$-pure bimodule $R$ as in \cite{BeGi} is of this type.

More generally, we show that $U(P)$ is a $PBW$-deformation of $A(R)$ if and only if $P$ is of $PBW$-type and $\ig{R}=\ig{R_P}$. For many examples, the latter relation can be checked by a straightforward computation, or by identifying $P$ with $\alpha(R)$, for some map $\alpha$ as above.  Therefore, in this situation, for concluding that $U(P)$ is a $PBW$-deformation it is enough to show that $P$ is of $PBW$-type.

In the view of the foregoing remark, in the second section we focus ourselves on finding necessary and sufficient conditions for a bimodule $P$ to be of $PBW$-type. One of the main results of this type is Theorem \ref{Th:PBWJac}, where we prove that $P$ is of $PBW$-type if and only if it satisfies the relation \eqref{eq:Jacobi}, for all $k\in \N$. We define the bimodules $P_k$ by the same recursion  formula \eqref{eq:recursion}, as in \cite{CS}, but we start with $P_0:=P\cap T^0=P^{\leq 0}$. Therefore, in general, the two filtrations may have different terms $P_k$, for $k>0$. Nevertheless, these bimodules coincide for all $k>0$, provided that  $P^{\leq 1}=0$. We will refer to relation \eqref{eq:Jacobi} as the Jacobi condition $(\cJ_k)$.

Furthermore, we define the central extension  of $A(R)$ by $D(P):=T[z]/\ig{P^*}$, where $P^*$ contains all elements $f^*$ obtained by external homogenization, where  $f$ is arbitrary in $P$. As in \cite{CS}, we say that the central extension $D(P)$ is regular if and only if $z$ is not a zero divisor in $D(P)$ or, equivalently, the annihilator $\an{}$ of $z$ is trivial.  By Theorem \ref{Th:JacAnn}, it follows that $P$ satisfies $(\cJ_n)$ if and only if , the $n$-degree component of $\an{}$ is trivial. In particular, $P$ is of $PBW$-type if and only if $z$ is regular.

In light of the previous description of bimodules of $PBW$-type, in the third section we study in detail the regularity of  central extensions. In fact,  for an arbitrary  strongly graded connected $S$-ring $D$ and a central element $z\in \D^1$, we will find homological obstructions which prevent $z$ of being regular.

Since, in general,  there is not a minimal projective resolution of $S$ as a left $A$-module, we  need a substitute of it, with similar properties in small degrees. In the case when $A^1$ is projective as a  left $S$-module, we can choose it as follows
\begin{equation}\label{eq:resolution}
\cdots\rightarrow  A \ot V_n\xrightarrow{d_n} A \ot V_{n-1}\xrightarrow{d_{n-2}}\cdots  \xrightarrow{d_3} A \ot V_2\xrightarrow{d_2} A \ot A^1\xrightarrow{d_1} A  \xrightarrow{d_0} S\rightarrow 0,
\end{equation}
where $V_*$ are projective graded left $S$-modules, $d_*$ are graded morphisms and $V_n^i=0$, for all $n>1$ and $i=0,1$.
Using this resolution  we produce a sequence
\[
\cdots \xrightarrow{\delta_3} M_2  \xrightarrow{\delta_2}M_1\xrightarrow{\delta_1} M_0\xrightarrow{\delta_0}S\xrightarrow{}0
\] of graded $D$-modules $M_n$ and graded $D$-linear maps $\delta_n:M_n\to M_{n-1}$; see \S\ref{ssec:3.5} for more details on the construction of this sequence. The sequence defined by the maps $\delta_0,\dots,\delta_{p+1}$ is denoted by $^p M_*$. If $p=1$ this sequence is a complex which is exact in degree $0$. In the main result of this part, Theorem \ref{th:mainth}, we prove that $z$ is regular if and only if $^2M_*$ is a complex and $\h_1({}^1M_*)=0$. Moreover, if $z$ is regular, then $(M_*,\delta_*)$ is a resolution of $S$ by projective $D$-modules.

To measure the size of a graded $S$-module $V=\oplus_{p\in\N}V^p$ we introduce the complexity of $V$, an invariant that we denoted by $c_V$. By definition, $c_V=-1$ for $V=0$, or $c_{V}:=\sup\{n\mid V^n\neq 0\}-1$, otherwise. Note that we also have $c_V=-1$, if $V$ is not trivial and $V=V^0$.  Recall that $V_3$ denotes the graded $S$-module defining the third term of the resolution \eqref{eq:resolution}. The complexity of $V_3$, plays an important role,  because $z$ is regular if and only if $z$ is $c_{V_3}$-regular and $\h_1({}^1M_*)$, cf. Theorem \ref{th:regandcompl}.

We have already noticed that a similar result was proved in \cite{CS} for central extensions. Since in that paper $T$ is a tensor algebra, for any central extension $D(P)$ of $A(R)$ the first homology group of the corresponding complex $^1M_*$ is always trivial.  This property does not hold if $T$ is not a tensor $S$-ring, even if $S$ is a field. An example, showing that the condition $\h_1({}^1M_*)=0$ cannot be dropped in general, will be presented in the last section. On the other hand, in the fourth section, we consider a special case of central extensions for which the complex $^1M_*$ is exact in degree $1$.

Let $A$ be a connected strongly graded $S$-ring, Thus, there is a surjective canonical morphism of graded $S$-rings from $T_A:=T_S(A^1)$ to $A$. Let $I_A$ denote its kernel. We assume that there exists a bimodule of relations $R$ for $A$. Therefore, by definition, $R$ generates $I_A$ and $$R\cap (A^1I_A+I_AA^1)=0.$$ Under the additional assumption that $R$ is projective as a left module, we show that in the resolution \eqref{eq:resolution} we can take $V_2=R$.

Let $D$ be a central extension associated to a filtered map $\alpha:R\to T_A$, so $D=D(P)$, where $P=\alpha(R)$. Under the above hypothesis we are able to construct a bimodule of relations $R_D$ for $D$ and to identify the complex $^1M_*$ associated to $D$ with the exact sequence
\[
D\ot R_D\xrightarrow{d'_2}D\ot D^1\xrightarrow{d'_1}D\rightarrow S\rightarrow 0,
\]
which is the beginning of the resolution of $S$ by projective left $D$-modules, similar to \eqref{eq:resolution}. Hence $^1M_*$ is exact in degree $1$, and we conclude in this situation that $z$ is regular if and only if either $c_{V_3}=-1$, or $z$ is $c_{V_3}$-regular, see Theorem \ref{th:homological_condition}.

In the last section of the paper we prove our versions of the $PBW$-theorem. For simplicity we present here only one of them. Let $A$ be a strongly graded $S$-ring such that there exists an $R$  bimodule of relations for  $A$. We suppose that $A^1$ and $R$ are projective left $S$-modules. For any filtered map $\alpha:R\to T_A$, we show in Theorem \ref{te:PBW} (b) that $P:=\alpha (R)$ is of $PBW$-type (equivalently, $U(P)$ is a PBW-deformation of $A$)  if and only if either $c(A)=-1$ or $P$ satisfies $(\cJ_1)-(\cJ_{c(A)})$.

Here $c(A)$  denotes the homological complexity of $A$, that is the complexity of the graded $S$-module $\tor_{3}^A(S,S)$. We note that $c(A)\leq c_{V_3}$ for any graded $S$-module $V_3$ that appears in the third term of a resolution as in \eqref{eq:resolution}. Thus $P$ is of $PBW$-type, provided that it satisfies the Jacobi conditions $(\cJ_1)-(\cJ_{c_{V_3}})$. In conclusion, $c(A)$ is the least number of Jacobi relations that $P$ must satisfy for being of $PBW$-type.

In the remaining part of the paper we present some examples and applications. Let $A$ be a strongly graded $S$-ring. The case when $R$ is an $N$-pure bimodule of relations is settled in Theorem \ref{te:PBW-pure}. This result is similar to \cite[Theorem 3.4]{BeGi}, but instead of working with a regular Von Neumann ring $S$, we impose that $R$ and $A^1$ are projective left $S$-modules.

The $PBW$-deformations of twisted tensor products are characterized in Theorem \ref{te:PBW_twisted} and Theorem \ref{te:PBW-winterspoon}. Specializing the first  result to the case when $S$ is a field and $A$ is a Koszul algebra we get \cite[Theorem 4.6.1]{HvOZ}. For smash products of Koszul algebras by finite dimensional Hopf algebras,  we recover \cite[Theorems 0.4 and 3.1]{WW}.

Specific forms of the $PBW$-theorem are obtained when $S$ belongs to some special classes of ring, such as: of weak dimension zero  (Theorem \ref{te:wd0}), separable algebras (Theorem \ref{PBW_separable}) and multi-Koszul algebras (Theorem \ref{te:PBW_multi_Koszul}).

Choosing in a convenient way the generators and relations,  we prove a $PBW$-type theorem for a graded $S$-ring $A:=T/I$, where $T$ is a quadratic $S$-ring, with $S$ a separable algebra (Theorem \ref{te:PBW_quadratic}). In particular, we can apply the latter result to the case when $T$ is a polynomial ring, to obtain a $PBW$-theorem for commutative algebras (Theorem \ref{co:PBW_quadratic}).

Among the examples considered in subsection \ref{ssec:examples}, we mention here split central extensions (Example \ref{rem:trivcomplex}) and trivial central extensions (Example \ref{ex:trivial_extension}), which yield us counterexamples to the vanishing of the first homology group of the corresponding complex $^1M_*$. We also show that the homological complexity of $A$ is finite, provided that the complexity of $A$ as a graded $S$-module is so and, in this case we have $c(A)\leq c_A+2$, cf. \S\ref{co:3n(A)}.

\section{Preliminary results.}\label{sec:1}
\setcounter{subsection}{1}
	Throughout this paper we fix an associative unital ring $S$  {and a field $\Bbbk$}. If $\cX=\oplus_{n\in \mathbb{Z}}\cX^n$ is a graded $S$-bimodule then the $d$-shifting of $\cX$ is the graded $S$-bimodule $\cX(d)$ which coincides with $\cX$ as a bimodule, but its grading is given by $\cX(d)^n=\cX^{n+d}$. If  $X$ and $Y$ are graded $S$-bimodules, then a morphism of $S$-bimodules $f:X\to Y$ is said to be of \emph{degree $d$} if $f\left(X^n\right)\subseteq Y^ {n+d}$. A morphism of degree zero will be simply called a graded morphism. Thus $f:X\to Y$ is is a morphism of degree $d$ if and only if $f:X\to Y(d)$ is graded (of degree $0$).

	Let $T:=\oplus_{n\geq 0}T^{n}$ be an associative, unital and positively graded $S$-ring. In addition, we assume that $T$ is \emph{connected}, that is  $T^0=S$. In this case every homogeneous component $T^n$ is an $S$-subbimodule of $T$. We shall denote by $T^{\leq m}:=\oplus_{n\leq m}T^{n}$ the $m$-th term of the standard filtration on $T$. The projection  onto the homogeneous component $T^n$ will be denoted by $p^n:T\to T^n$.

	If $U:=\left\{F^nU\right\}_{n\geq 0} $ is an increasing filtration on an algebra $U$,  then the associated graded algebra will be denoted by $\gr{U}=\oplus_{n\geq 0}\gr^{n} U$. The  $n$-th homogeneous component of $\gr U$  is defined for $n\geq 0$ by $\gr^n U :=(F^nU)/(F^{n-1}U)$, where $F^{-1}U=0$, by convention.

	Let $P$ denote an $S$-subbimodule  of $T$. Let $U(P)$ be the quotient $S$-ring of $T$ by $\ig{P}$, the ideal generated by $P$. This $S$-ring  is filtered, with respect to the filtration given by:
\begin{equation*}
	U(P) ^{\leq m}=\frac{T^{\leq m}+\ig{P} }{\ig{P}  }\cong \frac{T^{\leq m} }{T^{\leq m}\bgi\ig{P} }.
\end{equation*}%
	Let $\pi_P :T\longrightarrow U(P) $ denote the canonical projection. Taking on $T$ the standard filtration, $\pi_P$ is a filtered algebra map. The goal of this paper is to investigate the graded $S$-ring  $\gr U(P)$, by comparing it with another graded $S$-ring which will be  defined as a quotient of $T$ by a homogeneous ideal canonically associated to $P$. Basically, the homogeneous ideal that we are going to construct is generated by the highest degree homogeneous parts of all elements in $P$.

	More precisely, following \cite{Li}, for every $a\in T\setminus \left\{ 0\right\} $ we set  $\LH{a}:=a_m$, where $a=\sum_{i=0}^m a_i$, with $a_{i}\in T^{i}$ and $a_{m}\neq 0$. By definition, we set $\LH{0} =0$. This construction yields us a map $\mathrm{LH} :T\longrightarrow\bgr_{n\in \N}T^{n}$, which in general is {not linear}. For a subset $X\subseteq T$ we denote by $\LH{X}$ the set of all elements $\LH{x}$ with $x\in X$.

	In particular, if  $X\subseteq T$ is an {$S$-}subbimodule of $T$ and $n\in \mathbb{N}$, we  define $R_{X}^{n}:=\LH{X} \bgi T^{n}$ and we {set $X^{\leq n}:=X\bgi T^{\leq n}$.} The first elementary  properties of these sets are proved in the following lemma.

\begin{lemma}\label{Lemma:Sp}
	We have $R_X^{n}=p ^{n}\left( {X^{\leq n}}\right) $ for every $n\in \mathbb{N}$. Thus, $R_X^{n}$ is an $S$-subbimodule of $T$ and  $\LH{X} =\bgr_{n\in \N}R_X^{n}$. Hence, it makes sense to define the $S$-bimodule $R_{X}:=\oplus_{n\in \N}R_X^{n}$ and  we have the relation $\left\langle \mathrm{LH}\left(X\right) \right\rangle =\left\langle R_X\right\rangle $.
\end{lemma}

\begin{proof}
	If $x\in \LH{X}\bgi T^{n}$ then there exists $y\in X$ such that $x=\mathrm{LH}\left(y\right) $. In particular, $y-x$ is an element in $T^{\leq n-1}$ and {hence} $x=p ^{n}\left(y\right) \in p ^{n}\left( {X^{\leq n}}\right)$. To prove the other inclusion, let  $x\in p ^{n}\left( {X^{\leq n}}\right)$. Thus $x=p ^{n}\left( y\right)$ for some $y \in  {X^{\leq n}}$. We can assume that $x\neq 0$. Since $x\in T^{n}$ and $y-x\in T^{\leq n-1}$ it follows that $x=\LH{y}$. Thus $x\in\LH{X} \bgi T^{n}$. Furthermore, as $\LH{X} \subseteq \bgr_{n\in \N}T^{n}$, we get:
\begin{equation*}\ts
	\bgr\li_{n\in \N}R_X^{n}=\bgr\li_{n\in \N}\big( \mathrm{LH}\left( X\right)\bgi T^{n}\big) =\LH{X} \bgi\left( \bgr\li_{n\in \N}T^{n}\right)=\LH{X}.
\end{equation*}
	Hence $\left\langle \mathrm{LH}\left( X\right) \right\rangle =\left\langle\bgr_{n\in\N}R_X^{n}\right\rangle =\left\langle \oplus_{n\in\N}R_X^{n}\right\rangle =\left\langle R_X\right\rangle $.
\end{proof}

\begin{remark}\label{rem:isos}
	Let $P$ denote an $S$-subbimodule of $T$. The bimodule $R_P$ coincides with the\emph{ internal homogenization} {$P\,\,\widetilde{}=\text{Span}_{S-S}(\mathrm{LH}(P))$} defined in \cite[\S II.1]{NvO1}. Furthermore, if we denote by $\text{gr}P$ the graded module associated to the filtration $P^{\leq n}$, then the assignment $z+P^{\leq n-1}\mapsto p^n(z)$ for all $z\in P^{\leq n}$ gives an isomorphism  $\overline{p}^n:\gr^n P\to R_P^n $ of $S$-bimodules. We will denote by $\overline{p}$ the isomorphism of graded $S$-bimodules defined by the family $\{\overline{p}^n\}_{n\in\N}$.
\end{remark}
\vspace*{1ex}

\begin{lemma}\label{lem:grsum}Let $X$ and $Y$ be $S$-subbimodules of $T$.
\begin{itemize}
  \item [(1)] Assume $R^q_X\cap R^q_Y=0$ for every $q>n$. Then $(X+Y)^{\leq n}=X^{\leq n}+Y^{\leq n}$ and $R^n_{X+Y}=R^n_X\oplus R^n_Y$.
  \item [(2)]Let $f:T\to T'$ be an injective graded map. Then $R_{f(X)}=f(R_X)$.
\end{itemize}
\end{lemma}

\begin{proof}
	(1). Clearly, $X^{\leq n}+Y^{\leq n}\subseteq (X+Y)^{\leq n}$. Take an element $w\in (X+Y)^{\leq n}$ such that $\mathrm{LH}(w)\in T^m$, for some $m\leq n$. We write $w=x+y$, with $x\in X $ and $y\in Y$. If $x=0$ or $y=0$ then, obviously, both belongs to $T^{m}$.  Therefore, we can  assume that $x$ and $y$ are nonzero. Let $q,r$ such that $\mathrm{LH}(x)\in T^q$ and $\mathrm{LH}(y)\in T^r$. If $r<q$, then $\mathrm{LH}(x+y)=\mathrm{LH}(x)$ and $q= m$, so that $x\in X^{\leq n}$ and $y\in Y^{\leq n}$. The case  $r>q$ can be handled in a similar way. On the other hand, if $q=r>n$ then $0=p^q(w)=p^q(x)+p^q(y)$. Thus $p^q(x),p^q(y)\in R^q_X\cap R^q_Y=0$, which is not possible since $\mathrm{LH}(x)=p^q(x)$ cannot be zero. Thus $q=r\leq n$ and hence $x\in X^{\leq n}$ and $y\in Y^{\leq n}$. We have $$R^n_{X+Y}=p^n((X+Y)^{\leq n})=p^n(X^{\leq n}+Y^{\leq n})=p^n(X^{\leq n})+p^n(Y^{\leq n})=R^n_X+ R^n_Y.$$ The latter sum is obviously direct.

(2). Since $f$ is injective we have $(f(X))^{\leq n}=f(X^{\leq n})$ so that
\[
	R_{f(X)}^n=p_{T'}^n((f(X))^{\leq n})=p_{T'}^nf(X^{\leq n})=f^np_T^n(X^{\leq n})=f^n(R_X^n)=f(R_X^n).\qedhere
\]
\end{proof}

	For any graded $S$-subbimodule $R:=\oplus_{n\in\N}R^n$ {of $T$} let $A(R)$  denote the quotient of $T$ by $\ig{R}$. Note that $A(R)$ is a graded $S$-ring such that  $A(R)^0=S/R^0$. {Following \cite{BrGa}, for an $S$-subbimodule $P$ of $T$}, we call $A(R_P) $ the \emph{homogeneous algebra associated to} $P$. On the other hand, we shall refer to $U(P) $ as the \emph{non-homogeneous algebra associated to} $P$.

	Now we can relate $A(R_P)$ and $\gr U(P)$. The $S$-bilinear maps
\begin{equation*}
	\varphi_P ^{n}:T^{n}\longrightarrow \gr^{n}U( P), \quad \varphi_P ^{n}(x)= \left( x+\langle P\rangle \right) +U( P)^{\leq n-1}.
\end{equation*}
	are the components of a morphism $\varphi_P :=\oplus_{n\geq 0}\varphi_P ^{n}$ of graded $S$-rings.  Note that $\varphi_P $ is surjective, since for every $x\in T^{\leq n}$ we can write $x=\sum_{i=0}^{n}x_{i}$, {with $x_i\in T^i$ for all $i=0,\ldots,n$,} so  we have:
\[
	\left( x+\left\langle P\right\rangle \right) +U(P) ^{\leq n-1}=\left(x_{n}+\left\langle P\right\rangle \right) +U(P) ^{\leq n-1}=\varphi_{P}^{n}\left( x_{n} \right) .
\]

\begin{remark}\label{rem: Ker_varphi}
	The kernel of the map $\varphi_P^n$ is given by the relation $\ker(\varphi_P^n)=R_{\langle P\rangle}^n$.  Indeed, for any $x\in T^n$, the element $x$ belongs to the kernel of  $\varphi_P^n$ if and only if $x+\langle P\rangle\in U(P)^{\leq n-1}$. Therefore the required property holds if and only if there exists $y\in T^{\leq n-1}$ such that $x+y$ is an element of $\ig{P} $ or, equivalently, $x$ belongs to $\LH{\ig{P}}\bigcap T^n=R^n_{\langle P\rangle}$. Obviously, $R_{\langle P\rangle}$ is a homogeneous ideal, as it is the kernel of a morphism of graded rings. Since $R_P\subseteq R_{\ig{P}}$ we deduce that $\ig{R_P}\subseteq R_{\ig{P}}$.
\end{remark}

	It follows that $\varphi_P ^{n}$ factorizes through a surjective morphism
\begin{equation*}
	\Phi _{P}^{n}:A(R_P) ^{n}\longrightarrow \gr^{n}U(P),\quad  \Phi _{P}^{n}(x+\left\langle R_P\right\rangle)= \left( x+\left\langle P\right\rangle \right) +U(P) ^{\leq n-1}.
\end{equation*}%
	The family $\{\Phi _{P}^{n}\}_{n\in\N}$  defines a  graded $S$-ring map {$\Phi_P:=\bigoplus_{n\geq 0}\Phi_P^n$} from $A(R_P)$ to $\gr U(P) $. The kernel of the map $\Phi_P^n$ is given for every $n\geq 0$ by the relation:
 \begin{equation}\label{rem: Ker_Phi}
   \ker( \Phi_P^n)=\frac{\ker(\varphi_P^n)+{\langle R_P\rangle}}{\langle R_P\rangle}=\frac{R_{\langle P\rangle}^n+{\langle R_P\rangle}}{\langle R_P\rangle}.
 \end{equation}

\begin{lemma}\label{le:Phi_functorial}
The construction of $\Phi _{P}$ is natural in $P$. More precisely, let $f:T\to T'$ denote a  morphism of connected graded $S$-ring and let $P$ and $P'$ be $S$-subbimodules of $T$ and $T'$, respectively. If $f(P)\subseteq P'$, then the following diagram is commutative:
\begin{equation}\label{dia:Phi}
\begin{array}{c}\xymatrix@C=30pt@R=20pt{
	A(R_P) \ar[d]_{\ov f} 				\ar[r]^{\Phi _{P}} & \mathrm{gr}U(P) 		\ar[d]^{\wt f }\\
    A(R_{P'})\ar[r]_{\Phi _{P'}} &   \mathrm{gr}U( P')}
\end{array}
\end{equation}
	where $\ov{f}$ and $\wt f$ are the canonical maps induced by $f$.
\end{lemma}

\begin{proof}
	Let $r\in R_P^n$ {with $r\neq 0$}. By definition, there exists $p\in P\bgi T^{\leq n}$ such that $p=r+q$ where $q\in T^{\leq n-1}$, i.e. $r=\mathrm{LH}(p)$ and $f(p)=f(r)+f(q)$. If  $f(r)\neq 0$, then $f(r)=\mathrm{LH}(f(p))$, as $f$ is graded. Therefore, $f(r)\in \LH{P'}=R_{P'}^n$. We deduce that  $f(R_{P})\subseteq R_{P'}$. Thus there is a graded $S$-ring map $\ov f:A(R_{P})\to A(R_{P'})$. Similarly, $f$ induces a morphism $f': U(P)\to U(P')$ of filtered algebras, so we can consider the morphism of graded $S$-rings $\wt f:\gr U(P)\to\gr U(P')$, where $\wt f:=\gr f'$. It is easy to see that $\ov f$ and  $\wt f$ make the diagram \eqref{dia:Phi} commutative.
\end{proof}

	In order  to introduce the main object of investigation of this paper, let us take $R$ to be a graded $S$-subbimodule  of $R_P$. Thus $R:=\oplus_{n\geq 0}R^{n}$ and each $R^n$ is a subbimodule of $R^n_P$. Hence, there is a graded algebra morphism $\zeta_{P,R}:A(R)\longrightarrow A(R_P)$ which maps $x+\left\langle R\right\rangle$ to $x+\left\langle R_P\right\rangle $.

\begin{fact}[Main objective.]\label{fact:mainaim}
	The main goal of this paper is to find necessary and sufficient conditions for the natural composition $\Phi _{P,R}:=\Phi_P  \zeta_{P,R}:A(R)\to \gr U(P)$ {to be} a graded $S$-ring isomorphism. In this case, following \cite[Definition 1.1]{CS}, we shall say that $U(P)$ is a \emph{PBW-deformation} of $A(R)$.
\end{fact}

	We start with an easy but useful lemma,  showing us that $\Phi _{P,R}$ is an isomorphism if and only if $\Phi _{P}$ is an isomorphism and $\ig{R}=\ig{R_P}$.  This will allow us to approach in a unifying way several well-known results related to the existence of $PBW$-deformations of certain types of algebras, e.g. \cite{BeGi, BrGa, CS}.

\begin{lemma}\label{le:RS}
	The map $\Phi _{P,R}$ is an isomorphism if and only if $\Phi _{P}$ and $\zeta _{P,R}$ are both isomorphisms.
\end{lemma}

\begin{proof}
	By definition, $\Phi _{P,R}$,  $\Phi _{P}$ and $\zeta _{P,R}$ are surjective. If $\Phi _{P,R}$ is an isomorphism, then $\zeta _{P,R}$ is injective. Thus $\zeta _{P,R}$ is bijective and, a fortiori, $\Phi_P$ is bijective too.
\end{proof}

\begin{definition}\label{def:$PBW$-type}
	An $S$-subbimodule  $P\subseteq T$ is  of   \emph{$PBW$-type} if and only if $\Phi_P $ is an isomorphism.
\end{definition}

\begin{remark}\label{rem:large}
	 Note that, since $R\subseteq R_P$, the map $\zeta _{P,R}$ is bijective if and only if $R$ generates the ideal $\ig{R_P}$. Thus, in this case  $A(R)=A(R_P)$ and $\Phi _{P,R}=\Phi _{P}$. Therefore, $U(P)$ is a $PBW$-deformation of $A(R)$ if and only if $P$ is of $PBW$-type and $R$ generates $\ig{R_P}$.
\end{remark}

\begin{remark}\label{rem: $PBW$-type}
 	Note that the inclusion $\ig{R_P}\subseteq R_{\ig{P}}$ always holds,  as $R_{\ig{P}}$ is an ideal which contains  $R_P$. Taking into account the relation \eqref{rem: Ker_Phi}, the bimodule $P$ is of $PBW$-type if and only if  $R_{\ig{P}}\subseteq \ig{R_P}$. Moreover, in view of Remark \ref{rem:isos},  we can identify $R_X$ and $\text{gr}X$, for every submodule $X\subseteq T$. Then $P$ is of $PBW$-type if and only if $\text{gr}\langle P\rangle=\langle\text{gr}P\rangle$.
\end{remark}

\begin{example}\label{ex:alpha}
	A generic class of examples of $S$-bimodules $P$ and $R\subseteq R_P$, such that $\zeta_{P,R}$ is an isomorphism, can be constructed as follows. Let $R$ be a graded $S$-subbimodule of $T$. We can regard $R$ as a filtered $S$-bimodule with respect to the filtration $R^{\leq n}=\oplus_{m=0}^nR^m$.

	We assume that $\alpha:R\to T$ is a given morphism of filtered $S$-bimodules, that is $\alpha \left( R^{n}\right)\subseteq T^{\leq n}$ for all $n\in \mathbb{N}$.   For $i\in\N$ we define $\alpha_i:R\to T$ to be the unique $S$-bimodule map such that its restriction to $R^n$ coincides with  $p^{n-i}  \alpha$. Note that $\alpha_i(r)=0$, whenever $n<i$ and $r\in R^n$.

	By definition, $\alpha_i$ belongs to $\Hom_{S-S}(R,T)_{-i}$ , the set of $S$-bimodule map of degree $-i$. For simplicity we shall write $\alpha=\sum_{i=0}^\infty\alpha_i$, although the sum  does not make sense in general. By the above relation, we just mean that the sum $\sum_{i=0}^\infty\alpha_i(r)$ makes sense for all $r\in R$, because only a finite number of its terms are nonzero, and  it equals $\alpha(r)$. We will call $\alpha_i$ the \emph{component of degree} $-i$ of the \emph{filtered} map $\alpha$.

	Conversely, any family $\{\alpha_i\}_{i\in \N}$ with $\alpha_i\in\Hom_{S-S}(R,T)_{-i}$, defines in a unique way a filtered map $\alpha:R\to T$ by the relation:
\[
	\alpha(r)=\sum_{i=0}^\infty\alpha_i(r),
\]
	for all $r\in R$. \textit{Throughout we will  assume that the component $\alpha_0$ of a filtered map $\alpha:R\to T$ is the inclusion of $R$ into $T$.} We point out that under this assumption $\alpha$ has to be injective.

	Let $\alpha:R\to T$ be a filtered map. Let  $P:=\alpha(R)$.  We claim that $R=R_P$ so, in this case, we have  $\zeta_{P,R}=\id_{A(R)}$ as well.

	Indeed, for every $r\in R$, we have $\mathrm{LH}\left(\alpha(r)\right)=\mathrm{LH}\left(\alpha_0(r)\right) =\mathrm{LH}\left( r\right) $, since $\deg(\alpha_i)=-i$ and  $\alpha_0(r)=r$. Thus, for every $x\in R_P^n$, there is $r\in R$ so that $x=\mathrm{LH}\left(\alpha \left( r\right) \right) =\mathrm{LH} \left( r\right)$.  Then $x\in R\bgi T^{n}=R^{n}$. Conversely, for every $r\in R^{n}$, we have $\mathrm{LH}\left(\alpha \left( r\right) \right) =r$. Therefore, $r\in \mathrm{LH}\left( P\right) \bgi T^{n}=R_P^n$.

	By Remark \ref{rem: $PBW$-type}, in this  particular case, $U(P)$ is a $PBW$-deformation of $A(R)$, that is $\Phi_{P,R}$ is an isomorphism, if and only  if $P$  is of $PBW$-type.
\end{example}

\begin{example}\label{ex:CS}
 	We use the notation and the assumptions from \cite{CS}. Thus,   $T$ denotes the free $\Bbbk$-algebra generated by the indeterminates $\{X_1,\dots,X_n\}$. Note that $S=T^0=\Bbbk$. We fix two subsets $R^{\prime }:=\{r_1,\dots,r_m\}$ and $P':=\{r_1+l_1,\dots,r_m+l_m\}$ of $T$ as in the foregoing mentioned paper. We denote by $R$ and $P$ the $\Bbbk$-linear spaces spanned by $R^{\prime }$ and $P^{\prime }$, respectively. By assumption $R^{\prime }$ is a minimal set of homogeneous generators of the ideal $\ig{R'}=\ig{R}$, cf. \cite{CS}. Thus, $R'$ is a linearly independent set. Hence we can define a unique linear map $\alpha:R\to T$ such that $\alpha(r_i)=r_i+l_i$. Hence, $P=\alpha(R)$.

	By Example  \ref{ex:alpha} we have $R=R_P$, so $\Phi_{P,R}$ is an isomorphism if and only if $P$ is of $PBW$-type. The fact that $\Phi_{P,R}$ is an isomorphism means that ``the deformation $T/\ig{P'}$, of the graded $\Bbbk$-algebra $T/\ig{R'}$, is a $PBW$-deformation", cf.  \cite[Definition 1.1]{CS}.
\end{example}

\begin{example}\label{ex:BG}
 	As another instance of Example  \ref{ex:alpha} we consider the algebras introduced in \cite{BeGi}. In this paper the authors work  with quotients of the tensor algebra $T:=T_S(V)$ of an $S$-bimodule $V$,  where $S:=T^0$ is a regular Von Neumann ring. By assumption,  $P$ is an $S$-subbimodule  of $T^{\leq n}$ and one takes $R$ to be the graded $S$-bimodule with trivial homogeneous components, excepting  $R^n:=p^n(P)$, where $p^n :T\to T^{n}$ denotes the canonical projection. To refer  to this situation, we will say that $R$ in $n$-\textit{pure}. We will also say that a graded bimodule is \textit{pure} if it is $n$-pure for some $n$.

	Clearly, $R= p^n(P^{\leq n})=R_P^n\subseteq R_P$, so we can take into consideration the map $ \Phi_{P,R} :A(R)\rightarrow \mathrm{gr}U(P)$. Recall that  $ \Phi_{P,R}$ is an isomorphism if and only if $\ig{R}=\ig{R_P}$ and $P$ is of $PBW$-type. Furthermore, the ideal generated by $R$ and $R_P$ are equal if and only if $R=R_P$. We have to prove only the direct implication, the other being trivial. Let us suppose that the two ideals are equal. Thus $R_P^m\subseteq \ig{R}^m=0$, for all $m<n$. Since, $R_P\subseteq T^{\leq n}$ we get $R=R_P$.

	On the other hand, the latter relation holds if and only if $P^{\leq n-1}=0$. Indeed, if one assumes that $R=R_P$ and that there exists some nonzero $x\in P^{\leq n-1}$, then we can find $m\leq n-1$ such that $x\in P^{\leq m}$ but $x\notin P^{\leq m-1}$. Therefore, $p^m(x)$ is a nonzero element in $R_P^m$, fact that contradicts our hypothesis. The converse is obvious. We conclude that, in this example,  the map $\Phi_{P,R}$ is bijective if and only if $P^{\leq n-1}=0$ and $P$ is of $PBW$-type.
\end{example}

\begin{example}
	Consider a ring $S.$ We say that $U$ is a \emph{skew PBW extension }of $S$
	if $U$ is an $S$-ring endowed with a map $\sigma :S\rightarrow M_{n}\left(
	S\right) :s\mapsto \left( \sigma _{i,j}\left( s\right) \right) $, a matrix $%
	c=\left( c_{i,j}\right) \in M_{n}\left( S\right) $ and finite elements $%
	y_{1},\ldots ,y_{n}\in U$ fulfilling the following conditions:
	
	\begin{enumerate}
		\item[(i)] $\mathrm{Mon}\left( U\right) :=\left\{ y_{1}^{a_{1}}\cdots
		y_{n}^{a_{n}}\mid a_{1},\ldots ,a_{n}\in \mathbb{N}\right\} $ is a basis for
		$U$ as a left $S$-module;
		
		\item[(ii)] $y_{i}s-\sum_{j}\sigma _{i,j}\left( s\right) y_{j}\in S1_{U},$ for $%
		1\leq i\leq n$ and $s\in S;$
		
		\item[(iii)]  $y_{j}y_{i}-c_{i,j}y_{i}y_{j}\in S1_{U}+Sy_{1}+\ldots +Sy_{n},$ for $%
		1\leq i,j\leq n.$
	\end{enumerate}
	
	Under these conditions we will write $U=S\left\langle y_{1},\ldots
	,y_{n},\sigma ,c\right\rangle .$ If $y_{i}s-\sum_{j}\sigma _{i,j}\left(
	s\right) y_{j}=0=y_{j}y_{i}-c_{i,j}y_{i}y_{j}$ we will say that $U$ is a
	\emph{quasi-commutative skew PBW extension} of $S$ and write $U=S[
	y_{1},\ldots ,y_{n},\sigma ,c].$
	
	Note that in the case $\sigma _{i,j}\left( s\right) =0$ for $i\neq j,$ we
	recover the definition of skew PBW extension given in \cite[Definition 1]{GL}.
	
	Let $V:=Sx_{1}S\oplus \cdots \oplus Sx_{n}S$ be the free $S$-bimodule
	(equivalently left $S^{e}$-module where $S^{e}:=S\otimes _{\mathbb{Z}}S^{op}$%
	) with basis $\left\{ x_{1},\ldots ,x_{n}\right\} \ $and let $\delta
	_{i}\left( s\right) \in S,\rho _{i,j}\in S+Sx_{1}\oplus \cdots \oplus
	Sx_{n}\subseteq T_{S}\left( V\right) .$ Consider the quotients $U\left(
	P\right) $ and $A(R)$ of $T_{S}\left( V\right) $ by the two-sided ideals generated by:
	\begin{align}
	P:=&\mathrm{Span}_{S,S}\Big\{ x_{i}s-\sum_{t=1}^n\sigma _{i,t}\left( s\right)
	x_{t}-\delta _{i}\left( s\right) , x_{j}\otimes x_{i}-c_{i,j}x_{i}\otimes
	x_{j}-\rho _{i,j}\mid s\in S,i,j=1,\ldots ,n\Big\};\label{def:Pskew}\\
	R:=& \mathrm{Span}_{S,S}\Big\{ x_{i}s-\sum_{t=1}^n\sigma _{i,t}\left( s\right)
	x_{t},x_{j}\otimes x_{i}-c_{i,j}x_{i}\otimes x_{j}\mid s\in
	S,i,j=1,\ldots ,n\Big\}.\label{def:Rskew}
	\end{align}%
	Set $\mathrm{Mon}\left( U\left( P\right) \right) =\left\{ x_{1}^{\otimes
		a_{1}}\otimes \cdots \otimes x_{n}^{\otimes a_{n}}+\left\langle
	P\right\rangle \mid a_{1},\ldots ,a_{n}\in \mathbb{N}\right\} .$ It is easy
	to check that, if $\mathrm{Mon}\left( U\left( P\right) \right) $ is a basis
	for $U\left( P\right) $ as a left $S$-module, then $U\left( P\right)
	=S\left\langle y_{1},\ldots ,y_{n},\sigma ,c\right\rangle $ where $%
	y_{i}:=x_{i}+\left\langle P\right\rangle .$ Conversely every skew PBW
	extension $S\left\langle y_{1},\ldots ,y_{n},\sigma ,c\right\rangle $ is
	isomorphic to $U\left( P\right) $ where $\delta _{i}\left( s\right) $ and $%
	\rho _{i,j}$ corresponds to the elements $y_{i}s-\sum_{j=1}^n\sigma
	_{i,j}\left( s\right) y_{j}\in S1_U$ and $y_{j}y_{i}-c_{i,j}y_{i}y_{j}\in S1_{U}+Sy_{1}+\ldots +Sy_{n}$ through the
	bijection $f:S+Sx_{1}+\cdots +Sx_{n}\rightarrow S1_{U}+Sy_{1}+\cdots +Sy_{n}$
	of left $S$-modules mapping $1_{S}$ to $1_{U}$ and $x_{i}$ to $y_{i}$ for
	every $i.$ This correspondence between skew PBW extensions and $U\left(
	P\right) $ such that $\mathrm{Mon}\left( U\left( P\right) \right) $ is a
	basis for $U\left( P\right) $ as a left $S$-module induces also a
	correspondence between quasi-commutative skew PBW extensions and $A\left( R\right)$ such that $\mathrm{Mon}%
	\left( A\left( R\right) \right) $ is a basis for $A\left( R\right) $ as a
	left $S$-module.
	
	Our aim here is to show that the following are equivalent for an $S$-ring $U:$
	\begin{enumerate}
		\item $U$ is a PBW-deformation of the quasi-commutative skew PBW extension $%
		S[ y_{1},\ldots ,y_{n},\sigma ,c]$ of
		$S;$
		\item $U=S\left\langle y_{1},\ldots ,y_{n},\sigma
		,c\right\rangle $ is a skew PBW extension of $S$.
	\end{enumerate}
	Assume $\left( 1\right) $ holds true. By the foregoing $U$ is a
	PBW-deformation of $A\left( R\right) $ such that $\mathrm{Mon}\left( A\left(
	R\right) \right) $ is a left $S$-module basis. By Remark \ref{rem:large}, $%
	U=U\left( P\right) =T_{S}\left( V\right) /\left\langle P\right\rangle $
	where $P$ is an $S$-subbimodules of $T=T_{S}\left( V\right) $ such that $P$
	is of PBW-type, $R\subseteq R_{P}$ and $\left\langle R\right\rangle
	=\left\langle R_{P}\right\rangle .$ Since $x_{i}s-\sum_{j=1}^{n}\sigma
	_{i,j}\left( s\right) x_{j}\in R^{1}\subseteq R_{P}^{1},$ there is $\tau _{i}\left( s\right) \in P^{\leq 1}$ such that $%
	x_{i}s-\sum_{j=1}^{n}\sigma _{i,j}\left( s\right) x_{j}=p^1\left(
	\tau _{i}\left( s\right) \right) $. Thus $ \delta _{i}\left( s\right):=x_{i}s-\sum_{j=1}^{n}\sigma _{i,j}\left( s\right)
	x_{j}-\tau
	_{i}\left( s\right) \in T_{S}\left( V\right) ^{\leq 0}=S.$ Similarly $x_{j}\otimes
	x_{i}-c_{i,j}x_{i}\otimes x_{j}\in R^{2}\subseteq R_{P}^{2}$ so that there is some $\tau _{ij}\in P^{\leq 2}$ such that
	$x_{j}\otimes x_{i}-c_{i,j}x_{i}\otimes x_{j}=p^2\left( \tau
	_{ij}\right) .$ Therefore $\rho _{i,j}:=x_{j}\otimes x_{i}-c_{i,j}x_{i}\otimes x_{j}-\tau
	_{ij}\in T_{S}\left( V\right) ^{\leq 1}=S+V.$ By using
	the relations of the form $\tau _{i}\left( s\right) $ we can further require
	$\rho _{i,j}\in S\oplus
	Sx_{1}\oplus \cdots \oplus Sx_{n}.$ Set $y_{i}:=x_{i}+\left\langle
	P\right\rangle .$ By construction we easily get $y_{i}s-\sum_{j=1}^{n}\sigma
	_{i,j}\left( s\right) y_{j}\in S1_{U}$ and $y_{j}y_{i}-c_{i,j}y_{i}y_{j}\in
	S1_{U}+Sy_{1}+\ldots +Sy_{n}$. Moreover the canonical isomorphism $\Phi
	_{P,R}=\oplus _{t\in \mathbb{N}}\Phi _{P,R}^{t}:A\left( R\right) \rightarrow
	\mathrm{gr}U\left( P\right) $ and the fact that $\mathrm{Mon}\left( A\left(
	R\right) \right) $ is a left $S$-module basis forces $\mathrm{Mon}\left(
	U\left( P\right) \right) $ to be a left $S$-module basis. Thus $U=U\left(
	P\right) =S\left\langle y_{1},\ldots ,y_{n},\sigma ,c\right\rangle $ is a skew PBW extension of $S.$
	
	Conversely, assume $\left( 2\right) .$ By the foregoing we can take $
	U:=U\left( P\right) =T_{S}\left( V\right) /\left\langle P\right\rangle $
	where $P$ is as in \eqref{def:Pskew} and $\mathrm{Mon}\left( U\left( P\right)
	\right) $ is a left $S$-module basis. Now define $R$ as in \eqref{def:Rskew} and consider $A\left( R\right)
	=T_{S}\left( V\right) /\left\langle R\right\rangle .$ In view of the
	relations we imposed, it is clear that $\mathrm{Mon}\left( A\left( R\right)
	\right) $ generates $A\left( R\right) $ as a left $S$-module. Since $\mathrm{%
		Mon}\left( U\left( P\right) \right) $ is a left $S$-module basis and
	\begin{equation*}
		\Phi _{P,R}\left( x_{1}^{\otimes a_{1}}\otimes \cdots \otimes x_{n}^{\otimes
			a_{n}}+\left\langle R\right\rangle \right) =\left( x_{1}^{\otimes
			a_{1}}\otimes \cdots \otimes x_{n}^{\otimes a_{n}}+\left\langle
		P\right\rangle \right) +U^{\leq a_{1}+\cdots +a_{n}}
	\end{equation*}%
	we deduce that $\mathrm{Mon}\left( A\left( R\right) \right) $ is a basis for
	$A\left( R\right) $ as a left $S$-module (hence, by the foregoing, $A\left(
	R\right) $ is a quasi-commutative skew PBW extension of $S$) and that $\Phi _{P,R}$ is
	invertible (hence $U$ is a PBW-deformations of $A\left( R\right) $).
	
	The fact  that $\mathrm{Mon}\left( A\left( R\right) \right) $ is a left $S$%
	-module basis forces $\sigma :S\rightarrow M_{n}\left( S\right) $ to be a
	ring homomorphism. As a consequence $W:=Sx_{1}\oplus \cdots \oplus Sx_{n}$
	becomes an $S$-bimodule with left regular action and right action given by $%
	x_{i}\times s\mapsto \sum_{j=1}^n\sigma _{i,j}\left( s\right) x_{j}$, and $A(R)\cong T_S(W)/\ig{x_{j}\otimes x_{i}-c_{i,j}x_{i}\otimes x_{j}\mid i,j=1,\ldots,n}$.
\end{example}

\section{Bimodules of \texorpdfstring{$PBW$}{PBW}-type and the Jacobi relations }

	We keep the notation from the preceding section: $T=\oplus_{n\in\N}T^n$ is a connected graded $S$-ring  and $P$ is an $S$-subbimodule of $T$. Recall that, by definition, $R_P^n:=\LH{P}\bgi T^n=p^n(P^{\leq n})$ where $P^{\leq n}:=P\bgi T^{\leq n}$.

	Let $R$ be a graded $S$-subbimodule of $R_P$. In view of Lemma \ref{le:RS}, the canonical map $\Phi_{P,R}$ is an isomorphism  if and only if $P$ is of $PBW$-type and $R$ generates $\ig{R_P}$. We have also seen that for several important examples the latter condition is automatically satisfied. Henceforth, in order to prove that $\Phi_{P,R}$ is bijective the main step is to check that $P$ is of $PBW$-type. The purpose of this section is to find necessary and sufficient conditions for a bimodule $P$ to be of $PBW$-type.

	A first characterization of these bimodules is given in the following proposition.

\begin{proposition}\label{Th:$PBW$}
	Let $P'\subseteq P$ denote $S$-subbimodules of $T$.
\begin{enumerate}
 	\item The bimodule $P$ is of $PBW$-type if and only if $\mathrm{LH}\left( \left\langle P\right\rangle \right)\subseteq \left\langle \mathrm{LH}\left( P\right) \right\rangle $.

	\item  If $\mathrm{LH}\left( P\right) \subseteq\left\langle \mathrm{LH}\left( P^{\prime }\right) \right\rangle $, then $P^{\prime }$ is of $PBW$-type whenever $P$ is so. In this case, $\left\langle P\right\rangle=\left\langle P^{\prime }\right\rangle $.
\end{enumerate}
\end{proposition}

\begin{proof}
	We know, in view of Remark \ref{rem: $PBW$-type}, that $P$ is of $PBW$-type if and only if $R_{\ig{P}}\subseteq \ig{R_P}$. We conclude the proof of the first part of the proposition by using  Lemma \ref{Lemma:Sp}.

	Let us assume that $P$ is of $PBW$-type. Then $\ig{\mathrm{LH}(P)}=\ig{\mathrm{LH}(P')}$ by the standing hypothesis and the inclusion $P'\subseteq P$. By the first part,  $\mathrm{LH}\left(\left\langle P^{\prime }\right\rangle \right) \subseteq \mathrm{LH}\left(\left\langle P\right\rangle \right) \subseteq \left\langle \mathrm{LH}\left(P\right) \right\rangle =\left\langle \mathrm{LH}\left( P^{\prime }\right)\right\rangle $, so that $P'$ is of $PBW$-type too. In the diagram \eqref{dia:Phi} let us take $f:=\Id$. Since:
\[
	\left\langle R_{P^{\prime }}\right\rangle=\left\langle \mathrm{LH}\left( P^{\prime }\right) \right\rangle{=}\left\langle \mathrm{LH}\left( P\right) \right\rangle =\left\langle R_{P}\right\rangle,
\]
	the horizontal maps and  $\ov f$  are bijective. It follows that $\wt f =\text{gr}f'$ is an isomorphism. We conclude that $f'$ is an isomorphism, so $\left\langle P\right\rangle =\left\langle P^{\prime }\right\rangle $.
\end{proof}

	Let  $q^n_P:P^{\leq n}\to\gr^n P$ denote  the canonical projection. Recall from \cite[Definition 4.3]{NvO1} that a filtered morphism $f:M\to N$ is said to be \emph{strict} if $f(F^nM)=f(M)\cap F^nN$ for each $n\in\mathbb{N}$. It is easy to check that such a morphism is invertible if and only if its filtered components are all invertible.

\begin{lemma}\label{le:strict}
	 Let $P$ be an $S$-subbimodule of $T$. The $S$-subbimodule $P^{\leq{n-1}}$ is a direct summand of $P^{\leq n} $ for all $n$ if and only if there exists a strict isomorphism $i_P:\gr P\to P$ of filtered  $S$-bimodules.  %such that $i_P(q^n_P(a))-a\in P^{\leq n-1}$}, for all $a\in P^{\leq n}$.
\end{lemma}

\begin{proof}
 	Assume that $P^{\leq{n-1}}$ is a direct summand of $P^{\leq n}$. Then the exact sequence:
 \[
 	0\rightarrow P^{\leq n-1}\rightarrow P^{\leq n}\xrightarrow{q^n_P}\gr{^n P}\rightarrow 0
 \]
	splits.  It follows that  $q^n_P$ has a section $i^n_P$, which is a morphism of $S$-bimodules. We will regard $i_P^n$ as a morphism from $\gr ^n P$ to $P$. Let $i_P:\gr P\to P$ be the codiagonal  morphism corresponding to the family $\{i_P^n\}_{n\in\N}$.  If $\gr{^{\leq n} P}$ is the $n$-th term of the standard filtration on $\gr P$ and  $i_P^{\leq n}$ is the restriction of $i_P$ to $\gr{^{\leq n} }P$ then $i_P(\gr ^{\leq n}P)= i_P^n(\gr^{\leq n} P)\subseteq P^{\leq n}$. Thus $i_P$ is a morphism of filtered bimodules. The following diagram has commutative squares and exact sequences.
 \[\xymatrix{
	0\ar[r]&\gr{^{\leq n-1} P}\ar[r]\ar[d]^{i_P^{\leq n-1}}& \gr{^{\leq n} P} \ar[r]\ar[d]^{i_P^{\leq n}}& \gr{^n P} \ar[r]\ar@{=}[d]&0\\
	0\ar[r]&P^{\leq n-1}\ar[r]& P^{\leq n} \ar[r]^{q^n_P}& \gr{^n P} \ar[r]&0}
  \]
 	Applying the Short Five Lemma one shows inductively that $i_P^{\leq n}$ is bijective for every $n$. In conclusion, $i_P$ is a strict isomorphism of filtered bimodules.

 	Conversely, let $i_P:\gr P\to P$ be a strict isomorphism of filtered  $S$-bimodules and let $i_P^{\leq n}$  denote the $n$-th filtered component of $i_P$. The rightmost square of the above diagram is not commutative anymore. But we can make it so, by replacing $\id_{\gr ^n P}$ with an appropriate isomorphism $j^n:\gr^n P \to \gr^n P$. Thus the  sequence on the bottom of the diagram splits as the other one does.
%
%If $a\in P^{\leq n}$, then $q^n_Pi^n_Pq^n_P(a)=q^n_P(a)$. So $(i^n_Pq^n_P)(a)-a\in P^{\leq n-1}$, for all $a\in P^{\leq n}$.
 %It remains to prove that $i_P$ is bijective.
%
% Let us prove that $i_P$ is surjective too. It is enough to show by induction on $n$ that any $a\in P^{\leq n} $ is in the image of $i_P$. If $n=0$, this property is trivially satisfied, as $q^0_P$ is the identity of $P^{\leq 0}$, so $a=i^0_P(a)$ in this case. Let us assume that $P^{\leq {n-1}}$ is a subset of the image of $i_P$. We can assume that $a$ is not in $P^{\leq {n-1}}$. We know, by the construction of $i_P$, that $(i_Pq_P)(a)=a+a'$, where $a'\in P^{\leq n-1}$. By induction hypothesis, $a'$ belongs to the image of $i_P$, so $a$ is also an element of this set.
\end{proof}

\begin{remarks}
	(1) We keep the notation from the above lemma and we assume that the  isomorphism $i_P$ exists. Then  $(i^n_P q^n_P)(x)-x\in P^{\leq {n-1}}$ for all $x\in P^{\leq n}$, since $i_P^n$ is a section of $q_P^n$ by construction.

	(2) If $\gr P$ is projective as an $S$-bimodule, then every component $\gr ^n{P}$ is so. Hence the exact sequence from the proof of the preceding lemma splits. Thus, by the previous lemma, the isomorphism  $i_P$ exists in this case as well.
\end{remarks}

\begin{theorem}\label{te:alpha}
	Let $T$ be a graded connected $S$-ring  and let $P$ be an $S$-subbimodule of $T$. If $P^{\leq n}$ is a direct summand of $P^{\leq n+1}$ as an $S$-bimodule, then there is a morphism  $\alpha:R_P\to T$ of filtered $S$-bimodules such that $ P=\alpha(R_P)$.  Such an $\alpha$ exists, provided  $\gr P$ is projective as an $S$-bimodule. In particular, if $S$ is a separable algebra over a field, then for every $P$ there exists a morphism of filtered modules $\alpha $ such that  $P=\alpha(R_P)$.
\end{theorem}

\begin{proof}
By Remark \ref{rem:isos}, we have an isomorphism of graded $S$-bimodules $\overline{p}:\gr P \to R_P$ induced by the projections $p^n:T \to T^n$, for $n\in\N$.  Let $\alpha:=i_P (\overline{p})^{-1}$, where $i_P:\gr P\to P$ is the isomorphism of filtered bimodules from Lemma \ref{le:strict}. We can regard $\alpha$ as a map whose codomain is $T$. Clearly, $\alpha$ is a morphism of filtered bimodules, since $i_P$ and $(\overline{p})^{-1}$ respects the filtrations and the gradings, respectively. By construction, the image of $\alpha$ coincides with $P$.

Recall that, throughout the paper, all filtered morphisms $\alpha:R_P\to T$ must satisfy the relation $\alpha_0(r)=r$, for any $r\in R_P^n$. Given such an  $r$, set $x:=(\overline{p}^n)^{-1}(r)\in\gr{^nP}$. Then $\alpha(r)=\alpha(\overline{p}^n(x))=i_P(x)=i^n_P(x)$ so that $\alpha_0(r)= p^n(\alpha(r))=\overline{p}^nq^n_P(\alpha(r))=\overline{p}^nq^n_P(i^n_P(x))=\overline{p}^n(x)=r$. The second claim is an immediate consequence of  the previous remark.

By definition, $S$ is separable if and only if $S$ is  projective as a bimodule over itself. It is well known that $S$ is separable if and only if all $S$-bimodules are projective,  so the theorem is completely proved.
\end{proof}

\begin{corollary}\label{co:BG}
	Keep the notation and the assumptions from Example \ref{ex:BG}. The $S$-ring $U(P)$ is a $PBW$-deformation of $A(R)$ if and only if there exists  a morphism of filtered $S$-bimodules $\alpha:R\to T$ such that $P=\alpha(R)$ and $P$ is of $PBW$-type.
\end{corollary}

\begin{proof}
	We have seen in  Example \ref{ex:BG} that $R=R_P$ and $U(P)$ is a $PBW$-deformation of $A(R)$ if and only if $P^{\leq n-1}=0$ and $P$ is of $PBW$-type. Henceforth, it is enough to show that the relation $P^{\leq n-1}=0$ holds if and only if there exists  a morphism of filtered $S$-bimodules $\alpha:R\to T$ such that $P=\alpha(R)$.

	If $P^{\leq n-1}=0$ then $P^{\leq m}$ either vanishes or coincides with $P$, depending on the fact that $m$ is less than $n$ or not. Hence   $P^{\leq m-1}$ is a direct summand of $P^{\leq m}$. Thus we can apply the previous theorem to prove the existence of the morphism of filtered bimodules $\alpha$ such that $P=\alpha(R)$.

	For proving the other implication note that $P^{\leq n-1}=P\cap T^{\leq n-1}=\alpha(R)\cap T^{\leq n-1}=\alpha(R^n)\cap T^{\leq n-1}$. Thus we have to prove that $\alpha(R^n)\cap T^{\leq n-1}=0$. Let $r\in R^n$ be such that $\alpha(r)\in T^{\leq n-1}$. Then $r-\alpha(r)\in T^{\leq n-1}$, since we always have $\alpha_0(r)=r$. Thus $r\in T^{\leq n-1}$ and hence $r=0$.
\end{proof}

\begin{corollary}\label{cor:largeR}
Let  $P$ be an $S$-subbimodule of a connected graded $S$-ring $T$ such that $P^{\leq n}$ is a direct summand of $P^{\leq n+1}$, for all $n$.  Let $R$ be a graded subbimodule of  $R_P$. The $S$-ring $U(P)$ is a $PBW$-deformation of $A(R)$ if and only if $\ig{R}=\ig{R_P}$ and there exists a morphism $\alpha : R \rightarrow T$ of filtered $S$-bimodules  such that $P'= \alpha(R)$ is a subbimodule of $P$ of PBW-type which generates $\langle P\rangle$.
\end{corollary}

\begin{proof}
Let us assume that $U(P)$ is a $PBW$-deformation of $A(R)$. By Theorem \ref{te:alpha} there is a morphism $\alpha : R_P \rightarrow T$ of filtered $S$-bimodules such that $P=\alpha(R_P)$. We take by definition  $P':=\alpha(R)\subseteq P$. Since $\Phi_{P,R}$ is an isomorphism, by  Remark \ref{rem:large}, it follows that $P$ is of $PBW$-type and $\ig{R}=\ig{R_P}$.  For proving that $P'$ is a subbimodule of $PBW$-type which generates $\ig{P}$ we use  Proposition \ref{Th:$PBW$} (2). It is enough to check that $\langle \mathrm{LH}\left( P\right) \rangle \subseteq \left\langle \mathrm{LH}\left( P^{\prime }\right) \right\rangle $. Note that, in view of Lemma \ref{Lemma:Sp}, we have $\langle \mathrm{LH}\left( P\right) \rangle=\langle R_{P}\rangle $. Thus,
	$\langle \mathrm{LH}\left(P\right) \rangle=\langle R_P\rangle=\ig{R}=\langle R_{P'}\rangle =\langle \mathrm{LH}\left( P^{\prime}\right) \rangle$,  where for the third equation we used the fact that  $R_{P'}$ and $R$ coincide, cf. Example \ref{ex:alpha}.

	Conversely, let us assume that $\ig{R}=\ig{R_P}$ and there exists a morphism $\alpha : R \rightarrow T$ of filtered $S$-bimodules  such that $P'= \alpha(R)$ is a subbimodule of $P$ of PBW-type which generates $\langle P\rangle$. Thus $A(R)=A(R_P)$ and $U(P)=U(P')$ and $\Phi_{P,R}=\Phi_{P}=\Phi_{P'}$. Since $P'$ is of $PBW$-type it follows that $U(P)$ is a $PBW$-deformation of $A(R)$.
\end{proof}

\begin{remark}\label{re:alpha}
 	Let  $R$ be a graded $S$-subbimodule of $T$, where $T$ is a graded connected $S$-ring. Examples of  subbimodules $P\subseteq T$ such that $P^{\leq n}$ is a direct summand of $P^{\leq n+1}$, for all $n$,  and $U(P)$ is a $PBW$-deformation of $A(R)$ can be obtained us follows. We start with a graded subbimodule $R'$ of $\ig{R}$ such that $ R'$ is projective and $R\subseteq R'$, together with a morphism $\alpha:R'\to T$ of filtered $S$-bimodules. If $P:=\alpha(R')$ is of $PBW$-type, then $U(P)$ is a $PBW$-deformation of $A(R)$, which satisfies the required properties, since $R_P=R'$. Basically, Theorem   \ref{te:alpha} says that any  $PBW$-deformation $U(P)$ of $A(R)$ is of this type, provided that $P^{\leq n}$ is a direct summand of $P^{\leq n+1}$, for all $n$. Indeed, by definition of $PBW$-deformations, we have $R\subseteq R_P$, so we can take $R'=R_P$. Since $\zeta_{P,R}$ is an isomorphism, $\ig{R}=\ig{R'}$.  The existence of $\alpha$ follows by Theorem \ref{te:alpha}.
\end{remark}

\begin{definition}
	Let $P$ be a submodule of $T$ with filtration $\left\{F^nP\right\}_{ n\in\mathbb{N}}$ such that $F^nP\subseteq T^{\leq n}$. We say that the filtration  is \emph{jacobian} if the condition $F^{n+1} P\bgi T^{\leq n}\subseteq F^nP$ holds for every $n\geq 0$.
\end{definition}

\begin{lemma}\label{lemma:jacobian}
	Let $P'\subseteq P$ be submodules of $T$ with filtrations $\left\{F^nP\right\}_{ n\in\mathbb{N}}$ and $\left\{F^nP'\right\}_{ n\in\mathbb{N}}$, respectively. Assume that $F^nP'\subseteq F^nP\subseteq T^{\leq n}$ and that the  filtration of $P$ is jacobian. Then $F^nP'=F^nP$, for all $n$, if and only if $p^n(F^nP')=p^n(F^nP)$, for all $n$.
\end{lemma}

\begin{proof}
	We assume that $p^n\left(F^nP'\right)=p^n\left(F^nP\right)$ for all $n\geq 0$. To show that the terms of the two filtrations coincide, we  proceed by induction on $n$. If $n=0$, then $F^0P'=p^0(F^0P')=p^0(F^0P)=F^0P$.

	Let us assume that we have $F^{n-1}P'=F^{n-1}P$, for some $n>0$. We consider the following commutative diagram with exact rows:
\[\xymatrix{
	0\ar[r]&F^nP'\bgi T^{\leq n-1}\ar@{^(->}[d]\ar[r]& F^nP' \ar@{^(->}[d]\ar[r]^{p^n}& p^n(F^nP') \ar@{^(->}[d]\ar[r]&0\\
	0\ar[r]&F^nP\bgi T^{\leq n-1}\ar[r]& F^nP \ar[r]^{p^n}& p^n(F^nP) \ar[r]&0}
\]
	By assumption, the cokernel of the rightmost vertical map is trivial, for every $n\in \mathbb{N}$. Thus, by the Snake Lemma,  we deduce that the cokernels of the first two vertical arrows are isomorphic.  By the inductive hypothesis and the fact that the filtration on $P$ is jacobian, we get the following chain of inclusions:
\begin{equation*}\ts
	F^nP'\bgi T^{\leq n-1}\subseteq F^nP\bgi T^{\leq n-1}\subseteq F^{n-1}P=F^{n-1}P'\subseteq F^nP'\bgi T^{\leq n-1}.
\end{equation*}
	Therefore $F^nP'\bgi T^{\leq n-1}=F^nP\bgi T^{\leq n-1}$, so the cokernel of the leftmost vertical arrow vanishes. Therefore, the cokernel of the middle vertical map must be trivial as well, that is $F^nP'=F^nP$.
\end{proof}

\begin{fact}[The Jacobi conditions.]\label{fa:Jacobi}
	Our next goal is to show that  $P$ is of $PBW$-type if and only if some relations that generalize the Jacobi identity are satisfied.

	For all $n\geq 0$, we set
\begin{equation}\label{eq:defPn}\ts
	P_n:= \sum\li_{i+j+k=n}T^{\leq i}\cdot P^{\leq j}\cdot T^{\leq k}
\end{equation}
	Obviously,  $P_n\subseteq P_{n+1}$ and  $T^nP_m$ and $P_mT^n$ are subbimodules of $P_{n+m}$.  Moreover, $\ig{P}=\bgr_{n\in\N}P_n$.  Thus,  $\{P_n\}_{n\in\N}$ is an exhaustive increasing filtration, compatible with the multiplication of $T$.

 	On the other hand, if $I=\oplus_{n\in\N}I^n$ denotes the graded ideal generated by $R_P$, for each $n$ we have:
\begin{equation*}\ts
	p^n\left(\sum\li_{i+j+k=n}T^{\leq i}\cdot P^{\leq j}\cdot T^{\leq k}\right)=\sum\li_{i+j+k=n}T^{i}\cdot p^j\left(P^{\leq j}\right)\cdot T^{k}=\sum\li_{i+j+k=n}T^{ i}\cdot R_P^{j}\cdot T^{ k}{=} \ig{R_P}^n=I^n.
\end{equation*}
	where the last equality follows from the fact that $\langle R_P\rangle$ is homogeneous. Therefore, $I^n=p^n(P_n)$. Note that $P$ is of $PBW$-type if and only if  $I^n=R_{\ig{P}}^n$ for all $n\geq 0$.

	Our aim now is to show that the $S$-bimodule $P$ is of $PBW$-type if and only if it satisfies a sort of Jacobi conditions. More precisely, we shall say that $P$ satisfies the \emph{Jacobi condition} $(\cJ_n)$ if and only if
\begin{equation*}\label{eq:J_k} \tag{$\cJ_n$}
	 P_{n+1}\bgi T^{\leq n}\subseteq P_{n}.
\end{equation*}
\end{fact}

\begin{remark}\label{rem:Qn}
	Recall that a connected graded $S$-ring $T$ is called \emph{strongly graded} provided that $T^nT^m=T^{n+m}$, for all $n,m$. Obviously, $T$ is strongly graded if and only if $T^1$ generates $T$ as an $S$-ring. Throughout the remaining part of this section we assume that $T$ is strongly graded.

	Under this condition, $P_n$ can be computed using the relation $P_0=P\bgi T^0$ and the recursive equation:
\begin{equation}\label{eq:P_n}
  	P_{n+1}=T^1P_n+P_nT^1+P^{\leq n+1},
\end{equation}
	Indeed, let $\{P'_n\}_{n\in\mathbb{N}}$ denote the sequence which one obtains by using the relation \eqref{eq:P_n}, starting with $P'_0=P_0$. Thus we have to prove that $P'_n=P_n$, for all $n\in \mathbb{N}$.

	We proceed by induction to show the inclusion $P'_n\subseteq P_n$. By definition of $P'_0$, the relation holds for $n=0$. Let us assume that the inclusion holds for some $n$. We have:
\[
	P'_{n+1}=T^1P_n'+P_n'T^1+P^{\leq n+1}\subseteq \sum\li_{i+j+k=n}T^{\leq i+1}P^{\leq j}T^{\leq k}+ \sum\li_{i+j+k=n}T^{\leq i}P^{\leq j}T^{\leq k+1}+P^{\leq n+1}=P_{n+1}.
\]
	To prove the other inclusion, we first remark that $T^iP_j'$ and $P'_jT^i$ are subbimodules of $P'_{i+j}$. Moreover, one can easily see by induction that the sequence $\left\{ P'_n\right\}_{n\in\mathbb{N}}$ is increasing. Hence, for any $p$, $q$ and $r$ we have $T^{ p}P^{\leq q} T^{r}\subseteq T^{p} P_q'  T^{r}\subseteq P_{p+q+r}'$. Thus, clearly, $P_n=\sum_{i+j+k=n}T^{\leq i}P^{\leq j} T^{\leq k}\subseteq P_n'$.

	Let $I:=\ig{R_P}$. In a similar way one proves that $\{I^n\}_{n\in\mathbb{N}}$ is uniquely determined by the relations $I^0=R_P^0=P\bigcap T^0$ and the recursive formula:
\begin{equation}\label{eq:I_recursive}
 	\ts I^{n+1}=T^1I^n+I^nT^1+R^{n+1}_P.
\end{equation}
\end{remark}

\begin{theorem}\label{Th:PBWJac}%\label{lemma:Jacobi}
	Let $T$ be a connected graded $S$-ring and $P\subseteq T$ an $S$-subbimodule. The following assertions are equivalent:
 \begin{enumerate}
  	\item The bimodule  $P$ is of $PBW$-type.

  	\item The filtration $\{P_n\}_{n\in\N}$ is Jacobian, that is $P$ satisfies the condition $(\cJ_n)$, for all $n\in\N$.

  	\item $P_m\bgi T^{\leq n}\subseteq P_n$, for all  $n\in\N$ and $m>n$.
  	
  	\item The relation $\ig{P}\bgi T^{\leq n}=P_n$ holds, for all $n\in\N$.
 \end{enumerate}
\end{theorem}

\begin{proof}
	Let $I:=\ig{R_P}$. Since the ideal generated by $P$ is the union of all subspaces $P_{m}$, we have
\[\ts
	\ig{P}\bgi T^{\leq n}=\bgr\li_{m\in\N}(P_m\bgi\li T^{\leq n})=P_n\bgr\li\big( \bgr\li_{m>n}(P_m\bgi T^{\leq n})\big),
\]
	where the last equality holds as $P_m \subseteq T^{\leq n}$ for $m\leq n$ and the sequence $\{P_m\}_{m\in\N}$ is increasing. Thus (4) is equivalent to the fact that $ \bgr_{m> n}(P_m\bgi T^{\leq n}) $ is a subset of $ P_n$, for all $n\in\N$. This property, in turn, is equivalent to (3).

	Obviously (3) implies (2). Let us prove the implication (2)$\Longrightarrow$(3). For every $n\in\N$ and $m> n$ we have to show that  $P_m\bgi T^{\leq n}\subseteq P_n$. We proceed by induction. The base case $m=n+1$ is true by $(\cJ_n)$. If $m\geq n+2$ we have
\[\ts
 	P_m\bgi T^{\leq n}\subseteq P_m\bgi T^{\leq m-1}\bgi T^{\leq n}\subseteq P_{m-1}\bgi T^{\leq n}\subseteq P_n,
\]
	where for the last two inclusions we used the condition $(\cJ_{m-1})$ and the induction hypothesis.

	In order to prove  that (1) and (4) are equivalent we apply Lemma \ref{lemma:jacobian} to the filtrations $\left\{P_n\right\}_{n\geq 0}$ and $\big\{\ig{P}^{\leq n}\big\}_{n\geq 0}$ of $\ig{P}$. Note that the later filtration is jacobian, and (4) is obviously equivalent to the fact that the two filtrations coincides. Thus, in view of the above mentioned lemma, it  is enough to show that (4) holds if and only if $p^n(P_n)=p^n(\ig{P}^{\leq n})$ for all $n$. Thus (4) is true if and only if $I^n=R^n_{\ig{P}}$ for all $n$,  that is if and only if $P$ is of PBW-type.
\end{proof}

	Another equivalent form of Jacobi conditions is related to a special class of bimodules that generates  a graded ideal $I$ in $T$. This new concept will be also useful in the fourth section to describe the second term of a projective resolution of $S$ as an $A$-module, where $A:=T/I$.
	
	Before to define bimodule of  relations we prove the following result.
\begin{lemma}\label{lem:foot}
A graded subbimodule $R$ of $T$ generates the graded ideal $I$ if and only if $R+\wt I=I$, where $\wt{I}:=IT^1+T^1I$.
\end{lemma}
	
\begin{proof}
  Obviously $\wt I$  is a graded ideal of $T$ and  $\wt{I}\subseteq I$. 	Let us assume that $R$ generates $I$. By definition of the ideal generated by a subbimodule, we have the relation $I^{n}=T^1I^{n-1}+I^{n-1}T^1+R^{n}$.  Thus $I=\wt{I}+R$. Conversely, let $J:=\ig{R}$. For $n=0$, then $J^0=R^0=I^0$, since $\wt{I}\subseteq T^{\geq 1}$. Let us suppose that $I^{n-1}=J^{n-1}$. Then, by definition of $J$ and induction hypothesis, we get:
\[
	 I^n=(\wt{I})^n+R^n=T^1I^{n-1}+I^{n-1}T^1+R^n=T^1J^{n-1}+J^{n-1}T^1+R^n=(\wt{J})^n+R^n= J^n.
\]
	We conclude that $I=J$.
\end{proof}

\begin{definition}\label{def:minimality}
	Let $I$ be a graded ideal in a strongly graded connected $S$-ring $T$. A graded  $S$-subbimodule $R$  is a \emph{bimodule of relations} for $T/I$ if and only if  $R$ is a complement subbimodule of $\wt I$ in $I$, that is $\cR\oplus \wt I=I$.
\end{definition}

Let us remark that, by Lemma \ref{lem:foot},  $R$ is a bimodule of relations for $T/I$ if and only if $R$ generates $I$ and $R\bgi \wt I=0$. Note that $R$ is isomorphic to $I/\wt I$ as $S$-bimodules.

	A few properties of  bimodules of relations are proved in the following proposition.

\begin{proposition}\label{pr:minimal}
	Let $\wt{R}$ be a graded $S$-subbimodule which generates the graded ideal $I$ in $T$.
\begin{enumerate}
	\item If $I/\wt I$ is a projective $S$-bimodule then there exists a bimodule of relations for $T/I$, which is also projective as a bimodule.
	
	\item Let us assume that there exists a graded $S$-bimodule complement $R$ of $\wt{R}\bgi \wt{I}$ in $\widetilde{\cR}$. Then  $R$ is a bimodule of relations for $T/I$.

	\item Let  $R$ be a minimal  $S$-bimodule which generates $I$. If $R\bigcap \wt{I}$ is a direct summand in $R$, then $R$ is a bimodule of relations for $T/I$.
\end{enumerate}

\end{proposition}

\begin{proof}
	Since $I/\wt I$ is projective, the canonical projection $I\to I/\wt I$ has an $S$-bilinear section. Let $R$ denote its image. Hence $R\oplus \wt I=I$, which means that $R$ is a bimodule of relations for $T/I$. By construction $R$ is projective as an $S$-bimodule. Hence we have (1).
	
	Let us prove (2). Since $\wt R$ generates $I$, we have $I=\wt R + \wt I$. On the other hand, $\wt R=R\oplus (\wt R\bigcap \wt I)$. Thus
\[
	I=\wt R+\wt I=R+ (\wt R\bigcap \wt I)+\wt I= R+\wt I.
\]	
	Clearly, $R\bigcap \wt I = R\bigcap (\wt R\bigcap \wt I)=0$. Thus $R$ is a bimodule of relations for $T/I$.

	It remains to prove (3). Let us assume that $R$ is a minimal bimodule which generates $I$. Here by minimal we mean that for any submodule $R'\subseteq R$ which generates $I$ we have $R'=R$. Let us assume that $R$ is minimal and that $R'$ is a subbimodule complement of  $R\bigcap\wt I$ in $R$. Thus,
\[
	I=R+\wt I=R'+(R\bigcap\wt I)+\wt I=R'+\wt I.
\]
	We deduce that $R'$ generates $I$, so $R'=R$. Obviously, $R\bigcap \wt I=R'\bigcap (R\bigcap \wt I)=0$, as $R'$  is a subbimodule complement of $R\bigcap \wt I$. In conclusion $R$ is a bimodule of relations for $T/I$.
\end{proof}

\begin{corollary}
	Let $S$ be a separable $\Bbbk$-algebra. There exists a bimodule of relations $R$ for $T/I$, where $I$ is an arbitrary graded ideal in $T$.
\end{corollary}

\begin{proof}
 	Every $S$-bimodule is projective, since $S$ is separable. Now apply Proposition \ref{pr:minimal} (1).
\end{proof}

	Recall that a graded bimodule $R$ is called $n$-\textit{pure} if $R=R^n$.  It is \textit{pure} if it is $n$-\textit{pure} for some $n$.

\begin{corollary}\label{co:Pure}
Let $R$ be a pure bimodule which generates the ideal $I$. Then $R$ is a bimodule of relations for $T/I$.
\end{corollary}

\begin{proof}
	Let $n$ be a nonnegative integer such that $R=R^n$. Since $\wt{I}\subseteq T^{\geq n+1}$, we have $R\bigcap\wt I=0$.
\end{proof}

\begin{proposition}\label{pr:minimal'}
	Let $P\subseteq T$ be an $S$-subbimodule such that $R:=R_P$ is a bimodule of relations of  $T/\ig{R}$.  Let $I=\ig{R}$.
\begin{enumerate}
	\item The relation $\left(T^1P_n+P_nT^1\right)\cap P\subseteq P^{\leq n}$ holds for all $n\in \N$.
	
	\item The Jacobi relation $(\cJ_n)$ holds if and only if $\left(T^1P_n+P_nT^1\right)\cap T^{\leq n}\subseteq P_n$.
\end{enumerate}
\end{proposition}

\begin{proof}
Let $x\in \left(T^1P_n+P_nT^1\right)\cap P$. Since $P_n$ is a subset of $T^{\leq n}$ it follows that $x\in P^{\leq n+1}$. Thus, by Lemma \ref{Lemma:Sp} and in view of \S\ref{fa:Jacobi}, we get:
\begin{equation*}
 	p^{n+1}\left(\left(T^1P_n+P_nT^1\right)\cap P^{\leq n+1}\right)\subseteq \wt{ I}^{n+1}\cap R^{n+1}.
\end{equation*}
	Since $R$ is a complement subbimodule of $\wt I$ in $I$ we have  $p^{n+1}\left(x\right)=0$, that is $x$ is an element of $P^{\leq n}$.

	To prove the second part of the statement we want to apply Lemma \ref{lem:grsum} for $X:=T^1P_n+P_nT^1$ and $Y:=P^{\leq n+1}$. 	Note that $R_X^{n+1}=T^1p^n(P_n)+p^n(P_n)T^1=T^1I^n+I^nT^1$ and that $R_Y^{n+1}=R^{n+1}$ are the only non zero components of $R_X$ and $R_Y$ having degree greater than $n$. Since $R_X^{n+1}\cap R_Y^{n+1}=0$, by the lemma we deduce that $(X+Y)^{\leq n}=X^{\leq n}+Y^{ \leq n}$. Since $Y^{ \leq n}=P^{\leq n}\subseteq P_n$ we deduce that $(\cJ_n)$ holds, i.e. $(X+Y)^{\leq n}\subseteq P_n$, if and only if $X^{\leq n}\subseteq P_n$.
%Clearly, if  $(\cJ_n)$ holds then  $\left(T^1P_n+P_nT^1\right)\cap T^{\leq n}$ is a subset of $P_n$. To prove the other implication we pick $x\in \left(T^1P_n+P_nT^1+P^{\leq n+1}\right)\bigcap T^{\leq n}$. Thus $x=y-z$, for some $y\in T^1P_n+P_nT^1$ and $z\in P^{\leq n+1}$. Since  $x\in T^{\leq n}$ we get $p^{n+1}(z)=p^{n+1}(y)$. On the other hand, proceeding as in the proof of the first part of the  proposition, we can show that $p^{n+1}(z)$ and $p^{n+1}(y)$  belong to $R\bgi \wt I$, so they equal $0$. Hence $y\in \left(T^1P_n+P_nT^1\right)\cap T^{\leq n}\subseteq P_n$. Moreover, $z\in P\bigcap T^{\leq n}=P^{\leq n} \subseteq P_n$. Thus, $x\in P_n$.
\end{proof}

	In order to prove another characterization of $S$-bimodules of $PBW$-type, similar to \cite[Theorem 4.2]{CS}, we  introduce the analogous notion of the central extension associated to $P$, see  \cite[p. 2]{CS}.

\begin{fact}[The graded $S$-ring $D(P)$.] \label{fa:P_z}
	Consider the polynomial ring $\tz$ in the variable $z$ with coefficients in $T$, regarded as a graded ring with homogeneous component of degree $n$ given by $\tz^n=\sum_{i=0}^n T^iz^{n-i}$.

	By \cite[\S II.11]{NvO1}, the \emph{external homogenization} of $a$ is the polynomial $a^*=\sum_{i=0}^d a_iz^{d-i}\in \tz^d$, for every $a=\sum_{i=0}^d a_i\in T$, with $a_i\in T^i$ and $a_d\neq 0$.
If $X$ is a subset of $T$, then we shall use the notation $X^*:=\{x^*\mid x\in X\}$.

	The quotient of $\tz$ by the ideal generated by $P^*$ will be denoted by $D(P)$. We shall refer to $D(P)$ as the \emph{central extension associated to} $P$. The evaluation map at $z=0$ defines a graded ring morphisms $\ev{0}:\tz\to T$. Similarly, evaluation at $z=1$ induces a  morphism of filtered rings $\ev{1}:\tz\to T$, where  on $\tz$ we take the standard filtration associated to the grading of this ring.  The kernels of $\ev{0}$ and $\ev{1}$ are the ideals generated by $z$ and $z-1$, respectively.

	Let us remark that $\ev{0}$ and $\ev{1}$ map $x^*$ to $\mathrm{LH}(x)$ and $x$ respectively, for all $x\in P$. Thus the evaluation maps factorize through ring homomorphisms  $\phi_0:D(P)\to A(R_P)$ and $\phi_1:D(P)\to U(P)$.

	If there is no danger of confusion $P$, we shall use the notation $A=A(R_P)$, $U=U(P)$ and $D=D(P)$.
\end{fact}

\begin{lemma}\label{le:ev_1}
	Keeping the foregoing notation, we have:
\begin{enumerate}
 	\item The relations $\ev{1}(\ig{P^*}^n)=P_n$ and $\ker(\phi_1)=(z-1)D$ hold, where $\ig{P^*}^n:= \ig{P^*}\bgi \tz^n$ denotes the homogeneous component of degree $n$ of $\ig{P^*}$. Thus $\phi_1$ induces an isomorphism of $S$-rings $D/(z-1)D\cong U$.

	\item The relations $\ev{0}(\ig{P^*}^n)=\ig{R_P}^n$ and $\ker(\phi_0)=zD$ hold. Thus $\phi_0$ induces an isomorphism of $S$-rings $D/zD\cong A$.
\end{enumerate}
\end{lemma}

\begin{proof}
	We shall prove only (1). The second part of the lemma can be proved in a similar way. Let us denote $T[z]$ by $\tz$. For every $j\in\mathbb{N}$, let $(P^*)^j:=P^*\bgi \tz^j$. We claim that $\ev{1}\left((P^*)^j\right)\subseteq P^{\leq j}$. Indeed, an element $y$ belongs to $(P^*)^j$ if and only if $y=x^*$, for some $x\in P$, and $y\in T_z^j$. Let $x_0,\dots, x_r\in T$ be the homogeneous components elements of $x$ with $x_r\neq 0$. Then $y=x^*=\sum_{i=0}^rx_iz^{r-i}$ belongs to $\tz^r$. On the other hand, by assumption, $y\in \tz^j$. Thus $r=j$ and $\ev{1}(y)=x\in P\bgi T^{\leq j}=P^{\leq j}$, so our claim has been proved. Thus we get:
\begin{equation*}
	\ev{1}\big(\ig{P^*}^n\big)  = \sum_{i+j+k=n}\ev{1}\left(\tz^i (P^*)^j \tz^k\right)= \sum_{i+j+k=n}T^{\leq i} \ev{1}\left((P^*)^j\right) T^{\leq k}\subseteq\sum_{i+j+k=n}T^{\leq i} P^{\leq j} T^{\leq k}\stackrel{\eqref{eq:defPn}}{=} P_n.
\end{equation*}
	Conversely, for every $x\in P^{\leq j}$, we have $x=\sum_{i=0}^{j'}x_i$, where $j'\leq j$, each $x_i$ belongs to $T^i$ and $x_{j'}\neq 0$. Then, by definition, $x^*=\sum_{i=0}^{j'}x_iz^{j'-i}$. Since $\ev{1}(z^{j-j'}x^*)=x$, we get $P^{\leq j}\subseteq \ev{1}(\ig{P^*}^j)$. Thus
 \[
  	P_n=\sum_{i+j+k=n}T^{\leq i} P^{\leq j} T^{\leq k}\subseteq\sum_{i+j+k=n}T^{\leq i} \ev{1}(\ig{P^*}^j) T^{\leq k}= \ev{1}\left(\sum_{i+j+k=n}T_z^{ i}\ig{P^*}^j T_z^{k}\right)=\ev{1}\left(\ig{P^*}^n\right) .
 \]
 	In conclusion,  $\ev{1}\left(\ig{P^*}^n\right)$ and  $P_n$ are equal. We deduce that the ideals $\ig{P}$ and $ \ev{1}\left(\ig{P^*}\right)$ are equal.

	Let $\pi_{P^*}:\tz\to D$ denote the canonical projection. Now we can prove that $D/(z-1)D$ and $ U$ are isomorphic by applying the Snake Lemma to the next diagram, which has commutative squares and exact rows and columns.
\[\xymatrix @R=15pt{
	&&0\ar[d]&0\ar[d]\\
	&&\ig{P^*}\ar@{-->}[r]\ar[d]&\ig{P}\ar[d]\\
	0\ar[r]&(z-1)\tz\ar@{-->}[d]\ar@{^(->}[r]&\tz\ar[d]^{\pi_{P^*}}\ar[r]^{\ev{1}}&T\ar[d]^{\pi_P}\ar[r]&0\\
	0\ar[r]&\ker{\phi_1}\ar@{^(->}[r]&D\ar[r]^{\phi_{1}}&U \ar[r]&0}
\]
	The maps $\pi_{P^*}$ and $\ev{1}$ induce the vertical and the horizontal dashed arrows, respectively. We have proved  that the the image of latter one is the ideal generated by $P$, so it is surjective. On the other hand,  $\pi_{P^*}$ is surjective by definition.  It follows that the vertical dashed arrow is surjective as well. Therefore, $\ker(\phi_1)= (z-1)D $.
\end{proof}

	The class of the variable $z$ in $ D $ will be denoted by $z$ as well. The annihilator of the central element $z$ in $ D $ is a two-sided ideal. Since $z$ is homogeneous of degree $1$, this ideal is graded. We denote by $\an{}^n$ the homogeneous component of degree $n$ of $\an{}$.

\begin{lemma}\label{lemma:exactsequence}
	For every $n\in\N$ there exists an exact sequence:
\begin{equation}\xymatrix{
	0 \ar[r] & P_n \ar@{^(->}[r] & \ig{P}^{\leq n} \ar[r]^-{\tau_n} & (z-1)D\bgi D^n \ar[r] & 0.
}
\end{equation}
	In particular, the relation $P_n=\ig{P}^{\leq n}$ holds for some $n$ if and only if $(z-1)D\bgi D^n=0$.
\end{lemma}

\begin{proof}
	Note that  $(z-1)\tz\bgi \tz^n=0$. Hence, by applying the Snake Lemma to the  diagram:
\[\xymatrix @R=15pt{
	&0\ar[r]\ar@{=}[d]&\ig{P^*}^n\ar[r]^-{\ev{1}}\ar[d]&\ig{P}^{\leq n}\ar[d]\\
	0\ar[r]&(z-1)\tz\bgi \tz^n\ar[d]_{\pi'}\ar@{^(->}[r]&\tz^n\ar[d]^{\pi_{P^*}^n}\ar[r]^{\ev{1}}&T^{\leq n}\ar[d]^{\pi_P^{\leq n}}\ar[r]&0\\
	0\ar[r]&(z-1)D\bgi D^n\ar@{^(->}[r]&D^n\ar[r]^{\phi_{1}^n}&U^{\leq n} \ar[r]&0}
\]
	we obtain the exact sequence:
\begin{equation*}\xymatrix{
	0 \ar[r] & \ig{P^*}^n \ar[r]^-{\ev{1}} & \ig{P}^{\leq n} \ar[r]^-{\tau_n} & (z-1)D\bgi D^n \ar[r] & 0,
}
\end{equation*}
where $\tau_n$ denotes the corresponding connecting morphism.  We conclude remarking that, by the proof of the Lemma  \ref{le:ev_1}, we have  the relation $\ev{1}\left(\ig{P^*}^n\right)=P_n$.
\end{proof}

\begin{remark}\label{re:tau_n}
	By the proof of the Snake Lemma \cite[Chapter 4.1]{Ve}, the connecting morphism $\tau_n$ is defined as follows. For an element $x$ in $\ig{P}^{\leq n}$ we choose $y\in \tz^n$ such that $\ev{1}(y)=x$. Then $\pi^n_{P^*}(y)$ is an element in $(z-1)D\bgi D^n$, whose class in the cokernel of the map $\pi'$ coincides with $\tau_n(x)$. Since $\pi'=0$ we get that $\tau_n(x)=\pi_{P^*}^n(y)$. Let $x=\sum_{i=0}^k x_i$, where $k\leq n$ and $x_i\in T^i$, with $x_k\neq 0$. Thus we may take $y=z^{n-k}x^*$, so $\tau_n(x)=z^{n-k}x^*+\ig{P^*}$.
\end{remark}

\begin{theorem}\label{Th:JacAnn}
	The Jacobi relation $(\cJ_n)$ holds for some $n\in\N$ if and only if $\an{}^{n}=0$.
\end{theorem}

\begin{proof}
	Let us consider the following diagram:
\[\xymatrix @R=15pt{
	0 \ar[r] & P_n \ar@{^(->}[r]\ar@{^(->}[d] & \ig{P}^{\leq n} \ar[r]^-{\tau_n}\ar@{^(->}[d] & (z-1)D\bgi D^n \ar[d]^-{z\cdot} \ar[r] & 0\\
	0 \ar[r] & P_{n+1} \ar@{^(->}[r] & \ig{P}^{\leq n+1} \ar[r]^-{\tau_{n+1}} & (z-1)D\bgi D^{n+1} \ar[r] & 0
}
\]
	Its rows are exact, cf.  Lemma \ref{lemma:exactsequence}. If $x\in \ig{P}^{\leq n}$ and we write $x=\sum_{i=1}^kx_i$ as in the previous remark, then $\tau_n(x)=z^{n-k}x^*+\ig{P^*}$. On the other hand, if we regard $x$ as an element in $\ig{P}^{\leq n+1}$, then $\tau_{n+1}(x)=z^{n+1-k}x^*+\ig{P^*}$. It follows that the squares of the diagram are commutative.

	The kernel of the rightmost arrow is $(z-1)D\bgi D^n\bgi\an{}$. Since every $x\in\an{}^n$ may be written $x=(1-z)x\in (z-1)D\bgi D^n$, it follows that the kernel of this map is $\an{}^n$.

	The Snake Lemma applied to this diagram yields us the exact sequence:
\begin{equation*}\xymatrix{
0 \ar[r] & \an{}^n \ar[r] & {P_{n+1}}/{P_n} \ar[r] & {\ig{P}^{\leq n+1}}/{\ig{P}^{\leq n}},
}
\end{equation*}
	where the last arrow is induced by the canonical inclusion of  $P_{n+1}$ into $\ig{P}^{\leq n+1}$. Therefore, we can identify $\an{}^n$ with  $\left({\ig{P}^{\leq n}\bgi P_{n+1}}\right)/P_n$. Obviously this quotient bimodule is isomorphic to $\left({P_{n+1}\bgi T^{\leq n}}\right)/{P_n}$, so $\an{}^n=0$ if and only if $(\cJ_n)$ holds.
\end{proof}

{
\begin{theorem}\label{th:Rees}
	Let $T$ be a connected graded $S$-ring and $P\subseteq T$ an $S$-subbimodule. The following assertions are equivalent:
 \begin{enumerate}
  	\item $P$ is of PBW-type.

	\item $\an{}=0$.

  	\item $(z-1)D(P)\bgi D(P)^n=0$, for all  $n\in\N$.

  	\item $\ig{P^*}= \ig{\ig{P}^*}$, or equivalently $D(P)=D(\langle P\rangle)$.
 \end{enumerate}
\end{theorem}

\begin{proof}
	By Theorem \ref{Th:PBWJac}, Lemma \ref{lemma:exactsequence}  and Theorem \ref{Th:JacAnn} it follows that (1), (2) and (3) are equivalent.  We now want to prove that (3) implies (4). Let us consider the morphism $\varfun{\phi_1}{D(P)}{U(P)}$ which maps $w+\langle P^*\rangle\in D(P)$ to $\ev{1}(w)+\langle P\rangle\in U(P)$ and its analogue $\varfun{\phi_1'}{D(\langle P\rangle)}{U(\langle P\rangle)=U(P)}$. The inclusion $P^*\subseteq \ig{P}^*$ induces a surjective graded map $\varfun{\Lambda_P}{D(P)}{D(\langle P\rangle)}$ sending $w+\langle P^*\rangle\in D(P)$ to $w+\langle \ig{P}^*\rangle\in D(\langle P\rangle)$. Note that $\phi_1'   \Lambda_P=\phi_1$. Thus $\ker\left(\Lambda_P\right)^n\subseteq (z-1)D(P)\bgi D(P)^n=0$. Therefore $\Lambda_P$ is injective, so $\ig{P^*}= \ig{\ig{P}^*}$.

	To complete the proof we show that (4) implies (1). Consider the following commutative diagram with exact rows:
\[\xymatrix{
	0\ar[r]&zD(P)\ar[d]\ar[r]& D(P)\ar[d]^-{ \Lambda_P}\ar[r]^{\phi_0}& A(R_P) \ar[d]^-{\Psi_P}\ar[r]&0\\
	0\ar[r]&zD(\langle P\rangle)\ar[r]& D(\ig{P})\ar[r]^{\phi'_0}& A(R_{\ig{P}}) \ar[r]&0}
\]
	where $\phi_0\left(w+\langle P^*\rangle\right)=\ev{0}(w)+\langle P\rangle\in U(P)$, and $\phi'_0$ is defined in a similar way. The map $\Psi_P$ is induced by the inclusion $R_P\subseteq R_{\ig{P}}$. Clearly, all vertical arrows are surjective. By Snake Lemma and taking into account the standing hypothesis, it follows that $\Psi_P$ is injective. Thus $\ig{R_P}=\ig{R_{\ig{P}}}$, so $P$ is of PBW-type.
\end{proof}
}

\begin{remark}\label{rem:alphaz}
	The ring $D(\ig{P})$ plays the same role of the Rees ring of $U$, see \cite[Proposition 4.4 (i)]{Li}, where the Rees ring of $U$ is denoted by $\widetilde{U}$ and $T=\Bbbk\langle X\rangle$. In other words, Theorem \ref{th:Rees} tells us that $P$ is of PBW-type if and only if $D(P)$ is the Rees ring of $U$.\vspace*{1ex}
\end{remark}

\begin{fact}[The central extension associated to $\alpha$.]\label{fa:extension}
	Let us consider the case when $P$ is given by a filtered map  $\alpha:R\to T$, as in Example \ref{ex:alpha}.  We denote by $\alpha_i:R\to T$ the component of degree $-i$ of $\alpha$. Recall that we have assumed that $\alpha_0$ always is the inclusion of $R$ into $T$, and for any element $r\in R$, we may write $\alpha(r):=\sum_{i=0}^{\infty}\alpha_i(r)$. If $P=\alpha(R)$, then we can describe the ideal generated by $P^*$ as follows.

	Let $ \alpha_z:R\to \tz$ denote the morphism of graded $S$-bimodules given by:
\begin{equation}\label{eq:alpha_z}
	\alpha_z(r)=\sum_{i=0}^{\infty}\alpha_i(r)z^{i}.
\end{equation}
	Using this morphism we now define the graded $S$-subbimodule $P_z\subseteq \tz$ by the equation $P_z=\alpha_z(R)$.

	We will refer to $T_z/\ig{P_z}$ as the \emph{central extension associated to $\alpha$}.
\end{fact}

\begin{proposition}\label{le:h(P)}
 	If $T$ is a graded connected $S$-ring and $P$ is the $S$-bimodule associated to a filtered morphism $\alpha:R\to T$, then $\ig{P^*}=\ig{P_z}$.
\end{proposition}

\begin{proof}
	The ideal $\ig{P_z}$ is generated by $P_z=\sum_{n\in\N}\alpha_z(R^n)$. If $r\in R^n$ is  nonzero, then $\alpha_z(r)=\alpha(r)^*$. Thus, $\ig{P_z}$ is included into  $\ig{P^*}$.

	Conversely, let $y=x^*$ denote an element in $P^*$. Since $x\in P$, there are $r_0, \dots r_d$, with $r_i\in R^i$ and $r_d\neq 0$, such that $x=\sum_{i=0}^d \sum_{j=0}^{i}\alpha_j({r_i})$. Since $\alpha_j({r_i})\in T^{i-j}$ and $\alpha_0(r_d)\in T^d$ is a nonzero element,
\[
	y=\sum_{i=0}^d\sum_{j=0}^{i}\alpha_j({r_{i}})z^{d+j-i}=\sum_{i=0}^d \alpha_z({r_i})z^{d-i}.
\]
	The computation  above shows us that  $y\in\ig{P_z}$. Thus the proposition is proved, as any generator of  $\ig{P^*}$ belongs to the ideal generated by $P_z$.
\end{proof}

	In the next sections we shall see that there exists a numerical  invariant, called  the homological complexity of $A$ and denoted by $c(A)$, such that $\an{}=0$, provided  that $\an{}^{\leq c(A)}=0$ and the homology in degree one of a suitable complex vanishes.  Thus, in the case when $c(A)$ is finite, for proving that $P$ is of $PBW$-type it is necessary and sufficient to show that a certain homology group vanishes and a finite number of the  Jacobi conditions $(\cJ_k)$ hold, namely those for $k\leq c(A)$. As a matter of fact, we shall be able to characterize in a homological way the vanishing of $\an{}$ for a more general class of connected graded $S$-rings $D$ and $A$.

\section{A homological criterion for regularity of a central element}\label{sec:3}

	Throughout this section we fix a connected strongly graded $S$-ring $\D=\oplus_{n\geq 0}\D^n$. We assume that  there exists a central element $z\in \D^1$ and we denote the quotient ring ${\D}/{zD}$ by $\A$. Note that $A$ is strongly graded as well. Since $zD\subseteq \D^{+}$, we have $\A^0=\D^0=S$. Let $\varfun{\pi}{\D}{\A}$ and $\an{}$  denote the canonical projection and the kernel of the multiplication map $z\cdot(-):\D\to \D$, respectively.

\begin{definition}\label{def:regular}
	The element $z$ will be called $n$-\textit{regular} if and only if the homogeneous components of $\an{}$ are trivial in degree less than or equal to $n$, i.e. if and only if $\an{}\bigcap D^ {\leq n}=0$, for all $k\leq n$. We shall say that $z$ is \textit{regular} if and only if it is $n$-regular for all $n$.
\end{definition}

	Roughly speaking, assuming that $S$ has a special type of projective resolution as a left $A$-module, in the main result of this section we will show that  the regularity of $z$ is equivalent to the fact that a certain sequence $(M_*,\delta_*)$, canonically associated to the above resolution,  defines a complex of $\D$-modules. At once that it exists, this complex yields a projective resolution of $S$ as a $D$-module.

	We will use several times the following isomorphism, known as the tensor-hom adjunction formula:
	\begin{equation}\label{eq:adjunction}
	\Hom_{S',S''}(X\ot Y,Z)\cong \Hom_{S,S''}(Y, \Hom_{S'}(X,Z)),
	\end{equation}
	where $X$ is an $(S',S)$-bimodule, $Y$ is an $(S,S'')$-bimodule and $Z$ is an $(S',S'')$-bimodule.
	
	As a direct consequence of the tensor-hom adjunction formula, we deduce that $M\ot N$ is a projective left $M$-module, provided that $M$ and $N$ are $(S,S)$-bimodules, which are  projective left $S$-modules.

\begin{fact}[A projective resolution of $S$.]\label{killing01}
	Let us construct a suitable projective resolution of the left $A$-module $S$. By definition we take  $d_0:A\to S$ to be the projection onto $0$-component. Let  $d_1: A \ot  A ^1\to  A $  denote  the $A$-linear map  which is induced by the multiplication of $ A $. Since $d_1$ is a graded map and $d_1^1$ is an isomorphism, we have $K_1=\ker(d_1)=\bigoplus_{i\geq 2}\ker(d_1^i)\subseteq A^+\otimes A^1$. Thus for any strongly graded $S$-ring $A$ we have an exact sequence of length $1$ as above, namely:
\begin{equation}\label{eq:resolution_1}
	A\ot A^1\xrightarrow{d_1} A \xrightarrow{d_0} S\rightarrow 0.
\end{equation}
	If $A^1$ is projective as a left $S$-module, then $A\ot A^1$ is a projective $A$-module. Hence we can complete the above sequence to a projective resolution:
\begin{equation}\label{eq:special_resolution}
	\cdots\rightarrow  A \ot V_n\xrightarrow{d_n} A \ot V_{n-1}\xrightarrow{d_{n-2}}\cdots  \xrightarrow{d_3} A \ot V_2\xrightarrow{d_2} A \ot A^1\xrightarrow{d_1} A  \xrightarrow{d_0} S\rightarrow 0,
\end{equation}
	where  the $S$-modules $V_k$ are projective $S$-modules. Such a resolution exists since, for a left $A$-module $M,$ there is a projective left $S$-module $V$ together with a graded $A$-linear epimorphism  $A\ot V\to M$.

	As  a matter of fact, we will show that  we can choose $V_n$ in  \eqref{eq:special_resolution} such that $ V_n^i=0$, for all $n\geq 2$ and $i=0,1$. Indeed, if $K_1=\ker(d_1)$ then $K_1^0\subseteq (A\ot A^1)^0=0$. On the other hand, $d_1^1:A^0\ot A^1\to A^1$ is an isomorphism as $d_1$ is induced by the multiplication of $A$. Thus $K_1^1=0$.
	
	For every $j\geq 2$ we pick  a surjective map  $\nu^j:V_2^j\to K_1^j$, where $V_2^j$ is a a projective $S$-module. Let $V_2=\oplus_{j\geq 2} V_2^j$. The induced  map $\nu:V_2\to K_1$ is surjective. Let  $d_2'$ denote the composition  of the arrows
\[
	A\ot V_2 \xrightarrow{A\ot \nu}	A\ot K_1 \rightarrow K_1.
\]
Hence $d_2'(a\otimes v)=a\nu(v)$. Clearly, $d_2'$ is surjective and and $V_2^i=0$, for $i=0,1$. Thus we can take $A\ot V_2$ as the term of degree $2$ of the resolution \eqref{eq:special_resolution}. The differential $d_2$ is the composition of the inclusion map $i_1:K_1\to A\ot A^1$ with $d_2'$. Let us assume  that we constructed the sequence \eqref{eq:special_resolution} up to degree $n$, such that $V_k^i=0$, for $2\leq k\leq n$ and $i=0,1$. Thus,  $(A\ot V_n)^i=0$, so $K_n=\ker(d_n)$ is trivial in degree $0$ and $1$. Now we can construct $V_{n+1}$ and $d_{n+1}$ as in the case $n=1$.

Throughout the remaining part of this section we assume that $A^1$ is projective as a left $S$-module, We also fix a resolution $(A\ot V_*,d_*)$ as in  \eqref{eq:special_resolution}, and which satisfies the condition that $V_n^i=0$, for all $n\geq 2$ and $i=0,1$.
\end{fact}

\begin{fact}[The normalized bar complex.] \label{fact:normbarcompl}
	Let $ A $ be a  graded connected $S$-ring and let $\cX$ be a graded left $ A $-module. Since we work with graded objects, we can regard $A$ as an augmented ring, whose augmentation ideal is $A^+$. By \cite[Exercise 8.6.4]{We}, the sequence $B_*(A,X)$ is exact. Recall that \textit{the bar resolution} is defined by $B_n(A,X)=A\ot (A^+)^{\ot n}\ot X$. The differential maps are defined by:
\[
	b_n(a_0\ot a_1\ot\cdots \ot a_{n}\ot x)=\sum_{i=0}^{n-1}(-1)^{i}a_0\ot  a_1\ot\cdots \ot a_ia_{i+1}\ot\cdots\ot a_{n}\ot x+(-1)^na_0\ot a_1\ot a_2\ot\cdots\ot a_{n} x,
\]
	for all $a_0\in A$, $a_1,\dots,a_n\in A^+$ and $x\in X$. In spite of its name, $B_*(A,M)$ is not a resolution of $X$, as $B_n(A,X)$ is not a projective $A$-module.

	Nevertheless, if $A$ and $X$ are flat as left $S$-modules, then we can use the bar resolution for computing $\tor_n^A(Y,X)$, for any right $S$-module $Y$, because it is easy to see that $B_n(A,X)$ are flat $A$-modules, so $B_*(A,X)$ is a flat resolution of $X$.

	Let $\Omega_\ast( A ,\cX)= S\ot_A B_*(A,X)$ be the  \textit{normalized bar complex}. In view of the foregoing remarks, we can use it to compute $\Tor_\ast^ A (S, X)$ provided that $A$ and $X$ are flat left $S$-modules. It is not difficult to prove that $\Omega_n( A ,\cX)= ( A ^+)^{\ot n}\ot X$ and $d_1$ is the restriction of the module structure map of $\cX$ to $A ^+\ot \cX $. For all $n>1$, the differential morphism $d_n:\Omega_n( A ,\cX)\to\Omega_{n-1}( A ,\cX)$ satisfies the relation:
\[
	d_n(a_1\ot\cdots \ot a_{n}\ot x)=\sum_{i=1}^{n-1} (-1)^{i-1} a_1\ot\cdots \ot a_ia_{i+1}\ot\cdots\ot a_{n}\ot x+(-1)^{n-1}a_1\ot a_2\ot\cdots\ot a_{n-1}\ot a_{n}x.
\]
	It is well-known that $\Tor_n^A(Y,X)$  is a graded abelian group, whose component of degree $m$ will be denoted by $\Tor_{n,m}^A(Y,X)$. Thus $\Tor_n^ A (S,X)\cong\oplus_{m=0}^\infty\Tor_{n,m}^ A (S,X)$. We will refer to $n$ and $m$  as being the\textit{ homological} degree and the \textit{internal} degree, respectively.

	The complex $\Omega_\ast( A , \cX)$ decomposes as  a direct sum  of subcomplexes $\Omega_\ast( A ,\cX)=\oplus_{m \geq 0}\,\Omega_\ast( A ,\cX,m)$, where  $\Omega_n(\A,\cX,m)$ is the $S$-submodule of $\Omega_n( A ,\cX)$ generated by all $a_1\ot\cdots \ot a_n\ot x$, with $x\in X^{m_{n+1}}$ and $a_i\in A^{m_i}$, satisfying the relation $\sum_{i=1}^{n+1} m_i=m$. In the case when $A$ and $X$ are flat left $S$-modules, the $n$-th homology group of  $\Omega_\ast( A ,\cX, m)$ coincides with $\Tor_{n,m}^ A (S,X)$. In particular, for  $n>m$, we have  $\Tor_{n,m}^ A (S,X)=0$, as there are no nontrivial $n$-chains in $\Omega_\ast( A ,\cX,m)$.

	We shall use the notation $\Omega_*(A)$ for $\Omega_*(A,S)$. Note that $\Omega_n(A)\cong (A^+)^{\ot n} $. Via this identification,  the formula of the differential $d_n$ becomes:
\[
	d_n(a_1\ot\cdots \ot a_{n})=\sum_{i=1}^{n-1} (-1)^{i-1} a_1\ot\cdots \ot a_ia_{i+1}\ot\cdots\ot a_{n}.
\]
\end{fact}

\begin{fact}[The morphisms $_\nu u$ and $^\nu v$.]\label{ssec:u-v}
	Recall that our next goal is to show that the vanishing of $\an{}$ is equivalent to the fact that a certain sequence $(M_*,\delta_*)$ is a complex in the category of left $D$-modules. In order to define it, we need the constructions below.
	
	Let $ A $ denote a connected graded $S$-ring. We assume that $X$ and $Y$ are graded left  $S$-modules. If $u: A \ot X\to A \ot Y$ is a graded $ A $-linear  map, then  we define  $\wh{u}:X\to  A \ot Y$ by $\wh{u}(x)=u(1\ot x)$. By construction $\wh{u}$ is graded $S$-linear and $u=\left(m_{ A }\otimes Y\right) \left( A  \otimes \wh{u}\right)$.
	Clearly, if $u': A \ot X\to  A \ot Y$ is another graded $ A $-linear map such that $\wh{u}=\wh{u'}$, then $u=u'$. Moreover, if $u_1: A \ot \cX\to  A \ot \cY$ and  $u_2: A \ot \cY\to A \ot \cZ$ are graded $ A $-linear maps, then the following relation holds true:
	\begin{equation}\label{eq:f_hat}
	\wh{u_2u_1}=(m_ A \ot Z)( A \ot\wh{u}_2)\wh{u}_1.
	\end{equation}
	
	Let $\nu: A \to \Gamma$ denote a surjective morphism of connected graded $S$-rings. Obviously, any graded $ A $-linear map $u: A \ot \cX\to  A \ot \cY$ induces a unique $\Gamma$-linear morphism $\fd{u}{\nu}:\Gamma\ot \cX\to\Gamma\ot \cY$ such that the first diagram from Figure~\ref{fig:morphisms} is commutative. Note that, if $u'=(\nu\ot Y)\wh{u}$, then $_\nu u(t\ot x)=t\cdot u'(x)$.
	
	Furthermore, let  $v:\Gamma\ot \cX\to\Gamma\ot \cY$ be a graded $\Gamma$-linear map, where $X$ is now a projective $S$-module. Since $ A \ot X$ is a projective $ A $-module, we can choose a morphism $\fu{v}{\nu}: A \ot \cX\to  A \ot \cY$ which makes commutative the second square in Figure~\ref{fig:morphisms}.
	
	If $\nu$ has a section $\mu$, which is a morphism of $S$-graded rings, then we can choose $\fu{v}{\nu}$ in a canonical way. Note that in this case we do not need $X$ to be  projective. Indeed,   we may define $v':\cX\to  A \ot \cY$ by $v'=(\mu\ot Y)\wh{v}$, where $\wh{v}$ is constructed as above, replacing the ring $ A $ and the map $u$ with $\Gamma$ and $v$, respectively. By definition, $v'$ is a morphism of graded $S$-modules.
	We can now extend  in a unique way $v'$ to a graded $ A $-linear map $\fu{v}{\nu}: A \ot \cX\to  A \ot \cY$. Clearly, $\fu{v}{\nu}$ makes  the second square in Figure~\ref{fig:morphisms} commutative. Note that $\fu{v}{\nu}(a\ot x)=a\cdot v'( x)$, so $\fu{v}{\nu}$ only depends on $\nu$ and $v$.
\begin{figure}[ht] \[\xymatrix{
	A \ot \cX\ar[r]^u\ar[d]_{\nu\ot \cX} & A \ot \cY\ar[d]^{\nu\ot \cY}\\
	\Gamma\ot \cX\ar[r]_{\fd{u}{\nu}}& \Gamma\ot \cY} \qquad \xymatrix{ A \ot \cX\ar[r]^{\fu{v}{\nu}}\ar[d]_{\nu\ot \cX} & A \ot \cY\ar[d]^{\nu\ot \cY}\\
	\Gamma\ot \cX\ar[r]_{v}& \Gamma \ot \cY}
\]
	\caption{The morphisms $_\nu u$ and $^\nu v$.}\label{fig:morphisms}
\end{figure}
\end{fact}

\begin{fact}[The maps $\delta_n:M_n\to M_{n-1}$.]\label{ssec:3.5}
	We are now turning back to the investigation of regularity of $z\in \D$. Recall that in this section we assumed that there exists a projective resolution as in \eqref{eq:special_resolution}.  Thus $V_0=S$, $V_1=\A^1$ and $d_1$ is induced by the multiplication in $A$. Moreover, every $V_n$ is a projective graded left $S$-module and $d_i$ is a morphism of graded $A$-modules.

	For $n=0$ we let $\partial_0:D\ot V_0\cong D\to S$ to be the canonical projection. For $n>0$, since $V_n$ is projective, there is a morphism of graded $D$-modules $\partial_n:D\ot V_n\to D\ot V_{n-1}$ as in \S\ref{ssec:u-v}, such that:
\[
 	(\pi\ot V_{n-1}) \partial_n=d_n (\pi\ot V_n).
\]
	Let us note that there exists a left $S$-linear section $\sigma^1:A^1\to D^1$ of $\pi^1$, as $A^1$ is a projective left $S$-module. It is easy to see that  $\partial_1$ can be chosen such that $\partial_1(x\ot a)=x\sigma^1(a)$, for all $x\in D$ and $a\in A^1$, as the latter map satisfies the above relation.

 	For easiness of notation, let $N_n:=\D\ot V_n$ and $M_n:=N_n\oplus N_{n-1}(-1)$. Thus $\partial_n:N_n\to N_{n-1}$. By construction, we have  $ (\pi\ot V_{n-2}) \partial_{n-1} \partial_n=0$. Therefore the image of  $\partial_{n-1}  \partial_n$ is included into $zN_{n-2}$, which coincides with the kernel of $\pi\otimes V_{n-2}$, as $V_{n-2}$ is flat. Thus $\partial_{n-1}  \partial_n$ can be regarded as a morphism from $N_n$ to $zN_{n-2}$. On the other hand, the multiplication by the central element $(-1)^{n-1}z$ defines a surjective morphism of graded $D$-modules from $N_{n-2}\left(-1\right)$ to $zN_{n-2}$. Hence, for every $n\geq 2$, we have the following diagram of graded morphisms:
\begin{equation*}\xymatrix @C=25pt @R=20pt{
 	& & N_n \ar[d]^-{\partial_{n-1}  \partial_n} & \\
	N_{n-2}\left(-1\right) \ar[rr]^-{(-1)^{n-1} z\cdot} & & zN_{n-2} \ar[r] & 0
}
\end{equation*}
 	Note that $\D\ot V_n$ is projective as an object in the category of graded $D$-bimodules, since $V_n$ is projective. Thus, there is a graded $D$-linear map $\varfun{f_n}{N_n}{N_{n-2}\left(-1\right)}$ such that:
\begin{equation}\label{eq:partial_n}
	\partial_{n-1}  \partial_n=\left(-1\right)^{n-1}z\cdot f_n.
\end{equation}

	We are now ready to define our candidates for the differential maps of a complex of left $D$-modules $(M_\ast,\delta_\ast)$. In the case when $S$ is a field these maps  were constructed in  \cite[\S3]{CS}.

	We first take $\delta_0:\D\to S$ to be the canonical projection. Since $V_0=S$ and $V_1=\A^1$, for all $x\in N_1$ and $y\in N_0(-1)$ we can define $\delta_1$  by the formula:
\begin{equation}\label{eq:delta1}
	\delta_1(x,y)=\partial_1(x)+zy.
\end{equation}
	On the other hand, if $n>1$ and $(x,y)\in M_n$, then we set:
\[
 	\delta_n(x,y)=\big(\partial_n(x)+(-1)^{n-1}z\cdot y, \partial_{n-1}(y)+f_n(x)\big).
\]
 	Note that the map $\delta_n$ is a morphism of graded $D$-modules, for all $n\in \N$. In general the sequence $\big\{ \delta_n\big\}_{n\in\N}$ does not define a complex. Nevertheless,
\begin{equation}\label{eq:M_2}
 	M_2\xrightarrow{\delta_2} M_1\xrightarrow{\delta_1} M_0\xrightarrow{\delta_0}S\xrightarrow{}0
\end{equation}
	is always a complex. Indeed the image of $\delta_1$ is included into $\oplus_{n>0}M_0^n$, as $N_1$ and $N_0(-1)$ do not contain elements of degree zero and $\delta_1$ is a graded map. Thus $\delta_0 \delta_1=0$, since $\delta_0$ is the projection onto $\D^0$. Furthermore, for $(x,y)\in M_2$ we have:
\[
	(\delta_1 \delta_2)(x,y)=\delta_1\big(\partial_2(x)-zy, \partial_1(y)+f_2(x)\big)=(\partial_1 \partial_2)(x)-z\partial_1(y)+z\partial_1(y)+zf_2(x)=0,
\]
	where in the above equations we used the fact that $\delta_1$ is $\D$-linear and the relation $\partial_1 \partial_2+zf_2=0$, which holds by the construction of $f_2$.
\end{fact}

	A necessary and sufficient condition for $(M_\ast, \delta_\ast)$ being a complex is given in the following result.

\begin{lemma} \label{lemma:deltacomplex}
	We keep the notations and assumptions from \S\ref{ssec:3.5}. For $n\geq 2$ we have $\delta_n  \delta_{n+1}=0$ if and only if
\begin{equation}\label{eq:f_n}
	\partial_{n-1}  f_{n+1}+f_n \partial_{n+1}=0.
\end{equation}
\end{lemma}

\begin{proof}
	For $(x,y)\in M_{n+1}$, let $(x',y')=(\delta_n  \delta_{n+1})(x,y)$. Using \eqref{eq:partial_n}, by direct computation we get:
\begin{align*}
 	x'&=\partial_n(\partial_{n+1}(x))+(-1)^nz\partial_n(y)+(-1)^{n-1}z\partial_n(y)+(-1)^{n-1}zf_{n+1}(x)=0,\\
 	y'&=\partial_{n-1}(\partial_n(y))+\partial_{n-1}(f_{n+1}(x))+f_n(\partial_{n+1}(x))+(-1)^nzf_n(y)=\partial_{n-1}(f_{n+1}(x))+f_n(\partial_{n+1}(x)).
\end{align*}
	For $n\geq 2$, we conclude that  $\delta_n  \delta_{n+1}=0$ if and only if the relation \eqref{eq:f_n} holds.
\end{proof}

	For showing that $\an{}=0$ if and only if $(M_\ast,\delta_\ast)$ is a complex, we also need the following.

\begin{lemma}\label{lemma:annz}
	For any flat $S$-module $V$ we have  $\mathrm{ann}_{\D\otimes V}(z)=\an{}\ot V$. In particular, if $\an{}=0$, then $\mathrm{ann}_{\D\otimes V}(z)=0$.
\end{lemma}

\begin{proof}
	We have seen that $\an{}$ is the kernel of the $D$-linear map from $D$ to $D(1)$, which is defined by the multiplication by $z$. Since $V$ is a flat $S$-module, we get an exact sequence:
\begin{equation*}
	0\xrightarrow{} \an{}\otimes V \xrightarrow{} \D\otimes V \xrightarrow{z\cdot}  (\D\otimes V)(1).
\end{equation*}
	Therefore, $\mathrm{ann}_{\D\otimes V}(z)=\an{}\otimes V$, so the lemma is proved.
\end{proof}

\begin{lemma}\label{lemma:M_complex}
	For any $n\geq 2$ the relation:
\begin{equation}\label{eq:times_z}
	z\cdot \left(\partial_{n-1}  f_{n+1}+f_n \partial_{n+1}\right)=0
\end{equation}
	holds. In particular, if $\an{}=0$ then $\left(M_\ast,\delta_\ast\right)$ is a complex.
\end{lemma}

\begin{proof}
	In view of relation \eqref{eq:partial_n} we have:
\begin{align*}
	z\cdot  \left(\partial_{n-1}  f_{n+1}+f_n \partial_{n+1}\right)& =\partial_{n-1}  \left(z\cdot f_{n+1}\right)+\left(z\cdot f_n\right) \partial_{n+1} \\
 	& =(-1)^{n}\partial_{n-1} \partial_n  \partial_{n+1}+(-1)^{n-1}\partial_{n-1} \partial_n  \partial_{n+1},
\end{align*}
	which proves the first claim. Moreover, the left-hand side of relation \eqref{eq:f_n} is a graded map from $N_{n+1}$ to $N_{n-2}(-1)$. The $\D$-modules $N_{n-2}(-1)$ and $N_{n-2}$ are isomorphic (but not as graded modules). Thus, to conclude the proof it is enough to remark that  $\mathrm{ann}_{N_{n-2}}(z)=\left\{m\in N_{n-2}\mid z\cdot m=0\right\}$ is trivial for all $n\geq 2$, in view of  Lemma \ref{lemma:annz}.
\end{proof}

\begin{fact}[The truncated sequence $^pM_\ast$.]\label{fact:truncseq}
	Even though $(M_\ast,\delta_\ast)$ is not in general a complex, for $p\geq 0$ we shall denote the truncated sequence of maps:
\begin{equation*}
	M_{p+1}\xrightarrow{\delta_{p+1}}  M_p  \xrightarrow{\delta_{p}}   M_{p-1}  \xrightarrow{\delta_{p-1}} \cdots  \xrightarrow{\delta_{1}} M_0 \xrightarrow{} 0
\end{equation*}
	by ${}^pM_{\ast}$. We have already seen in \S\ref{ssec:3.5} that $^1M_\ast$ is always a complex. Let us prove that $\h_0({^1M_\ast})=S$. Since  $^1M_\ast$ is a complex of graded $D$-modules, it can be written it as the direct sum of the complexes:
\begin{equation}\label{eq:M_1^k}
	M_1^k\xrightarrow{\delta_{1}^k}   M_0^k\xrightarrow{\delta_{0}^k}  S^k \xrightarrow{} 0.
\end{equation}
	We want to show that all of them are exact. Let us first consider the case $k>0$. Then $S^k=0$, $M_0^k=\D^k$ and $ M_1^k=N_1^k\oplus N_0^{k-1}=\left(\D^{k-1}\otimes \A^1\right)\oplus \D^{k-1}$. Hence the sequence \eqref{eq:M_1^k} coincides with:
\begin{equation*}
	\left(\D^{k-1}\otimes \A^1\right)\oplus \D^{k-1} \xrightarrow{\delta_{1}^k}   \D^k\xrightarrow{} 0 \xrightarrow{} 0.
\end{equation*}
	Thus in this case we have to check that $\delta_{1}^k$ is surjective. Recall that $\sigma^1$ denotes a graded $S$-linear section of $\pi^1$. By the definition of  $\delta_1$ and the construction of $\partial_1$ in \S\ref{ssec:3.5}, for $x,y\in D^{k-1}$ and $a\in \A^1$, we have:
\begin{equation*}
	\delta_{1}^k\left(x\otimes a,y\right)\stackrel{\eqref{eq:delta1}}{=}\partial_1\left(x\otimes a\right)+zy=x\sigma^1(a)+zy.
\end{equation*}
	Moreover, since $\sigma^1$ is section of $\pi^1$, we have  $\D^1=\sigma^1(\A^1)\oplus\ker(\pi^1) =\sigma^1(\A^1)\oplus zS$. Thus, for  $k>0$ we get $\D^k=\D^{k-1}\D^1= \D^{k-1}\sigma^1(\A^1)+ zD^{k-1}$. Thus  every element $x\in\D^k$ can be written as a sum $x=\sum_{i=0}^{r}x_i\sigma\left(a_i\right)+zy$, with $x_i\in \D^{k-1}$, $y\in \D^{k-1}$ and $a_i\in \A^1$. This means that $\delta_1^k$ is surjective.

	If $k=0$, then $M_1^0=0$, as $\A^1$ contains only elements of degree one. The sequence \eqref{eq:M_1^k} becomes:
\begin{equation*}
	0\xrightarrow{} M_0^0 \xrightarrow{\delta_0^0}   S\xrightarrow{} 0,
\end{equation*}
	which is clearly exact, as $\delta_0^0=\mathrm{id}_S$, being induced by the projection $\pi$.

	We set $\delta'_1(x,y)=\delta_1(x,y)$ and $\delta''_1(x,y)=0$. For  $n\geq 2$ and $(x,y)\in M_n$ we shall use the notation  $\delta_n(x,y)=(\delta'_n(x,y),\delta''_n(x,y))$, where $\delta'_n(x,y)=\partial_n(x)+(-1)^{n-1}z\cdot y$ and $\delta''_n(x,y)=\partial_{n-1}(y)+f_n(x)$.
\end{fact}

\begin{lemma}\label{lemma:proj_resolution}
	Let $n\geq 1$ and let  $(x,y)\in M_n$.
\begin{enumerate}
 	\item If $\delta'_n(x,y)=0$, then there is $y'\in N_{n-1}(-1)$ such that $(x,y)-(0,y')$ is a boundary in $(M_\ast,\delta_\ast)$ and $z\cdot y'=0$. Moreover, if $\delta_n \delta_{n+1}=0$ then $y'$ can be chosen such that $\partial_{n-1}(y')=\delta''_n(x,y)$.
  	
  	\item If $\an{}=0$ then the sequence $\left({M_*},\delta_*\right)$ is an acyclic complex.
\end{enumerate}
\end{lemma}

\begin{proof}
	Let us prove the first part. By definition of $\delta'_n$ we have:
\begin{equation} \label{eq:rel01}
	\partial_n(x)=(-1)^{n}z\cdot y.
\end{equation}
	Let $\pi_i:=\pi\otimes V_i$. Recall that  $\pi_{i-1} \partial_i=d_i \pi_i$. By  \eqref{eq:rel01} we get $d_n\left(\pi_n\left(x\right)\right)=\pi_{n-1}\left(\partial_n\left(x\right)\right)=0$. Thus  $\pi_n(x)$ is an $n$-cycle in $(\A\ot V_\ast,d_\ast)$, which is an acyclic complex. It follows that there exists some $u\in \D\otimes V_{n+1}$ such that $\pi_n(x)=d_{n+1}\left(\pi_{n+1}(u)\right)=\pi_n\left(\partial_{n+1}(u)\right)$. Since $x-\partial_{n+1}(u)\in\ker(\pi_n)=zD\otimes V_n$,  there exists $v\in \D\otimes V_n$ such that:
\begin{equation}\label{eq:rel03}
	x=\partial_{n+1}(u)+(-1)^nz\cdot v.
\end{equation}
	By construction, the element $\left(u,v\right)$ belongs to $M_{n+1}$ and by the definition of $\delta_{n+1}$ we get:
\begin{equation*}
	(x,y)-\delta_{n+1}\left(u,v\right)=(0,y-f_{n+1}(u)-\partial_n(v)).
\end{equation*}
	Hence we may take $y':=y-f_{n+1}(u)-\partial_n(v)$. The relation $z\cdot y'=0$ follows by the computation:
\begin{equation*}
	z\cdot y' =z\cdot y-z\cdot f_{n+1}(u)-z\cdot \partial_n(v) = (-1)^n\partial_n\left(x-\partial_{n+1}(u)-(-1)^nz\cdot v\right)\stackrel{\eqref{eq:rel03}}{=}0,
\end{equation*}
	where for the second equality we used relations \eqref{eq:partial_n} and \eqref{eq:rel01}.

	Let us assume that $\delta_n \delta_{n+1}=0$. For $n=1$, the relation $\partial_0(y')=0=\delta_1''(x,y)$ obviously holds, as $\partial_0:D\to S$ is the projection and $\delta''_1=0$. For $n>1$ we have:
\begin{equation*}
	\left(0,\partial_{n-1}(y')\right)=\delta_n(0,y')=\delta_n(x,y)=\left(0,\delta''_n(x,y)\right).
\end{equation*}
	Let us prove the second part of the lemma.  By Lemma \ref{lemma:M_complex}, since $\an{}=0$, it follows that $(M_\ast,\delta_\ast)$ is a complex. By the first part of the lemma, every cycle of degree $n$ is homologous to $(0,y')$, for some $y'$ such that $z\cdot y'=0$ and $\partial_{n-1}(y')=0$. Hence $y'\in \mathrm{ann}_{\D\ot V_{n-1}(-1)}=\an{}\ot V_{n-1}(-1)=0$. Thus every $n$-cycle is a boundary, that is $\h_n(M_\ast)=0$.
\end{proof}

\begin{lemma}\label{lemma:vanishingH1}
	If the sequence ${^2M_*}$ is a complex and $H_1\left({^1M_*}\right)=0$ then $\an{}=0$.
\end{lemma}

\begin{proof}
	We shall prove by induction on $k$ that $\an{}^i=0$, for all $i<k$. If $k=0$, then this property is obviously true. Let us assume that $ \an{}^i=0$, for all $i<k$. We pick an element $y$ in $\an{}^k$. Then, by the definition of $\delta_1$, the couple  $(0,y)$ is a cycle of degree $1$. By hypothesis, there is $(u,v)\in M_2$ such that  $(0,y)=\delta_2(u,v)$. By Lemma \ref{lemma:proj_resolution}, since $\delta_2'(u,v)=0$, there is $y'\in\mathrm{ann}_{\D\otimes A^{1}}(z)$ such that $y=\partial_1(y')$. Because $\partial_1$ is graded, we can assume $y'\in\mathrm{ann}_{\D\otimes A^{1}}(z)^k=\an{}^{k-1}\otimes A^{1}$. Since  $\an{}^{k-1}=0$ we have $y'=0$. Thus $y=0$, so $\an{}^k=0$. In conclusion, we have $\an{}^i=0$, for all $i<k+1$.
\end{proof}

	We can now prove one of the main results of this section by applying the preceding two lemmas.

\begin{theorem}\label{th:mainth}
	Let $A:=D/zD$, where $\D$ is a strongly graded $S$-ring and $z\in \D^1$ is a central element. Suppose that $A^1$ is left $S$-projective.  The following assertions are equivalent.
\begin{enumerate}
	\item The element $z$ is regular.
	
	\item The sequence $\left(M_\ast,\delta_\ast\right)$ is a resolution of $S$ by projective left $D$-modules.
	
	\item The sequence $\left({}^2M_\ast,\delta_\ast\right)$  is a complex  and $H_1\left({}^1M_\ast\right)=0$.
\end{enumerate}
\end{theorem}

\begin{proof}
	The implication (1)$\Longrightarrow$(2) follows by Lemma \ref{lemma:proj_resolution}. The implication (2)$\Longrightarrow$(3) is trivial. Finally, using Lemma \ref{lemma:vanishingH1} we get the implication (3)$\Longrightarrow$(1).
\end{proof}

\begin{definition}\label{fact:complexity}
The \textit{complexity} of a graded left $S$-module  $V$ is the integer $c_V$ defined as follows. If $V=0$, then $c_V=-1$. Otherwise, one takes  $c_V=\sup\{n-1\mid V^n\neq 0\}$. We point out that the complexity will be mainly used for modules $V$ with $V^0=V^1=0$, such as $V=V_3$, see  \S\ref{killing01}.
\end{definition}

	Let us now turn back to the investigation of regularity of $z$. We assume that $c:=c_{V_3}$ is finite. Since $\D$ is strongly graded, for $n\geq c+1$, we have
\begin{equation}\label{eq:D}
	\left(\D\otimes V_3\right)^n=\D^{n-c-1}\left(\D\otimes V_3\right)^{c+1}.
\end{equation}
	Let $f:=f_2  \partial_3+\partial_1  f_3$. By definition, $f$ is a $\D$-linear map of degree $-1$. Let $f^k$ denote the restriction of $f$ to $\left(\D\otimes V_3\right)^k$, so $f^k:\left(\D\otimes V_3\right)^k\longrightarrow \D^{k-1}$.

	Let us assume in addition that $\an{}^{\leq c }=0$. By \eqref{eq:times_z},  it follows that $f^k=0$ for all $ k\leq c +1$.  Taking into account the equation \eqref{eq:D}, the submodule $(D\ot V_3)^{\geq c  +1}$ is generated by $(D\ot V_3)^{ c  +1}$.  Since $f$ is a morphism of graded $D$-modules and $f^{c  +1}=0$ it follows that $f^k=0$ for all $k\geq c  +1$. Thus, by Lemma \ref{lemma:deltacomplex},  it follows that $\left({^2M_*,\delta_*}\right)$ is a complex. Supposing also that $H_1\left({^1M_*}\right)=0$, by Theorem \ref{th:mainth} we deduce that $z$ is regular. Conversely, if $z$ is regular then it follows from Theorem \ref{th:mainth} that  $H_1\left({^1M_*}\right)=0$. Obviously, $z$ is $c \,$-regular.

	Notice that if the complexity of $V_3$ is not finite, $z$ is $c\,$-regular if and only if it is regular and in this case $H_1\left({^1M_*}\right)=0$ is automatically satisfied, in light of Theorem \ref{th:mainth}. Summarizing, we have the following.

\begin{theorem}\label{th:regandcompl}
	Let $A:=D/zD$, where $\D$ is a strongly graded $S$-ring and $z\in \D^1$ is a central element. Suppose that $A^1$ is projective as a left $S$-module. The element $z$ is regular in $D$ if and only if $H_1\left({^1M_*}\right)$ vanishes and $z$ is $c_{V_3}$-regular.
\end{theorem}

\begin{remark}\label{re:c(A)=-1}
	Recall that in  \S\ref{killing01} we chosen a  the resolution $A\ot V_*$ such that $V_3^0=0$. Thus, if $c_{V_3}=-1$, then $V_3=0$.  Keeping the notation from the proof of the preceding theorem, it follows $f_2  \partial_3+\partial_1  f_3=0$. This means that $\left({^2M_*,\delta_*}\right)$ is a complex. Therefore,  if $c_{V_3}=-1$, then $z$ is regular if and only if  $H_1\left({^1M_*}\right)=0$.
\end{remark}

\begin{remark}
	In the next section of the paper we will show that there is an optimal bound $c(A)$ such that $z$ is regular if and only if $z$ is $c(A)$-regular. If there is a minimal resolution $A\ot V_*$ of $S$ as a left $A$-module we will see that  $c(A)=c_{V_3}$.
	
Nevertheless, the freedom of choice of the resolution in the above theorem might be useful in the case when it is difficult to find the minimal resolution (if it exists).\vspace*{1ex}
\end{remark}

	We will give in the last section of the paper some examples of central extensions, whose first homology group $\mathrm{H}_1({}^1M_*)$ of the corresponding sequence $(M_*,\delta_*)$ is not trivial.

\section{On the vanishing of \texorpdfstring{$\h_1(^1M_\ast)$}{H1=0}.} \label{sec:vanH1}
	In this section our aim is to show that $^1M_\ast$ is exact in degree $1$, provided that $D$ is a central extension of $A$ as in the following subsection.

\begin{fact}[Notation and assumptions.]\label{fa:Notation}
	Let  $\pi_A:\Ga\to A $ denote the canonical surjective graded morphisms of $S$-rings from the tensor $S$-ring $T_A=T_S(A^1)$. Let $I_A:=\ker(\pi_A)$. By definition, $\pi_A$ extends $\pi_A^0=\Id_{S}$ and $\pi_A^1=\Id_{A^1}$. Thus $I_A^0=0=I_A^1$. We assume that $A^1$ is projective as a left $S$-module. Note that this condition is equivalent to the projectivity of $T_A$ as a left $S$-module. We also fix a projective resolution $(A\ot V_*,d_*)$ as in \eqref{eq:special_resolution}.

	We fix a central extension $D$ of $A$ which corresponds  to a filtered map $\alpha:R\to T_A$. Therefore, $D:=T_A[z]/\ig{\alpha_z(R)}$, see \S\ref{fa:extension}. There is a unique morphisms of graded $S$-rings $\pi_D: T_D\to D$ which lifts $\id_{S}$ and $\id_{D^1}$, where $R$ is a bimodule of relation for $A$. Note that $I_D=\ker(\pi_D)$ does not contain non-zero elements of  degree 0 and 1.

Under the above assumptions, we want to show that in some cases we can choose the resolution of $S$ such that $V_2$ and $R$ coincide up to an isomorphism of $S$-modules. The key points for achieving this goal are the following  useful lemmas.

\end{fact}

\begin{lemma}[Graded Nakayama Lemma] \label{le:Nakayama}
	Let $A$ be a graded connected $S$ring. Let $X$ be a graded left $A$-module such that $X^n=0$ for $n<0$. Then $X$=0 if and only if  $X/A^+X=0$. Furthermore a morphism $f:X\to Y $ is surjective if and only if its composition with the projection $Y\to Y/A^+Y$ is surjective.
\end{lemma}

\begin{proof}
	One proceeds by induction, using  the exact sequence:
\[
	0\xrightarrow{}\sum_{p=0}^{n-1}\A^{n-p}X^p\xrightarrow{}X^n\xrightarrow{}(X/\A^+X  )^n\xrightarrow{}0.\qedhere
\]	
If $p:Y\to Y/A^+Y$ denotes the projection then $\im (p\circ f)=(\im (f)+A^+Y)/A^+Y$. Thus  $p\circ f$ is surjective if and only if $\im (f)+A^+Y=Y$ if and only if $W/A^+W=0$ where $W:=Y/\im (f)$.
\end{proof}

\begin{lemma}\label{le:minimal}
	Let $\cX$ and $\cY$ be two graded $S$-bimodules. We assume that $f: A \ot \cX\to  A \ot \cY$ is a morphism of  graded $A$-modules such that $\cK=\ker(f)$ is an $S$-submodule of $ A ^+\ot \cX$. If $Z$ is a graded $S$-submodule  of $K$ such that $\cZ\oplus  (A ^+\cK)=\cK$, then there exists a graded morphism $g: A \ot \cZ\to  A \ot \cX$ such that the sequence
	\[
	A \ot \cZ \xrightarrow{g} A \ot \cX \xrightarrow{f}  A \ot \cY
	\]
	is exact and $\ker(g)\subseteq  A ^+\ot \cZ$. If the $S$-modules $A$ and $K$ are projective, then $\ker(g)$ and $Z$ are projective too. Furthermore, if $\ker(f)^{\leq p}=0$, for some $p$, then $\ker(g)^{\leq p+1}=0$.
\end{lemma}

\begin{proof}
	Let $\mu:A\ot Z\to K$ be the $A$-linear map induced by the action of $A$ on $K$. By definition of $Z$, the composition of $\mu$ with the projection $K\to K/A^+K\cong Z$ is  surjective.  By the graded version of the Nakayama Lemma we deduce that  $\mu$ is surjective. Hence, we can take $g:=i\circ \mu$, where $i$ denotes the inclusion of $K$ into  $A\ot X$. Obviously, $\im(g)=\ker(f)$.

	Let us assume that $A$ and $K$ are projective left $S$-modules. Obviously $Z$ is projective, being a direct summand of $K$. Thus  $A\ot Z$ is projective as an $S$-module. Using the exact sequence:
	\[
	0\rightarrow \ker(g) \rightarrow A\ot Z\rightarrow K\rightarrow 0,
	\]
	we deduce that $\ker(g)$ is a direct summand of $A\ot Z$ as an $S$-module, so $\ker(g)$ is projective.
	
	Let us now show that $\ker g\subseteq A^+\ot Z$. Let $x$ be an element in $\ker g \subseteq A^0\ot Z+A^+\ot Z$.  Thus  $x=1\ot u+\sum_{i=1}^p \lambda_i\ot u_i$, with $u,u_i\in Z$ and $\lambda_i\in A^+$.  Since $x\in \ker g$ it follows that $u+\sum_{i=1}^p \lambda_iu_i=0$. The sum from the left hand-side of the preceding relation belongs to $A^+K$, so $u\in Z\bigcap A^+K=0$. We conclude that $x\in A^+\ot Z$.
	
	For the last part of the lemma we notice that $Z^k\subseteq \ker(f)^k=0$ for $k\leq p$. On the other hand, $\ker(g)^k\subseteq (A^+\ot Z)^k=\oplus_{i=1}^k A^i\ot Z^{k-i}=0$, for $k\leq p+1$, so the lemma is completely proved.
\end{proof}

 \begin{fact}[Bimodule of relations for $A$ versus $V_2$.] \label{fa:A-resolution_2}
	Let us suppose that $R$ is left $S$-module such that $I_A=R\oplus\tilde{I_A}$ as left $S$-modules. Note that $R$ is not necessarily a bimodule, so it is not a bimodule of relations for $A$. Nevertheless $I_A=R\oplus\tilde{I_A}\subseteq RS+\tilde{I_A}\subseteq I_A$ so that $I_A=RS+\tilde{I_A}$ and hence, by Lemma \ref{lem:foot} we have that $I_A=\ig{RS}=\ig{R}$. Let $K_1$ denote the kernel of $d_1$. We consider the following diagram:
\[
\xymatrix @R=15pt{
	&0\ar[r]\ar@{=}[d]&I_A\ot A^1\ar[r]\ar@{_(->}[d]&I_A\ar@{_(->}[d]\\
	0\ar[r]&0\ar[d]\ar[r]&T_A\ot A^1\ar[d]_{\pi_A\ot A^1}\ar[r]^{m}&T_A^+\ar[d]^{\pi_A}\ar[r]&0\\
	0\ar[r]&K_1\ar[r]\ar@{=}[d]&A\ot A^1\ar[r]_{d_{1}}\ar[d]&A^+ \ar[r]&0 \\
	&K_1\ar[r]& 0& & }
\]
	By definition,  $m$ is the map induced by the multiplication in $T_A$. Thus $m$ is an isomorphism. Since $A$ is strongly graded, it follows that the rows are exact. The  middle column is exact as well, because $A^1$ is projective. Using Snake Lemma and the fact that  $R\oplus \wt I_A=I_A$, we get:
\[
	K_1\cong \frac{I_A}{ I_AA^1}=\frac{R+I_AA^1}{I_AA^1}+ \frac{I_AA^1+A^1I_A}{I_AA^1}\cong\frac{R}{R\bigcap I_AA^1}+A^1\left(\frac{I_A}{ I_AA^1}\right)\cong R+ A^+K_1.
\]
	Obviously, the sum $R+(A^+K_1)$ is direct. Using Lemma  \ref{le:minimal} we get an exact sequence:
\begin{equation}\label{eq:A-resolution}
	A \ot \cR\xrightarrow{d_2} A \ot \cV\xrightarrow{d_1} A \xrightarrow{d_0} S\xrightarrow{} 0
\end{equation}
where $\ker(d_2)\subseteq A^+\ot R$.
\end{fact}
	
\begin{fact}[A special projective resolution of $S$.]\label{fa:special_resolution}

	For the remaining part of this section we will assume that there exists a bimodule of relations $R$ for $A$ and that $R$ and $A^1$ are projective left $S$-modules. Thus we may consider the exact sequence \eqref{eq:A-resolution}.

We set $K_2:=\ker(d_2)$. We denote by  $i_1$ and $i_2$ the inclusion maps of $K_1$ and $K_2$ into $A\ot A^1$ and $A\ot R$, respectively.  	The map $d_2$ factors as $d_2=i_{1} d'_2$, where $d_2':A\ot R\to K_1$ is an epimorphism. Since $\im(i_2)=\ker(d_2)\subseteq A^+ \ot R$ it follows $S\ot_A i_2=0$. By dimension shifting \cite[Exercise 2.4.3]{We} we deduce that $\tor_3^A(S,S)\cong\ker(S\ot_A i_2)=S\ot_A K_2$. Thus $\tor_{3,n}^A(S,S)=0$ if and only if $(K_2/A^+K_2)^n=0$ if and only if $K_2^n=(A^+K_2)^n$. Note that $(A^+K_2)^n=\sum_{i=1}^n A^iK_2^{n-i}=A^1K_2^{n-1}$ as $A$ is strongly graded. Thus $\tor_{3,n}^A(S,S)=0$ if and only if $K_2^n=A^1K_2^{n-1}$.

Let $J:=\{n\in \mathbb{N}\mid \tor_{3,n}^A(S,S)\neq 0\}$. By induction we get that $K_2^p\subseteq\sum_{n\in J}AK_2^n$, for all $p\in\N$, so
\begin{equation}\label{K2gen}
  K_2=\sum_{n\in J}AK_2^n.
\end{equation}
	Our goal now is to produce a morphism of graded $A$-modules $d_3:A\ot V_3\to A\ot R$ such that $V_3$ is projective, $\im(d_3)=\ker(d_2)$ and, in addition, $V_3^n=0$ if and only if $\tor_{3,n}^A(S,S)=0$.
	
	For every $n\in J$ we pick a projective $S$-module $V_3^n$ together with an epimorphism $u_n':V_3^n\to K_2^n$. For $n\notin J$ we set $V_3^n=0$ and $u_n'=0$. Set $V_3:=\oplus_{n\in \mathbb{N}} V_3^n$ and let $u':=\oplus_{n\in \mathbb{N}} u_n':V_3\to K_2$. Using the constructions from \S\ref{ssec:u-v}, we get an $A$-linear map $u:A\ot V_3\to K_2$ such that $u(1\ot x)=u'(x)$, for all $x\in V$. By \eqref{K2gen} it follows that $u$ is surjective. Clearly,
	if $d_3=i_2u$, then $\im(d_3)=K_2=\ker(d_2)\subseteq A^+\ot R$. As a consequence $S\ot_A d_3=0$.

	Finally, we point out that, if $\tor_3^A(S,S)=0$, by \eqref{K2gen} we get $K_2=0$. Thus in this case, we may take $V_3=0$ and $d_3=0$.

	In conclusion, our claim about the existence of the morphism $d_3$ has just been proved. We can now use $d_3$ to extend the exact sequence  \eqref{eq:A-resolution} with one more arrow. The resulting exact sequence can be completed in the usual way to a projective resolution:
\begin{equation}\label{eq:special_resolution-V3}
	\cdots\rightarrow  A \ot V_n\xrightarrow{d_n} A \ot V_{n-1}\xrightarrow{d_{n-2}}\cdots  \xrightarrow{d_4} A \ot V_3 \xrightarrow{d_3} A \ot R\xrightarrow{d_2} A \ot A^1\xrightarrow{d_1} A  \xrightarrow{d_0} S\rightarrow 0.
\end{equation}
	By construction this resolution has the following properties: $V_n^i=0$, for all $n\geq 3$ and $i=0,1$; the component $V_3^n$ is zero if and only if $\tor_3^n(S,S)$ is so; for $n=1,2,3$ we have $S\ot _A d_n=0$.
\end{fact}

\begin{remark}\label{re:R_projectiv}
	Let us suppose that $R$ is a left $S$-module such that $R\oplus \wt I_A=I_A$. If $A$ is projective as a left $S$-module, by applying Lemma \ref{le:minimal} for $f=d_0$ and $g=d_1$, we deduce that $K_1$ is projective. We conclude, for free, that $R$ is projective as a left $S$-module, since  by the foregoing observations it is a direct summand of $K_1$. Thus,  in the case when $A$ is left $S$-projective, a bimodule of relations for $A$ is always left $S$-projective.
\end{remark}\vspace*{1ex}

\begin{remark}\label{re:d_2}
	The map $d_2$ can be described as follows. Proceeding as in Remark \ref{re:tau_n}, it is not difficult to show that the connecting morphism from $I_A$ to $K_1$ is precisely the restriction of $(\pi_A\ot A^1)  m^{-1}$ to $I_A$. We shall use a Sweedler-like notation, writing an element $x\in T_A\ot A^1$ as:
\[
	x=\sum x_{T_A}\ot x_{A^1}.
\]
	By construction,  $d_2$ is induced by the restriction of the $A$-module structure map of $K_1$ to $A\ot V_2$.  Via the identification $R\cong (R+I_AA^1)/(I_AA^1)\subseteq K_1$, for $a\in A$ and $r\in R$, we get:
	\[
d_2(a\ot r)=\sum a\pi_A(r_{T_A})\ot r_{A^1},
	\]
\end{remark}

\begin{remark}\label{re:SotAd3}
We have seen that, if there exists a complement left $S$-submodule  $R$ of $\wt I_A$ in $I_A$ which is projective, then there is a projective resolution $(A\ot V_*,d_*)$  such that $V_2=R$ and $S\ot_A d_3=0$. Let us show that the converse also holds.

Let us assume, that $S\ot_A d_{3}=0$. Thus $S\ot_A d_2'$ is an isomorphism, so we can identify $V_2$ and $K_1/(A^+K_1)$. In particular, the latter left $S$-module is projective. Recall that $\wt I_A=  A^1I_A+ I_AA^1$ and $K_1\simeq I_A/I_AA^1$. Therefore  $K_1/A^+K_1$ and $ I_A/\wt I_A$  are isomorphic. Since the former left module is projective, there exists a   \textit{left} module $R$ such that $I_A=R\oplus \wt I_A$. Clearly, $R\simeq V_2$.
\end{remark}\vspace*{1ex}

\begin{fact}[A presentation of $D$ by generators and relations.] \label{fa:D-presentation}
	Let $T_A[z]$ denote the polynomial ring in the indeterminate $z$ with coefficients in $T_A$. It is a graded connected $S$-ring, whose component of degree $1$ is $T_A[z]^{1}=A^1\oplus Sz$.

	There is a unique morphisms of graded $S$-rings $\pi_D: T_D\to D$ which lifts $\id_{S}$ and $\id_{D^1}$. Note that $I_D=\ker(\pi_D)$ does not contain non-zero homogeneous elements of  degree 0 and 1.

	Recall that the evaluation at $0$  induces a surjective graded $S$-ring map $\phi_0:D\to A$ whose kernel is precisely $zD$, cf. \S\ref{fa:P_z}. Since $z$ is central in $D$, the left and the right $S$-submodules generated by $z$ coincide. Note that, the ideal $\ig{\alpha_z(R)}$ is contained in $T_A[z]^{\geq 2}$, so $D^1=T_A[z]^1=A^1\oplus Sz$. Thus $Sz$ is a free $S$-module and	the homogeneous component of $\phi^1_0$ coincides with the canonical projection $\xi:A^1\oplus Sz\to A^1$. Note that $D^1$ is projective as a left $S$-module, since $A^1$ is projective by assumption.

	We denote by  $\zeta:A^1\to D^1$ the canonical inclusion, which is an $S$-bimodule section of $\xi$. By the universal property of the tensor $S$-ring, the maps $\xi$ and $\zeta$ induce morphisms of graded $S$-rings between $T_D:=T_S(D^1)$ and $T_A$, which will be denoted by the same symbols. Note that  $\zeta:T_A\to T_D $ is a section of $\xi$ so we can regard $T_A$ as a graded $S$-subring of $ T_D $ via $\zeta$. Henceforth, the $n$-degree homogeneous component $ T_D ^n$ is generated as a left or right $S$-module by all tensor monomials $x_1\ot \cdots \ot x_n$, where each $x_i$ is either in  $A^1$ or $x_i=z$, since every element in $D^1$ can be written in a unique way as  $a+sz$, for some  $a\in A^1$ and $s\in S$.

 	Using the universal property of the tensor $S$-ring $T_D$, there exists a unique morphism of graded $S$-rings $\gamma: T_D\rightarrow T_A[z]$, which extends the identity maps of $S$ and $D^1=T_A[z]^1$. Clearly $\gamma$ is surjective since $D$ is strongly graded and $\gamma^0$ and $\gamma^1$ are isomorphisms. Moreover, the $S$-ring morphism $\gamma \zeta$ and the inclusion of $T_A$ into $T_A[z]$ coincide, as they are equal on $T_A^{\leq 1}$.

	Obviously, the ideal generated by $[A^1,z]=\{z\ot a-a\ot z\mid a\in A^1\}$  is included into $\ker(\gamma)$. As a matter of fact, these ideals are equal. Indeed, for $x\in \ker(\gamma)^n$, there are $x_0, \dots, x_n$ so that $x_i\in T_A^i$ and
\[
	x=x'+\sum_{i=0}^n x_i\ot z^{\ot (n-i)},
\]
	where $x'$ is some element in the ideal generated by $[A^1,z]$. Since $\gamma(x)=\gamma(x')=0$ we deduce that the coefficients of the polynomial $\sum_{i=0}^n x_iz^{n-i}$ must be zero. Thus $x=x'$.

	Recall that throughout this section, $D:=T_A[z]/\ig{\alpha_z(R)}$. Let  $\pi_D':T_A[z]\rightarrow D$ denote the canonical map. Since the morphism of graded $S$-rings $\pi_D$ and $\pi'_D \gamma$ are equal on $T_D^{\leq 1}$, they must be identical.

	Finally, the set $\{z^n\mid n\in\N\}$ is a basis of $T_A[z]$, regarded  as a left $T_A$-module. Thus there is a unique morphism  $\lambda:T_A[z]\to T_D$ of left $T_A$-modules which maps $z^n$ to $z^{\ot n}$, for all $n\in\N$. Obviously, $\lambda$ is an $T_A$-linear section of $\gamma$.

	Summarizing, we have constructed the maps from the following diagram:
\[\xymatrix{
 	& & &T_A[z] \ar@<-0.75ex>[dl]_{\lambda}  \ar[dr]^{\pi'_D}& \\
	0\ar[r]& I_D\ar[d]_{\xi}\ar[r]& T_D\ar[d]_-{\xi}\ar[rr]^{\pi_D}\ar[ur]_{\gamma}& & D \ar[d]_-{\pi}\ar[r]&0\\
	0\ar[r]&I_A\ar[r]&T_A\ar[rr]^{\pi_A}\ar@<-0.75ex>[u]_{\zeta}& &A \ar[r]&0}
 \]
	If $\theta:R\to T_D$ is defined  by $\theta:=\lambda \alpha_z$, then:
\begin{equation}\label{eq:theta(r)}
 	\theta(r)=\sum_{i=0}^{\infty}\zeta\left(\alpha_i(r) \right)\otimes z^{\otimes i},
\end{equation}
	for every $r\in R$. By the definition of $\theta$, we also get $\theta(R)=\lambda\left(P_z\right)$.
\end{fact}

\begin{lemma}\label{le:Rz}
	Let $R_D:=\lambda(P_z)+ [z,A^1]$. The sum defining $R_D$ is direct and $R_D$ is an  $S$-bimodule of relations for $D$.
\end{lemma}

\begin{proof}
	Since $\lambda$ is a section of $\gamma$ it follows  that $\gamma^{-1}(P_z)=\lambda(P_z)\oplus\ker(\gamma)=\lambda(P_z)\oplus \ig{[z,A^1]}$. Notice that, in particular, the sum that appears in the definition of $R_D$ is direct. Since $\pi_D=\pi'_D \gamma$, we have $I_D=\ig{\gamma^{-1}(P_z)}=\ig{\lambda(P_z)+ [z,A^1]}=\ig{R_D}$.

	Let $\wt I_D:=I_D\ot D^1+D^1\ot I_D$. If $x\in  R_D\bgi \wt I_D$, then
\begin{equation*}
x=\theta(r)+\big(z\ot a-a\ot z\big),
\end{equation*}
	for some $r\in R$ and $a\in A^1$. By \eqref{eq:theta(r)} and taking into account that $\xi: T_D \to T_A$ is a morphism of $S$-rings, $\xi(z)=0$ and $\zeta$ is a section of $\xi$, we deduce that $\xi(x)=r$. On the other hand, $\xi(\lambda(P_z))\subseteq R$ and $\xi$ vanishes on  $[z,A^1]$. Thus $\xi(I_D)=I_A$. In particular, $r\in R\bgi \wt I_A=0$. It follows  that $x=z\ot a-a\ot z$, so $x$ is an element of degree $2$ in $T_D$. But $\wt I_D\subseteq T_D ^{>2}$, so $x=0$. This completes the proof of the fact that $R_D$ is a bimodule of relations for $D$.
\end{proof}

	In order to show that $^1M_\ast$ is exact in the present setting, we define a new morphism of graded $S$-modules $\rho:A^1(-1)\to [z,A^1]$ by $\rho(a)=a\otimes z-z\otimes a$, for all $a\in A^1$.

\begin{lemma}\label{le:gamma12}
	The maps $\theta : R \to \lambda(P_z)$ and $\rho:A^1\to [z,A^1]$  are isomorphisms of graded $S$-bimodules. The left $S$-module $R_D$ is projective.
\end{lemma}

\begin{proof}
Both $\theta$ and $\rho$ are surjective by construction. To show that $\rho$ is injective, let us take $a\in \ker \rho$.  Since $D^1=A^1\oplus Sz$ it follows that $D^1\ot D^1=(A^1\ot Sz)\oplus(Sz\ot A^1)\oplus (Sz\ot Sz)\oplus (A^1\ot A^1)$. On the other hand, $a\ot z$ and $z\ot a$ leave in the first and the second summands, respectively. Because $\rho(a)=0$, it follows that both $a\otimes z$ and $z\otimes a$ must be zero. Since $\{z\}$ is a basis for  the free left $S$-module $Sz$, it follows that $a=0$.
%Evaluating $\gamma  m_{T_D} (\zeta\xi\otimes \id_{T_D})$ at $\rho(a)$, we get the identity $az=0$ in $T_A[z]$. Thus  $a=0$.
For checking that $\theta$ is injective as well, we note that $\gamma \theta=\alpha_z$ and that $\alpha_z$ is injective.
	
Clearly, $R_D$ is left $S$-projective, since $R$ and $A^1$ have this property, by the standing assumptions in this section, and $\lambda$ and $\rho$ are isomorphisms.
\end{proof}

	 Recall that $D^1$ and  $R_D$ are left $S$-projective and that $R_D$ is a bimodule of relations for $D$. Thus, all properties of $A$ that we discussed above hold for $D$ as well. First, we get an  exact sequence:
\begin{equation}\label{eq:D_resolution}
 	D\ot R_D\xrightarrow{d'_2}D\ot D^1\xrightarrow{d'_1}D\rightarrow S\rightarrow 0,
\end{equation}
	which can be completed to a  projective resolution of $S$ as a $D$-module. Moreover, the kernels of $d_1'$ and $d_2'$ are included into $D^+$ and $D^+\ot R_D$, respectively. Secondly, by proceeding as in Remark \ref{re:d_2}, we get a similar description of  the map $d_2'$. For  $x\in D$ and $r\in R_D$,
\[
	d_2'(x\ot r)=\sum x\pi_D({r}_{T_D})\ot r_{D^1},
\]
On the other hand, $d_1'(x\ot y)=xy$, for all $x\in D$ and $y\in D^1$.

\begin{fact}[The morphisms $\delta_1$, $\delta_2$ and $f_2$.]\label{sec:explicitmaps}
	Let us explain the construction of the new maps that appear in Figure \ref{eq:diagram}.
\begin{figure}
\begin{equation*}
	\xymatrix @C=10pt @R=18pt {
		&T_D\otimes R \ar[rr]|-{\overline{\partial}_2} \ar[dl]|-{\xi\otimes R} \ar[dd]|-(.51){\hole} |-(.7){\pi_D\otimes R}&& T_D\otimes {A^1} \ar[dl]|-{\xi\otimes {A^1}} \ar[rr]|-{\overline{\partial}_1} \ar[dd]|-{\hole}|-(.7){\pi_D\otimes {A^1}}& & T_D \ar@<-0.45ex>[dl]_{\xi} \ar[dd]|-(.7){\pi_D} \\
		T_A\ot R\ar[rr]|-(.65){\bar{d}_2}\ar[dd]|-(.65){\pi_A\otimes R} & & T_A\ot {A^1}\ar[rr]|-(.7){\bar{d}_1}\ar[dd]|-(.65){\pi_A\otimes {A^1}}& &T_A\ar[dd]|-(.65){\pi_A} \ar@<-0.45ex>@{-->}[ur]_-{\zeta} &\\
		&D\otimes R \ar[dl]|-{\pi\otimes R} \ar[rr]|-(.3){\partial_2}|(.5){\hole} && D\otimes {A^1} \ar[dl]|-{\pi \otimes {A^1}} \ar[rr]|-(.3){\partial_1}|(.54){\hole} && D\ar[dl]|-{\pi} \\
		A \ot R\ar[rr]|-{d_2} &&  A \ot {A^1}\ar[rr]|-{d_1}&& A &
	}
	\end{equation*}
	\caption{}\label{eq:diagram}
\end{figure}

	By Remark \ref{re:d_2}, the maps  $\bar{d}_1$ and $\bar{d}_2$, which are defined by  $\bar{d}_1(t\otimes v)=tv$  and $\bar{d}_2(t\otimes r)=\sum t\cdot r_{T_A}\otimes r_{A^1}$, make commutative the front squares. Using the constructions from  \S\ref{ssec:u-v} we now define $\overline{\partial}_i:={}^{\xi}\bar{d}_i$, for $i=1,2$.  Thus, the parallelograms on the top of the diagram are commutative. Since $\zeta$ is a  section of $\xi$, we can choose  $\overline{\partial}_1$ and $\overline{\partial}_2$ such that:
\begin{equation*}
	\overline{\partial}_1(x\otimes a)=xa, \qquad  \overline{\partial}_2(x\otimes r)=\sum x\zeta(r_{T_A})\otimes r_{A^1} .
\end{equation*}
	As $\pi_D:T_D\to D$ is surjective, we can use once again the constructions from \S\ref{ssec:u-v} to define the arrows $\partial_i:={}_{\pi_D}\overline{\partial}_i$. If $\wt{\zeta}= \pi_D\zeta$, then:
\begin{equation*}
	\partial_1(x\otimes a)=xa,\qquad  \partial_2(x\otimes r)=\sum x\wt{\zeta}(r_{T_A})\otimes r_{A^1}.
\end{equation*}
	Since $\pi_D$ is surjective it follows that the parallelograms on the bottom of the Figure  \ref{eq:diagram} are commutative as well. In conclusion all squares from the same figure are commutative.

	Therefore, taking into account \S\ref{ssec:3.5}, for defining the differential maps $\delta_1$ and $\delta_2$ of $^1 M_\ast$, we can choose the maps $\partial_1$ and $\partial_2$ as in Figure \ref{eq:diagram}.

	The above explicit description of $\partial_1$ and $\partial_2$ allows us to compute  the map $f_2:D\otimes R\to D$,  which was defined in \eqref{eq:partial_n}. Note that, for $r\in R$ and $x\in D$, we have $\partial_1(\partial_2(x\otimes r) =x\wt{\zeta}(r)$, since $\wt{\zeta}$ is a ring morphism. Furthermore, $(\pi_D \theta)(r)=(\pi_D' \gamma \lambda \alpha_z)(r)=\pi_D'(\alpha_z(r))=0$. Thus, by \eqref{eq:theta(r)}, we get:
\begin{equation*}
	\wt{\zeta}(r)=-z\sum_{i=1}^{\infty}\wt{\zeta}\left(\alpha_i(r)\right)z^{i-1}.
\end{equation*}
	In view of \eqref{eq:partial_n}, for $x$ and $r$ as above, one can define $f_2$ by the relation:
\begin{equation}\label{eq:f_2}
	f_2(x\otimes r)=\sum_{i=1}^{\infty}x \wt{\zeta}\left(\alpha_i(r)\right)z^{i-1}.
\end{equation}
\end{fact}

\begin{lemma}\label{le:zetacomult}
	If $x\in T_A$, then $\sum\zeta(x)_{T_D}\otimes \zeta(x)_{D^1}=\sum \zeta\left(x_{T_A}\right)\otimes \zeta\left(x_{A^1}\right)$.
\end{lemma}

\begin{proof}
	Without loosing in generality, we may assume that $x\in T_A$ is a homogeneous element of degree $n$. Then $x=\sum x_{T_A}\ot x_{A^1}$, with $x_{T_A}\in T_A^{n-1}$ and $x_{A^1}\in A^1$. Since $\zeta$ is a morphism of rings   we get $\zeta(x)=\sum_{i=1}^p \zeta(x_{T_A})\ot \zeta(x_{A^1})$. We conclude by remarking that $\zeta(x_{T_A})\in T_D^{n-1}$ and $\zeta(x_{A^1})\in D^1$, as $\zeta$ is a graded map.
\end{proof}

\begin{theorem}\label{prop:H10}
	Let $R\subseteq T_A$ be a bimodule of relations for a strongly graded $S$-ring $A$ such that $A^1$ and $R$ are left $S$-projective. If $D$ is a central extension of $A$ associated to some filtered map $\alpha$, then the corresponding sequence $(^1M_\ast,\delta_*)$ is isomorphic to \eqref{eq:D_resolution}. In particular, $\h_1\left({}^1M_\ast\right)=0$.
\end{theorem}

\begin{proof}
	Recall that the sequence $^1M_\ast$ was defined using the maps $\partial_1$ and $\partial_2$ that were constructed in \S\ref{sec:explicitmaps}. Moreover we have:
\begin{equation}
	D\otimes R_D=D\otimes (\lambda(P_z)\oplus [z,A^1])\cong (D\otimes \lambda(P_z))\oplus (D\otimes [z,A^1]).
\end{equation}
	By definition, $M_1=(D\ot {A^1})\oplus D(-1)$ and $M_2=(D\ot R) \oplus  \big( D\ot {A^1}(-1)\big)$. Let ${\vartheta}_1:M_1\to D\ot D^1$ be the isomorphism which coincides with $D\ot \zeta$ on $D\ot {A^1}$ and mapping $x$ to $x\ot z$, for all $x\in D$. We also have an $S$-linear isomorphism ${\vartheta}_2:=(D\otimes \theta)\oplus(D\otimes\rho)$ from $M_2$ to $D\ot R_D$.  To conclude the proof, since $M_0=D$,  it is enough to show that the squares of the following diagram are commutative.
\begin{equation*}\label{eq:commsquares}
\xymatrix{
	M_2 \ar[d]_-{{\vartheta_2}} \ar[r]^-{\delta_2} & M_1 \ar[d]_-{\vartheta_1} \ar[r]^-{\delta_1} & M_0\ar@{=}[d] \\
	D\otimes R_D \ar[r]_-{d_2'}  & D\otimes D^1 \ar[r]_-{d_1'}  & D
}
\end{equation*}
	Let us pick a homogeneous element $r\in R^n$, for some $n\geq 2$. Thus:
\begin{align*}
	(d_2'  \vartheta_2)(1\otimes r) & =d_2'\left(1\otimes \theta(r)\right) = \sum \pi_D\big({\zeta(r)}_{T_D}\big)\ot {\zeta(r)}_{D^1}+\sum_{i=1}^{\infty}\pi_D\left(\zeta\left(\alpha_i(r)\right)\otimes z^{\otimes i-1}\right)\otimes z \\
 	&=\sum {\wt{\zeta}(r_{T_A})}\ot {\zeta(r_{{A^1}})}+\sum_{i=1}^{\infty}\wt{\zeta}\left(\alpha_i(r)\right)z^{i-1}\otimes z\\
 	& =(D\otimes \zeta)\left(\partial_2(1\otimes r)\right)+f_2(1\otimes r)\otimes z=(\vartheta_1 \delta_2 )(1\otimes r),
\end{align*}
	where for the third equation we used the relation from Lemma \ref{le:zetacomult}.
	Furthermore, for $a\in {A^1}$ we get:
\begin{align*}
	(d_2'  \vartheta_2)(1\otimes a)& =d_2'\big(1\otimes a\otimes z-1\ot z\otimes a)=a\otimes z-z\otimes a\\
	& = (D\otimes \zeta)(-z\cdot (1\otimes a))+\partial_1(1\otimes a)\otimes z= (\vartheta_1 \delta_2 )(1\otimes a).
\end{align*}
	Since $d_2'$ and $\vartheta_2$ are $D$-linear maps we conclude that the the left square of the above diagram is commutative. The other square of the same diagram is obviously commutative, by the definition of the maps $\delta_1$, $d_1'$ and $\vartheta_1$.
\end{proof}

\begin{theorem}\label{th:homological_condition}
	Let $R\subseteq T_A$ be a bimodule of relations of a strongly graded connected $S$-ring $A$ such that $R$ and $A^1$ are left $S$-projective.  If $D$ is the central extension of $A$ associated to the filtered morphism $\alpha:R\to T_A$, then $z$ is regular in $D$ if and only if either  $c_{V_3}=-1$ or $z$ is $c_{V_3}$-regular.
\end{theorem}

\begin{proof}
	Recall that, by Remark \ref{re:c(A)=-1}, for any resolution $A\ot V_*$ of  $S$ as in \eqref{eq:special_resolution}, with $V_n^i=0$ for $n\geq 3$ and $i\leq 1$, then either $c_{V_3}=-1$ or $c_{V_3}\geq 1$. Thus, in the former case, $z$ is regular as $V_3=0$. The theorem now is a direct consequence of Theorem \ref{th:regandcompl}, since by Theorem \ref{prop:H10} we have $H_1(^1M_*)=0$.
\end{proof}

\section{The Poincar\'e-Birkhoff-Witt Theorem and applications.}\label{sec:4}

	In this section we will use the homological characterization of regularity of $1$-degree homogeneous central elements  to prove the main results of the paper, several versions of Poincar\'e-Birkhoff-Witt Theorem, and to derive some applications of it.
	
	In the most general form, it can be stated for quotients of arbitrary strongly graded connected $S$-rings.
	
\begin{theorem}\label{th:PBW}
	Let $P$ denote an $S$-subbimodule of a connected strongly graded $S$-ring  $T$ such that $P\cap S=0$ and let $A:=T/\ig{R_P}$. We assume that $A^1$ is left $S$-projective and we fix a projective resolution $A\ot V_*$ as in \eqref{eq:special_resolution}, with $V_n^i=0$ for $n\geq 3$ and $i\leq 1$. The bimodule  $P$ is of PBW-type if and only if $H_1({^1M_*})=0$ and either $c_{V_3}=-1$ or $P$ satisfies $(\cJ_1)-(\cJ_{c_{V_3}})$, where $(^1M_*,\delta_*)$ corresponds  to the central extension  $D:=T[z]/\langle P^*\rangle$.
\end{theorem}
	
	\begin{proof} Note that the relations $P\cap S=0$ and $R_P^0=0$ are equivalent, and that they imply $(\cJ_0)$.  Thus, if  $P\cap S=0$, then $D^0= S= A^0$, as it is required in Section \ref{sec:3}.
	We have already noticed that $z$ is regular, provided that $c_{V_3}=-1$ and the first homology group of $^1M_*$ vanishes. If $c_{V_3}\geq 1$ and $P$ satisfies the Jacobi conditions $(\cJ_{n})$ for all $n\leq c_{V_3}$, then the central element $z$ in $D$ is $c_{V_3}$-regular, cf. Theorem \ref{Th:JacAnn}. Taking into account Theorem \ref{th:regandcompl}, it follows that $z$ is regular in $D$. Thus, using once again Theorem \ref{Th:JacAnn}, the Jacobi conditions $(\cJ_n)$ hold for all $n\in\N$, which means that $P$ is of $PBW$-type, cf. Theorem \ref{Th:PBWJac}. The converse follows immediately by applying Theorem \ref{Th:PBWJac}, Theorem \ref{Th:JacAnn} and Theorem \ref{th:mainth}.
	\end{proof}

\begin{fact}[The homological complexity of $A$.]\label{fa:notation_PBW}
	From now on we assume that there exists a bimodule of relations $R\subseteq T_A$ for $A$ and that $A^1$ and $R$ are projective as left $S$-modules. We will always work with   resolutions $A\ot V_*$ as in  \eqref{eq:special_resolution},  with $V_n^i=0$ for $n\geq 3$ and $i=0,1$. We denote  the class of all resolutions of this type by $\mathfrak{P}$.
	
	We define the \textit{homological complexity} of $A$ as the complexity of the graded module $\tor_{3}^A(S,S)$. The homological complexity of $A$  will be denoted by $c(A)$. Thus, by definition, if $\tor_{3}^A(S,S)=0$, then $c(A)=-1$. Otherwise, we have
\begin{equation}\label{eq:lbound_c(A)}
	c(A)= \sup\{n-1\mid\tor_{3,n}^A(S,S)\neq 0\}.
\end{equation}
	In particular, we deduce that $c(A)\geq 0$, provided that $\tor_{3}^A(S,S)\neq 0$.
	
	The homological complexity of $A$ should not be confused with the complexity of $A$ as a graded $S$-module, that is $c_A:=\sup\{n-1\mid A^n\neq 0\}$.
	
	We claim that $c(A)$ may be computed using the following  formula:
\begin{equation}\label{eq:complexity_inf}
	c(A)=\inf\{c_{V_3}\mid A\ot V_*\in \mathfrak{P}\}.
\end{equation}
	Indeed, let us denote the right-hand side of the above relation by $c'(A)$. By \S\ref{fa:special_resolution}, there exists a resolution $A\ot V_*$ in $\mathfrak{P}$, not necessarily minimal, such that $c(A)=c_{V_3}$. Then $c'(A)\leq c_{V_3}= c(A)$.
	
	To prove the other inequality, we may assume that $c(A)\geq 0$, otherwise \eqref{eq:complexity_inf} trivially holds.  We pick an arbitrary projective resolution $A\ot W_*$  in $\mathfrak{P}$. Then $W_3^{c(A)+1} \neq 0$,  since $\tor_{3}^A(S,S)$ is a graded subquotient of $W_3$. We get $c_{W_3}\geq c(A)$, so $c'(A)\geq c(A)$.
	
	If $R\neq 0$, it is worthwhile to mention that  the relation $c(A)=-1$ is equivalent to the fact  that the projective dimension of $S$ is equal to $2$. Let us assume that  $\tor_3^A(S,S)=0$. We fix a resolution $A\ot V_*$ in $\mathfrak{P}$ such that $c_{V_3}=-1$, that is $V_3=0$. This means that the map $d_2$ in \eqref{eq:A-resolution} is injective. Therefore, this sequence yields us a minimal resolution of $S$ of length $2$. The other implication is obvious.

\end{fact}\vspace*{1ex}		

\begin{remark}\label{re:c(AxV*)}
	Recall that a\textit{ minimal  resolution} of $S$ is an exact sequence $(A\ot V_*,d_*)$ of $A$-graded module, such that $V_n$ is projective and $\ker(d_n)\subseteq A^+\ot V_n$ (equivalently, $S\ot_A d_{n+1}=0$), for all $n\geq 0$. For these resolutions (if they exist),   $\Tor_{n}^A (S,S)$ and $V_n$ are isomorphic as graded modules. Moreover, for any strongly graded connected $S$-ring we have  $A^1=\tor_1^A(S,S)$. Thus, if there exists a minimal resolution $(A\ot V_*,d_*)$, then it is an exact sequence as in \eqref{eq:special_resolution} and $A^1$ is projective as a left $S$-module. In this case  $c(A)=c_{V_3}$. Furthermore, if $S$ is a field, then our definition of complexity coincides with that one from \cite{CS}.
	
	In light of \cite{Ei}, for any semisimple ring $S$, there exists a minimal resolution of $S$. In general, to determine it explicitly  may be an intricate problem. From this point of view,  it is an advantage that in Theorem \ref{th:PBW} we can use \textit{any}  resolution $A\ot V_*$ as in \eqref{eq:special_resolution}. On the other hand, working with a resolution $A\ot V_*$ such that $c_{V_3}=c(A)$, guarantees that the number of Jacobi relations that we  must check is minimum.
\end{remark}

We are now ready to prove a second version of the Poincar\'e-Birkhoff-Witt Theorem. We use the same notation as in the previous section: $A$ is a strongly graded connected $S$-ring, $\pi_A:T_A\to A$ denotes the canonical morphism of graded $S$-rings from the tensor $S$-ring of $A^1$ to $A$ and $I_A=\ker(\pi_A)$.
\begin{theorem}\label{co:PBW}
	Let $R$ be a generating subbimodule of $I_A$. We assume that $A^1$ and $R$ are left $S$-projective.
\begin{enumerate}
	\item If $P$ is a subbimodule of $T_A$ such that $R\subseteq R_P$, then  $U(P)$ is a PBW-deformation of $A$ if and only if the ideals generated by $R$ and $R_P$ coincides, $H_1({^1M_*})=0$ and either $c(A)=-1$ or $P$ satisfies $(\cJ_1)-(\cJ_{c(A)})$, where $(^1M_*,\delta_*)$ corresponds to the central extension  $D:=T[z]/\langle P^*\rangle$.
	
	\item Let $\alpha:R\to T_A$ be a map of filtered $S$-bimodules. If $P:=\alpha (R)$ and $R$ is a bimodule of relations for  $A$, then $U(P)$ is a PBW-deformation of $A$ if and only if either $c(A)=-1$ or $P$ satisfies $(\cJ_1)-(\cJ_{c(A)})$.
\end{enumerate}		
		
\end{theorem}

\begin{proof}
	Recall that $U(P)$ is a $PBW$-deformation of $A$ if and only if $\ig{R}=\ig{R_P}$ and $P$ is of $PBW$-type. In particular,
\[
	P\cap S=R_P^0=\ig{R_P}^0=\ig{R}^0=I_A^0=0.
\]
	Thus, the first part of the theorem follows by Theorem \ref{th:PBW}, where the resolution $A\ot V_*$ is taken such that $c(A)=c_{V_3}$. The second part is a consequence of (1), remarking that $D:=T[z]/\langle P^*\rangle=T[z]/\langle \alpha_z(R)\rangle$. Indeed, for a bimodule $P$ which is associated to a filtered morphism $\alpha$, we know that $R=R_P$ and $H_1\left({^1M_*}\right)=0$, cf. Theorem \ref{prop:H10}.
\end{proof}

\begin{theorem}\label{te:PBW}
Let $A$ be a strongly graded connected $S$-ring, which has a bimodule of  relations $R$.  Let $P$ be a subbimodule of $T_A$ such that $R\subseteq R_P$. We assume that $R$ and $A^1$ are left projective, and  $R_P$ is a projective $S$-bimodule. The $S$-ring $U(P)$ is a $PBW$-deformation of $A$ if and only if $\ig{R}=\ig{R_P}$, there exists a morphism of filtered $S$-bimodules $\alpha:R\to T_A$ such that $\alpha(R)\subseteq P\subseteq \ig{\alpha(R)}$ and either $c(A)=-1$ or $\alpha(R)$ satisfies the Jacobi conditions $(\cJ_1)-(\cJ_{c(A)})$.
\end{theorem}

\begin{proof}
	Since $R_P$ and $\gr{P}$ are isomorphic graded bimodules, it follows that the components of $\gr {P}$ are projective bimodules. Therefore, $P^{\leq n}$ is a direct summand of $P^{\leq n+1}$, for all $n$. According to Corollary \ref{cor:largeR}, $U(P)$ is a $PBW$-deformation if and only if $\ig{R}=\ig{R_P}$ and there exists a morphism $\alpha : R \rightarrow T_A$ of filtered $S$-bimodules  such that $\alpha(R)\subseteq P\subseteq \ig{\alpha(R)}$ and $\alpha(R)$ is of PBW-type. Note that $R_{P'}=R$, where $P'=\alpha(R)$, so we can now apply Theorem \ref{co:PBW} (2).
\end{proof}

\renewcommand\thesubsection{\thesection.\Alph{subsection}}

\subsection{Pure bimodule of relations}\label{fa:Pure}\hfill\vspace*{1.5ex}
		
	As a first application of our main results, we will consider the case when $A$ is a strongly graded connected $S$-ring such that the ideal $I_A\subseteq T_A$  is generated by an $n$-pure   $S$-subbimodule $R$.
Since $I_A\subseteq T_A^{\geq 2}$, we have $n\geq 2$.
	We know that  $R$ is a bimodule of relations for $A$, cf. Corollary \ref{co:Pure}.

	Assume that $R$ and $A^1$ are projective as left $S$-modules. Then, by definition of the homological complexity, either $c(A)=-1$ or $c(A)\geq n$. Indeed, $K_2=\ker(d_2)\subseteq A^+ \ot R$. Since $(A^+ \ot R)^{\leq n}=0$,  we have $K_2^i=0$, for all $i\leq n$. Thus $K_2^{i}=A^1K_2^{i-1}$ for the same values of $i$ as before. Hence, in view of \S\ref{fa:special_resolution}, we get $\tor_{3,i}^A(S,S)=0$, for all $i\leq n$.

	We now fix an $S$-subbimodule $P\subseteq T_A^{\leq n}$, as in Example \ref{ex:BG}. Thus, by construction, $R:=p^n(P)$ is $n$-pure. Moreover, $\Phi_{P,R}$ is an isomorphism if and only if $P^{\leq n-1}=0$ and $P$ is of $PBW$-type. By Corollary \ref{co:BG}, we know that the vanishing of $P^{\leq n-1}$ is equivalent to the existence of a morphism of filtered bimodules $\alpha:R\to T_A$ such that $P=\alpha(R)$. In this case we also have $R_P=R$.

\begin{theorem}\label{te:PBW_pure}
	We keep the above notation. We assume that $\Tor_3^A(S,S)$ is $(n+1)$-pure  and that $A^1$ and $R$ are projective.  Then $U(P)$ is of $PBW$-deformation of $A$, if and only either $\Tor_3^A(S,S)=0$ or
	\[
	P^{\leq n-1}=0 \quad\text{and}\quad (VP+PV)^{\leq n}\subseteq P.
	\]
\end{theorem}

\begin{proof}
	We can suppose that $\Tor_3^A(S,S)\neq 0$. We have noticed  that  $P^{\leq n-1}=0$ is a necessary condition for $\Phi_{P,R}$ being bijective so we can assume this relation is true. In particular $P=\alpha(R)$ for some morphism of filtered bimodules $\alpha:R\to P$. By the foregoing remarks it follows that $c(A)=n$.  Since $A^1$ and $R$ are projective, by Theorem \ref{co:PBW}, $P$ is of $PBW$-type if and only if $P$ satisfies the conditions $(\cJ_1)-(\cJ_n)$.  Hence, it is enough to prove that the relations $(\cJ_1)-(\cJ_n)$ holds, if and only if the second relation from the statement of the theorem is true. Since $P^{\leq n-1}=0$, we deduce by induction that $P_k=0$ for all $k\leq n-1$. On the other hand, by hypothesis and definition of $P_n$, we get $P_n=P$. Hence, the relations $(\cJ_1)-(\cJ_{n-1})$ hold and  $(\cJ_n)$ is equivalent to the second relation from the statement of the  theorem, cf. Proposition \ref{pr:minimal'} (2).
\end{proof}

\begin{remark}
	A similar result was obtained in \cite[Theorem 3.4]{BeGi}. In \textit{loc. cit.}, $S$ is a von Neumann regular ring, but  $A^1$ and $R$ are  not necessarily projective as left $S$-modules.
\end{remark}	\vspace*{1ex}

	We have seen that, for $P$ as in Example  \ref{ex:BG}, the condition $P^{\leq n-1}=0$ is equivalent to the existence of a morphism of filtered bimodules $\alpha :R\to T_S(V)$ such that $P=\alpha(R)$. We want now  to show that  $(VP+PV)^{\leq n}\subseteq P$ holds if and only for any $x\in X :=(R\ot V)\bigcap (V\ot R)$  and $0<i<n$ we have:
\begin{align*}
	&(V\ot \alpha_1-\alpha_1\ot V)(x)\in R \label{jpr0}\tag{$\cJ'_0$}\\
	 &\alpha_i  (V\ot\alpha_1 -\alpha_1\ot V)(x)=-(V\ot\alpha_{i+1} -\alpha_{i+1}\ot V)(x) \label{jpri}\tag{$\cJ'_i$}\\
	 &\alpha_n  (V\ot\alpha_1 -\alpha_1\ot V)(x)=0.\label{jprn}\tag{$\cJ'_n$}
\end{align*}
	In this case we will say that $\alpha$ satisfies the conditions \eqref{jpr0}--\eqref{jprn}.
\begin{theorem}[compare with {\cite[Therem 3.4]{BeGi}}]\label{te:PBW-pure}
	We keep the notation from Example \ref{ex:BG}.   Let $A:=A(R)$. We assume that  $A^1$ and $R$ are left projective  and that $\Tor_3^{A}(S,S)$ is $(n+1)$-pure.  Then $U(P)$ is a $PBW$-deformation of $A$, if and only if there is a morphism of filtered bimodules $\alpha:R\to T_S(V)$ such that $P:=\alpha(R)$ and either $\Tor_3^{A}(S,S)=0$ or the conditions \eqref{jpr0}--\eqref{jprn} are satisfied.
\end{theorem}

\begin{proof}
	It is enough to prove that, for any morphism of filtered bimodules $\alpha:R\to T_S(V)$ such that $P=\alpha(R)$, the relation $(V\ot P+ P\ot V)^{\leq n}\subseteq P$ holds if and only if $\alpha$ satisfies  \eqref{jpr0}--\eqref{jprn}. We first show that:
	\begin{equation}\label{eq:alpha(R)}
	(V\ot P+ P\ot V)^{\leq n}=(V\ot \alpha-\alpha\ot V)(X ).
	\end{equation}
	Let $t\in (V\ot P+ P\ot V)^{\leq n}$. There exist $v'_1,\dots v'_p,v''_1,\dots,v''_q\in V$ and $r'_1,\dots r'_p,r''_1,\dots,r''_q\in R$ such that
	\[
	t=\sum_{i=1}^p v'_i\ot \alpha(r'_i) +\sum_{j=1}^q \alpha(r''_j)\ot v''_j.
	\]
	Since $t\in T_S(V)^{\leq n}$ it follows that $\sum_{i=1}^p v'_i\ot  r'_i =-\sum_{j=1}^q  r''_j \ot v''_j$. Therefore $x:=\sum_{i=1}^p v'_i\ot  r'_i $ belongs to $X $ and $t=(V\ot \alpha-\alpha\ot V)(x)$. The other inclusion can be proved in a similar way. Let $t=(V\ot \alpha-\alpha\ot V)(x)$ be an element in the right hand-side of the relation that we are proving. Obviously, $t\in V\ot P+ P\ot V$. Since $\alpha_0(r)=r$, for all $r\in R$, and $\deg(\alpha_i)=-i$,  it is not difficult to show that $t$ is also an element in $T_S^{\leq n}(V)$.
	
	Let us assume that $P$ satisfies the relation $(V\ot P+ P\ot V)^{\leq n}\subseteq P$. Let $x\in X $. By \eqref{eq:alpha(R)}, there exists $r\in R$ such that $(V\ot \alpha-\alpha\ot V)(x)=\alpha(r)$. By equating the homogeneous components of the two sides of this relation, we deduce that $\alpha$ satisfies \eqref{jpr0}--\eqref{jprn}.
	
	Conversely, let $x\in X $ and $t:=(V\ot \alpha-\alpha\ot V)(x)$. By \eqref{jpr0} it follows that $r:=(V\ot \alpha_1-\alpha_1\ot V)(x)$ is an element in $R$. Using the relations \eqref{jpri}, for $i=1,\ldots,n$, we deduce that $t=\alpha(r)\in P$.
\end{proof}

\subsection{\texorpdfstring{$PBW$}{PBW}-deformations of twisted tensor products.}\hfill\vspace*{1.3ex}

	In the second application of our main result we investigate the $PBW$-deformations of a twisted tensor product $\bts$, where $A$ is a strongly graded connected $\mK$-ring and $S$ is a $\mK$-ring.
\begin{fact}[Some categorical constructions related to twisting maps.] \label{fa:category}
	We fix a ring $\mK$ and a $\mK$-ring $S$. We denote by $\TS $ the category whose objects are couples $(V,\mt_V)$, where $V$ is a $\mK$-bimodule and $\mt_V:S\otK V\to V\otK S$ is a $\mK$-bilinear map which satisfies the relations:
	\begin{equation} \label{eq:t_W}
	\mt_V   (m_S\otK V) =(V\otK m_S)   ( \mt_V \otK S)  (S \otK  \mt_V)\qquad \text{and}  \qquad  \mt_V (1\otK v)=v\otK 1.
	\end{equation}
	A morphism in $\TS $ from $(V,\mt_V)$ to $(W, \mt_W)$ is a $\mK$-bilinear map $f:V\to W$ which  commutes with $\mt_V$ and $\mt_W$, that is
	$$(f\otK S) \mt_V=\mt_W  (S\otK f).$$
	If $(V,\mt_V)$  is an object in $\TS$, then $S$ acts  on  $V\otK S$ to the right via the multiplication of $S$. On the other hand, using the notation $\mt_V(s\otK v)=\sum v_{\mt_V}\otK s_{\mt_V}$, we define the \textit{left} $\mt_V$\textit{-twisted action} on $V\otK S$ by:
	\[
	s\cdot(v\otK s'):=\sum v_{\mt_V}\otK s_{\mt_V} s'.
	\]
	We will denote the resulting $S$-bimodule by $V\ot_{\mt_V} S$. The mapping $(V,\mt_V)\mapsto V\ot_{\mt_V} S$ defines a functor $\Psi$ from $\TS $ to the category of $S$-bimodules $\SMS$. The functor maps a morphism $f$ in $\TS $ to $f\otK S$. Taking the left regular action on $S\otK V$,  it follows that $t_V$ is a morphism of left $S$-modules, where on the codomain of $t_V$ one takes the left $\mt_V$-twisted action.
	
	We  point out that, for a field $\mK$ and a finite dimensional $\mK$-algebra $S$ , the maps satisfying \eqref{eq:t_W} are used in \cite{Ta},  in a left-right symmetric version, to characterize the $S$-bimodule structures on $V\otK S$, regarded as a right module via regular action.
	
	The category $\TS $ is monoidal with respect to the tensor product:
	\[
	(V,\mt_V)\ot (W,\mt_W):=(V\otK W, \mt_{V\otK W}), \qquad \mt_{V\otK W}:=(V\otK \mt_W)  (\mt_V\otK W).
	\]
	The unit object is $(\mK, \mt_\mK)$, where $\mt_\mK:S\otK \mK\to  \mK\otK  S$ is the canonical isomorphism. Taking on the category $\SMS$ the tensor product of $S$-bimodules, it follows that  $\Psi$  is a monoidal functor.
	
	Let $(V,\mt_V)$ be an object  in $\TS $ and let $X$ be an arbitrary $S$-bimodule. A  $\mK$-bilinear map $g: V\to X$  is called $\mt_V$-\textit{invariant} if and only if
	\begin{equation}\label{eq:g_invariant}
	s\cdot g(v)=\sum  g(v_{\mt_V})\cdot s_{\mt_V}.
	\end{equation}
	The set of $\mt_V$-in variant maps will be denoted by $\Hom_{\mK-\mK}^{\mt_V}(V, X)$.

	We mention that $\Hom_{\mK-\mK}(V, X)$ becomes an $S$-bimodule where the left action is induced by the one of $X$ while the right one associates to $g:V\to X$ and $s\in S$ the map $gs:v\mapsto\sum  g(v_{\mt_V})\cdot s_{\mt_V}$. Thus,  $\Hom_{\mK-\mK}^{\mt_V}(V, X)=\Hom_{\mK-\mK}(V, X)^S$. Recall that, by definition, the \textit{centralizer }of an $S$-bimodule $M$, is the abelian group $M^S$ containing the elements $m\in M$ such that $sm=ms$, for all $s\in S$.

	For future references, some basic properties of $\mt_V$-invariant maps  are collected in the next lemma.
\end{fact}

\begin{lemma} \label{le:Psi}
	Let $f_i:V_i\ot_{\mt_{V_i}} S\to X_i\ot_{\mt_{X_i}} S$  denote a morphism in $\SMS$, where $(V_i,\mt_{V_i})$, $(X_i,\mt_{X_i})$ are objects in $\TS $ and $i=1,2$. Let $ \psi_{V_1,V_2}:\Psi(V_1\otK V_2,\mt_{V_1\otK V_2})\to\Psi (V_1,\mt_{V_1}) \ot \Psi(V_2,\mt_{V_2} )$  and  $ \psi_{X_1,X_2}: \Psi(X_1\otK X_2,\mt_{X_1\otK X_2})\to\Psi (X_1,\mt_{X_1}) \ot \Psi(X_2,\mt_{X_2} )$ denote the canonical isomorphisms.
\begin{enumerate}
	\item  There exists a canonical isomorphism of abelian groups:
	\[
		(-)^{\wh { }}: \Hom_{S-S}(V\ot_{\mt_V} S, X)\to  \Hom_{\mK-\mK}^{\mt_V}(V, X),\qquad f\mapsto \wh f,
	\]
	where $\wh f:V\to X$ is  the map $\wh f(v)=f(v\otK 1)$.
		
	\item  With the above notation, we have:
	\[
		\big(\psi_{X_1,X_2}^{-1}(f_1\ot f_2)\psi_{V_1,V_2}\big)^{\wh{\ }}=(X_1\otK X_2 \otK m_S)  (X_1\otK \mt_{X_2}\otK S)  (\wh f_1\otK \wh f_2).
	\]
		
	\item If the codomain of $f_1$ and the domain of $f_2$ coincides, then $(f_2  f_1)^{\wh{\ }}=(X_2\otK m_S)  (\wh f_2\otK S)  \wh f_1$.
		
	\item  Let $(V,\mt_V)$ be an object in $\TS$. If $U$ is a $\mt_V$-invariant subbimodule, i.e.  $\mt_V(S\otK U)\subseteq U\otK S$, then $(U,\mt_U)$ is an object in $\TS$, where $\mt_U$ is the restriction of $\mt_V$ to $S \otK {U}$.
		
	\item Let $f:\otp V\to X$ be a morphism of $S$-bimodules. Then, for any $U$ as in (4), we have  $\wh f(U)S=f(\otp U)$.
\end{enumerate}
\end{lemma}

\begin{proof}
	The first assertion immediately follows  by the fact that $(-)^{\wh{\ }}$ is the composition of the following sequence of isomorphisms:
\[
\Hom_{S-S}(V\ot_{\mt_V} S, X)=\Hom_{\mK-S}(V\ot_{\mt_V} S, X)^S\cong \Hom_{\mK-\mK}(V, X)^S=\Hom_{\mK-\mK}^{\mt_V}(V, X).
\]

	Let us prove (2). Since $ {f_1}$ and $f_2$ are morphisms of $S$-bimodules and the left $S$-module structure of  $\Psi(Y,\mt_{Y})$  is the $\mt_{Y}$-twisted action, we have:
	\[
	( {f_1}\ot {f_2})((v_1\otK 1)\ot  (v_2\otK 1))=\wh  {f_1}(v_1)\ot \wh {f_2}(v_2)=\sum_{j_1=1}^{p_1}\sum_{j_2=1}^{p_2} (x_1^{j_1}\otK 1)\ot [(x_2^{j_2})_{\mt_{X_2}}\otK (s_1^{j_1})_{\mt_{X_2}}s_2^{j_2}],
	\]
	where $\wh  {f_i}(v_i)=\sum_{j_i=1}^{p_i} x_i^{j_i}\otK s_i^{j_i}$, for certain $x_i^{j_i}\in X_i$ and $s_i^{j_i} \in S$. Thus
	\begin{align*}
	( 	\big(\psi_{X_1,X_2}^{-1}(f_1\ot f_2)\psi_{V_1,V_2}\big)^{\wh{\ }}(v_1\otK v_2)
	& =\psi_{X_1,X_2}^{-1}\big({f_1}(v_1\otK1)\ot {f_2}(v_2\otK 1)\big)\\
	& =\sum_{j_1=1}^{p_1}\sum_{j_2=1}^{p_2} x_1^{j_1}\otK (x_2^{j_2})_{\mt_{X_2}}\otK (s_1^{j_1})_{\mt_{X_2}}s_2^{j_2}.
	\end{align*}
	In a similar way we can prove the relation:
	\[
	\big[(X_1\otK  X_2\otK m_S)  (V_1\otK \mt_{X_2}\otK S)  (\wh f_1 \otK \wh {f_2})\big](v_1\otK v_2)=\sum_{j_1=1}^{p_1}\sum_{j_2=1}^{p_2} x_1^{j_1}\otK (x_2^{j_2})_{\mt_{X_2}}\otK (s_1^{j_1})_{\mt_{X_2}}s_2^{j_2},
	\]	
	so the proof of (2) is complete now.
	
	The identity of (3) is a precisely \eqref{eq:f_hat}. The fourth part of the lemma is obvious. Let us prove (5). For any $v\in U$ and $s\in S$, we have  $f(v\otK s)=\wh f(v)\cdot s$. Henceforth, $f(\otp U)\subseteq \wh f(U)S$. To conclude the proof we remark that
	$f(\otp U)$ is an $S$-subbimodule of $X$ which contains $\wh f(U)$.
\end{proof}

\begin{remark}\label{re:symmetric}
	The notions that we introduced in the previous subsection admit a left-right symmetric version. More precisely, instead of working with the category $\TS $ we can use the category $\ST $ whose objects are couples $(V, \mt'_V)$, with $V$ a $\mK$-bimodule and $\mt'_V:V\otK S\to S\otK V$ a $\mK$-bimodule map which satisfies the appropriate left-right symmetric versions of relations \eqref{eq:t_W}. Doing this change, we have also to replace the functor $\Psi$ with the functor $\Psi':\ST \to\SMS$, defined by $(V,\mt'_V)\mapsto S\ot_{\mt'_V} V$ and $f\mapsto S\otK f$. On $S\ot V$ one takes  the left regular and the right $\mt'_V$-twisted actions:
	\[
	s'\cdot(s\otK v)=s's\otK v,\qquad (s'\otK v)\cdot s=\sum s's_{\mt'_V} \otK v_{\mt'_V}.
	\]
	The map $\mt'_V$ is $S$-linear, with respect to the regular action on $V\otK S$ and $\mt'_V$-twisted  action on $S\otK V$.
\end{remark}

\begin{fact}[Twisted tensor products.]
	Twisted tensor products can be defined in several different ways, see for instance \cite{JPS} and the references therein. Here we suggest a new approach, which is based on the fact that $\TS$ is a monoidal category, so one may speak about algebras (or monoids) in $\TS$.
	
	As before we fix a ring $\mK$ and a $\mK$-ring $S$. By definition, a monoid in $\TS$ is an object $(A,\mt)$ together with a multiplication  $m_A: (A,\mt)\ot (A,\mt)\to (A,\mt)$ and a unit $u_A:(\mK, \mt_\mK)\to (A,\mt)$  that are morphisms in $\TS$ satisfying the \textit{associative}  and \textit{identity} axioms. In particular they are $\mK$-bilinear maps. As usual, we denote $u_A(1)$ by $1$. The couple  $(A,\mt)$ is an object in the category if and only if the relations \eqref{twist2} below hold. Moreover, we can show that $m_A$ and $u_A$ are morphisms in the category if and only if the equations  \eqref{twist1} are true as well. Of course the latter relations mean that  $(S,\mt)$ is an object of $\AT$.
	\begin{align}
	& \mt   (m_{S}\otK  A )=( A \otK m_{S})  (\mt \otK S)  (S\otK \mt )  \quad\text{and}\quad \mt (1\otK a)=a\otK 1 \label{twist2},\\
	& \mt   (S\otK m_{ A }) =(m_{ A }\otK S)  ( A \otK \mt )  (\mt \otK  A )
	\quad\text{and} \quad \mt (s\otK 1)=1\otK s.  \label{twist1}
	\end{align}
	Clearly, the multiplication of  $(A,\mt)$ is associative and $u_A$ satisfies the axiom of the unit if and only if $(A,m_A,u_A)$ is a $\mK$-ring.
	
	All in all, a monoid in $\TS$ consists in a $\mK$-ring $(A,m_A,u_A)$ together with a morphism of $\mK$-bimodules $\mt:S\otK A\to A\otK S$ such that $(A,\mt)$ and $(S,\mt)$  are objects in $\TS$  and  $\AT$, respectively. In this case, we say that $\mt$ is a \textit{twisting map}.  Taking into account the properties of the objects in the above mentioned categories, it follows that $\mt$ is left $S$-linear and right $A$-linear.
	
	Since $(A,\mt)$ is a monoid in $\TS$ and $\Psi$ is a monoidal functor, it follows that $\bts:=\Psi(A,\mt)$ is a monoid in $\SMS$, called the \textit{twisted tensor product} of $ A $ and $S$. The $S$-bimodule structure of $\bts$ is given by the left $\mt$-twisted and the right regular $S$-actions on $A\otK S$.	
	
	Up to the identifications $\mK \otK S\cong S$ and $\Psi(A,\mt)\ot \Psi(A,\mt)\cong \Psi(A\ot A,\mt_{A\otK A})$, the unit and the multiplication of  $\bts$ are the $S$-bimodule maps $u:=\Psi(u_A)$ and $m:=\Psi(m_A)$, respectively. Thus
\[
	(a\otK s)\cdot_\mt (a'\otK s')=\sum aa'_{\mt}\otK s_{\mt}s'
\]
	and  $u(s)=1\otK s$, for all $s\in S$. It is easy to see that $u$ is a morphism of rings.
	
	In the case when $\mt:S\otK A\to A\otK S$ is an invertible  twisting map  and $\mt'$ is the inverse of $\mt$, one proves that $\mt'$ is a twisting map too. Hence $(A,\mt')$ and $(S,\mt')$ are objects in $\ST$ and $\TA$, respectively, and the twisted tensor product $S\ot_{\mt'}  A $ makes sense. As a consequence, we deduce that $\mt'$ is right $S$-linear and left $A$-linear. In particular, both left and right $\mt$-twisted $S$-actions are isomorphic with the regular ones. By symmetry, a similar results holds for the $A$-module structures. Thus, if $ A $ is a projective (flat) left $\mK$-module, then the left $\mt$-twisted $S$-action on $ A \otK S$ is projective (flat) as well. Similarly, if $ A $ is a projective (flat) right $\mK$-module then the right $\mt'$-twisted $S$-action on $S\otK  A $ is projective (flat).
	
	If $ A :=\oplus_{n\in\N}  A ^n$ is a graded connected $\mK$-ring and the  twisting map  $\mt$  is compatible with the grading, in the sense that  all components $A ^n$ are $\mt$-invariant subbimodules, then $ A \ott S$ is graded with respect to the decomposition $ A \ott S=\oplus_{n\in\N} (A ^n\otK S)$. Note that the $0$-degree component can be identified with $S$, so $\bts$ is a graded connected $S$-ring. The restriction of $\mt$ to $S\otK  A ^n$ will be denoted by $\mt^n$.
	
	If $ A $ is strongly graded, then $\bts$ is also strongly graded. For more details on (graded) twisted tensor products the reader is referred to \cite[Section 4]{JPS}.
\end{fact}

\begin{fact}[Examples of (graded) twisted tensor products.]\label{fa:example_twisted}
	By \cite[Proposition 4.9]{JPS}, we know that for every object $(V,\mt_V)$ in $\TS$ there exists a graded twisting map $\mt^*: S\otK T_\mK(V)\to T_\mK(V)\otK S$, which lifts $\mt_V$. Moreover, if $\mt_V$ is invertible, then $\mt^*$ is invertible as well.
	
	We can explain the construction of $\mt^*$ very easy using the fact $\TS$ is a monoidal category. Indeed, we take $\mt^0$ and $\mt^1$ to be $\id_\mK$ and $\mt_V$, respectively. On the other hand, for $n\geq 2$, the $n$-th tensor power $(V,\mt_V)^{\otK n}$ in $\TS$ is a couple $(V^{\otK n},\mt_{V^{\otK n}})$, where $\mt_{V^{\otK n}}:S\otK V^{\otK n}\to V^{\otK n}\otK S$ satisfies the relations \eqref{eq:t_W}. We set $\mt^n:=\mt_{V^{\otK n}}$.   The component $\mt^n$ maps a tensor monomial in $S\ot_\mK V^{\ot_\mK n}$ to the element in $V^{\ot_\mK n}\ot_\mK S$ obtained using repeatedly $\mt_V$ to move the factor in $S$ until it reaches the most-right position.		It is not difficult to see that  $\mt^*$ is a graded twisting map, so the graded connected $S$-ring $T_\mK(V)\ot_{\mt^*} S$ exists.

	Since $(V^{\otK n},\mt^n)$ is an object in $\TS$, it follows that $V^{\otK n}\otK S$ is an $S$-bimodule with respect to the left $\mt^n$-twisted action and the right regular action.
	
	Let $\ol V:=V\otK S$ and $\ol T:=T_S(V\otK S)$. We will also use the notation  $T:=T_\mK(V)$. Clearly, the $\mK$-bilinear map  $\psi^n: T^n \otK S\to \ol T^n$ defined by
	\[
	\psi^n(v_1\otK\cdots\otK v_n\otK s)=(v_1 \otK 1_S)\ot \cdots \ot (v_{n-1} \otK 1_S)\ot (v_n \otK s)
	\]
	is an isomorphism. 	The family $\{\psi^n_T\}_{n\in\N} $ defines an isomorphism of graded $S$-rings $\psi_T:T\ot_{\mt^*} S\to \ol T$.
	
	We now assume that $R$ is a $\mt^*$-invariant graded subbimodule of $T$, that is $\mt^n(S\otK R^n)= {R^n\otK S}$, for all $n$. Hence, $V^{\otK p}\otK R^q\otK V^{\otK r}$ is $\mt^n$-invariant for all $p,q$ and $r$ with $p+q+r=n$.
	
	Let $I$ be the the ideal generated by $R$ and let $ A :=T/I$. It follows that $I$ is $\mt^*$-invariant, so  $\mt^*$ factorizes through a  graded morphism $\mt:S\otK  A \to  A \otK S$ of $\mK$-bimodules. It is not difficult to see that $\mt$ is a graded twisting map. If, in addition, $S$ is a flat as a left $\mK$-module, then we have the following identifications of graded $S$-rings:
	\[
	A \ot_{ \mt} S= \frac{T}{I}\ot_{ \mt} S \cong \frac{T\ot_{\mt^*} S}{I \otK S} \cong \frac{\psi_T(T \otK S)}{\psi_T(I \otK S)}= \frac{\ol T}{\ig{\psi_T(R \otK S)}}.
	\]
	
	Let $\ol A:=\bts$. Note that, if $R$ is a bimodule of relations for $A$, then $\ol R:=\psi_T (R\otK S)$ is an $S$-bimodule of relations for $\ol A$. Indeed, if $\ol I= \psi_T(I \otK S)$, then $\ol I:=\ig{\ol R}$ and $\psi_T(\tilde{I}\otK S)=\widetilde{\ol I}$. Clearly $\ol A\cong \ol T / \ol I$.

\end{fact}

\begin{fact}[The relationship between $\tor_*^A(\mK,\mK)$ and $\tor_*^{\ol A}(S,S)$.]\label{fa:bar_ttp}
	We suppose that $\mt: S\otK  A \to  A \otK S$ is a graded twisting map. Our goal is to investigate the homogeneous components of $\Tor_3^{{\ol A} }(S,S)$ with respect to the internal grading.
	
	Let us assume that $S$ is a flat left $\mK$-module. We choose a projective graded resolution $(W_*\otK A,d_*)$ of $\mK$  as a right $A$-module. Since $S$ is flat, the complex $(W_*\otK A\otK S,d_*\otK S)$ is acyclic and its homology in degree zero coincides with $A$. For every $n$, we identify $W_n\otK A\otK S$ and $W_n\otK \ol A$ as right $\ol A$-modules.
	
	We claim that the above complex yields us a projective resolution of $S$ as a right $\ol A$-module. We have to prove that $d_*\otK S$ is right $\ol A$-linear. For $w\in W_n$, $a,a'\in A$ and $s,s'\in S$ we get
\[
	(d_n\otK S)\big((w\otK a\otK s)\cdot (a'\otK s')\big)=\sum d_n(w\otK aa'_\mt)\otK s_\mt s'=\sum_{i=1}^p w_i\otK a_i aa'_\mt\otK s_\mt s',
\]
	where $d_n(w\otK 1)= \sum_{i=1}^p w_i\otK a_i$. By a  similar computation we obtain
\[
		\big((d_n\otK S)(w\otK a\otK s)\big)\cdot (a'\otK s')=\sum_{i=1}^p	(w_i\otK a_i a\otK s)\cdot (a'\otK s')=\sum_{i=1}^p w_i\otK a_i aa'_\mt\otK s_\mt s'.
\]
	By comparing the right-most terms of the above sequences of identities we conclude that $d_*\otK S$ is right $\ol A$-linear.  Clearly, these maps respect the grading on $W_*\otK A\otK S$, since $d_*$ do so.
	
	To compute $\tor_n(\mK,\mK)$ we apply the functor $(-)\ot_A \mK$ to the resolution $(W\otK A, d)$. If $(W_*,d_*')$ is the resulting complex, then $d_n'(w)=\sum_{i=1}^p w_ia_i^0$, where we have written $d_n(w\otK 1)$ as a sum of tensor monomials as above, and $a_i^0$ denotes the component of degree $0$ of $a_i$.
	
	To compute $\tor_n^{\ol A}(S,S)$  we now apply the functor $(-)\ot_{\ol A} S$ to the resolution $(W_*\otK \ol A, d_*\otK S)$. It is not difficult to see that the outcome complex is precisely $(W_*\otK S,d_*'\otK S)$.
	
	Therefore, since $S$ is flat as a left $\mK$-module, we get:
	\[
	\tor_{n,m}^{\ol A} (S,S)\cong \tor_{n,m}^ A (\mK,\mK)\otK S.
	\]
	In particular, if $\tor_{n}^ A (\mK,\mK)$ is $m$-pure, then $	\tor_{n}^{\ol A} (S,S)$ is $m$-pure as well.
\end{fact}

\begin{fact}[Notation and assumptions.]\label{fa:notation_WW}
	From now on, ${\ol A} := A \ot_{\mt} S$ is a graded twisted tensor product as in \S\ref{fa:example_twisted}. Thus, $(V,\mt_V)$ is an object in $\TS$ such that $\mt_V$ is graded. Let $\mt^*:S\otK T\to T\otK S$ denote the lifting of $\mt_V$, where $T:=T_\mK(V)$. We fix a $\mt^*$-invariant  $\mK$-bimodule of relations $R$ for  $A$. Thus, $A=T/I$,  where $I=\ig{R}$.  Let $\mt: S\otK A\to A\otK S$ denote the induced twisting map. By notation, $\ol V:=V\otK S$ and  $\ol T:=T_S(\ol V)$. Recall that there exists an isomorphism $\psi_T:T\ot_{\mt^*} S\to  \ol T$ of graded connected $S$-rings. Let $\ol R:=\psi_T(R\otK S)$ and $\ol I:=\psi_T(I\otK S)$. We have seen that $\ol R$ is an $S$-bimodule of relations for $\ol A$. Furthermore, $\ol A\simeq \ol T/ \ol I$.
	
	Finally, let us  assume that $S$ is a flat left $\mK$-module, that $A^1$ and $R$ are projective as right $\mK$-modules, $R$ is $n$-pure and $\tor_3^A(\mK,\mK)$ is $n+1$-pure.  By definition, it follows that ${\ol A}^1$ and ${\ol R}$ are projective as right $S$-modules. Since $\psi_T : T\otK S\to \ol T $ is a morphism of graded $S$-bimodules, $\ol{R}$ is $n$-pure. By \S\ref{fa:bar_ttp}, the graded module $\tor_3^{\ol A}(S,S)$ is $n+1$-pure.
	
	We fix an $S$-subbimodule $\ol P$ of $\ol{T}$ such that $\ol{R}= p_{\ol T}^n(\ol P)$, i.e. as in Example \ref{ex:BG}, where $p_{\ol T}^n$ stands for the projection from $\ol T$ to $\ol T^n$. Let $P:=\psi^{-1}_T(\ol P)$ and let $\psi_P$ denote the restriction of $\psi_T$ to $P$.
	
	In view of the right-hand side version of Theorem \ref{te:PBW-pure}, we deduce the following.
\end{fact}

\begin{theorem}\label{te:PBW_twisted}
	We keep the notation and assumptions from \S\ref{fa:notation_WW}. The $S$-ring  $\ol U:=\ol{T}/\ig{\ol P}$ is a $PBW$-deformation of ${\ol A}$, if and only if there is a morphism of filtered bimodules $\alpha:\ol{R}\to \ol{T}$, which satisfies the Jacobi conditions \eqref{jpr0}--\eqref{jprn}  and the relation $\ol P:=\ol{\alpha}(\ol{R})$.
\end{theorem}

\begin{remark}
	Specializing the above result to the case when $S$ is a field and $A$ is a Koszul algebra we get \cite[Theorem 4.6.1]{HvOZ}. Indeed, $A$ is Koszul if and only if $\tor_n^A(S,S)$ is concentrated in internal degree $n$, see \cite[Proposition 4.3]{RS}. In particular,  $A$ is quadratic ($A:=T/\ig{R}$ for some $R\subseteq T^2$), because  $\tor_2^A(S,S)$ is concentrated in internal degree $2$, cf. \cite[Theorem 2.3.2]{BGS}. Moreover, since $S$ is a field, there exists a minimal resolution $A\ot V_*$ of $S$, which can be used to compute the homological complexity of $A$:
\[
	c(A)=c_{V_3}=\sup\{n-1\mid V_3^n\neq 0\}=\sup\{n-1\mid \tor_{3,n}^A(S,S)\neq 0\}=2.
\]
	The above cited result now follows by Theorem \ref{te:PBW_twisted}. For smash products of Koszul algebras by finite dimensional Hopf algebras,  Theorem \ref{te:PBW_twisted} coincides with \cite[Theorem 3.1]{WW}.
\end{remark}

\begin{fact}[$PBW$-deformations of a twisted tensor product.] \label{fa:PBW_twisted}
	Note that we always identify $T\otK S$ and $\Psi(T, \mt^*)$ as $S$-bimodules. Similarly, $R\otK S=\Psi(R,\mt_R)$, where $\mt_R$ is the restriction of $\mt^*$ to $R\otK S$. Hence, given $\alpha\in \Hom_{S-S}(\ol R, \ol T)$,  the morphism $\alpha':=\psi ^{-1}_T  \alpha \psi_R$ is a morphism of $S$-bimodules from $\otp R$ to $T\ot_{ \mt^*} S$.   In view of the isomorphism from Lemma \ref{le:Psi} (1), to $\alpha'\in	\Hom_{S-S}(\otp R, \otp T)$ it corresponds a unique $\mt_R$-invariant $\mK$-bilinear map from $R$ to $T\otK S$, namely $\alpha_R:={\wh{\alpha'}}$.
	
	Our  goal is to reformulate in terms of $\alpha_R$ the properties of $\alpha$ which encode the fact that $\ol U$ is a $PBW$-deformation of ${\ol A} $, as in Theorem \ref{te:PBW_twisted}. In order to do that, we fix some $\alpha\in \Hom_{S-S}(\ol R, \ol T)$ and we denote by $\alpha_R$ the corresponding map in $\Hom_{\mK,\mK}^{\mt_R}(R,T\ot S)$. To clarify the notation we consider the diagram from Figure \ref{fig_2}.
\begin{figure}[ht]
\[\xymatrix@C=30pt @R=30pt {
	&P \ar[r]^{\psi_P}_\sim\ar@{_(->}[d] &\ol P \ar@{_(->}[d]\\
	& T\ot_{ \mt^*} S   \ar[r]^{\psi_T}_\sim&\ol T\\
	R \ar[ur]^{\alpha_R}\ar[r] & \otp R \ar[r]_-{\psi_R}^-\sim \ar[u]_{\alpha'}&\ol R\ar[u]_\alpha}
\]
\caption{}\label{fig_2}
\end{figure}

	Since $\psi_T$ is an isomorphism of graded bimodules, $\alpha$ is a map of filtered $S$-bimodules such that $\alpha_0(x)=x$ for all $x\in \ol R$, if and only if $\alpha_R$ is a map of filtered $\mK$-bimodules and $\alpha_0'(r\otK 1)=r\otK 1$, for all $r\in R$. Here, on  $T\otK S$ we take the filtration $F^p(T\otK S)=T^{\leq p}\otK S$. If the components of $\alpha$ are $\{\alpha_i\}_{i\leq n}$, with $\alpha_i:\ol R\to \ol T^{n-i}$, then the components of $\alpha_R$ are $\alpha_{R,i}:R\to T^{n-i}\otK S$, where $\alpha_{R,i}=(\psi ^{-1}_T \ol \alpha_i \psi_R )^{\wh{\ }}=\wh{\alpha'_i}$.
	
	We note that, by construction,  $\alpha_R$ is the unique $\mt_R$-invariant map such that $\alpha' (r\otK s)=\alpha_R(r)s$. Obviously, $\alpha_R$ is $\mt_R$-invariant if and only if its components $\{\alpha_{R,i}\}_{i\leq n}$ are so.
	
	Furthermore, the relation $\alpha(\ol R)=\ol P$ holds if and only if  $\alpha'(R\otK S)=P$. By Lemma \ref{le:Psi} (5), it follows that $\alpha(\ol R)=\ol P$ if and only if $\alpha_R(R)S=P$.
	
	For translating  the  Jacobi conditions \eqref{jpr0}--\eqref{jprn} we need some more notation. First, we set
	\[\textstyle
	X:=\left(R\otK V\right)\bigcap (V\otK R)\subseteq T,\qquad \ol X:=(\ol R\ot \ol V)\bigcap (\ol V\ot \ol R)\subseteq \ol T.
	\]
	Moreover, $\psi_T(R\otK V\otK S)= \ol R\ot \ol V$ and $\psi_T(V\otK R\otK S)=\ol V\ot \ol R$. Since $\psi$ is bijective and $S$ is flat as a left $\mK$-module, we deduce that $\psi_T(X\otK S)=\ol X$. On the other hand, $X$ is $\mt^*$-invariant. Thus, we can define $\mt_X$ as the restriction of $\mt^*$ to $S\otK X$ and $(X,\mt_X)$ is an object in $\TS$.
	
	If $\ol \vartheta: \ol X\to \ol T$ denotes the map $ \vartheta:=\ol V\ot \alpha-  \alpha \ot\ol V$, then we can define $\vartheta'$ and $\vartheta_X$ by  $\vartheta':=\psi^{-1}_T \vartheta \psi_X$ and $\vartheta_X:=\wh{\vartheta'}$, as in the following diagram.
\[\xymatrix@C=30pt @R=30pt {
	&  T\ot_{\mt^*} S   \ar[r]^-{\psi_T}_-\sim&\ol T\\
	X \ar[ur]^{\vartheta_X}\ar[r] & \otp X \ar[r]_-{\psi_X}^-\sim \ar[u]_{\vartheta'}&\ol X\ar[u]_\vartheta}
\]
	Note that $\vartheta_X$ is a morphism of filtered bimodules. To compute its components  we use the following diagram in the category $\SMS$, where we take $\varphi_i$ to be the unique $S$-bilinear map which makes the most-left square commutative.
\[\xymatrix@C=40pt @R=30pt {
	&(R\otK V)\otK \ar[r]_-{\psi_{R,V}}\ar@{.>}[d]_{\varphi_i}\ar@/^2.0pc/@[][rr]^{\psi_T^{n+1}\restr{(R\otK V)\otK S}}S   & (R\otK S)\ot \ol V  \ar[r]_-{\psi_R\ot\ol V}\ar[d]_{\alpha'_i\ot\ol V}& \ol R\ot \ol V \ar[d]^{\alpha_i\ot\ol V}\\
	&(T^{n-i}\otK V)\otK S \ar[r]^-{\psi_{T^{n-i},V}}\ar@/^-2.0pc/@[][rr]_{\psi_T^{n-i+1}}  & (T^{n-i}\otK S)\ot \ol V  \ar[r]^-{\psi_{T^{n-i}}\ot\ol V} & \ol T^{n-i}\ot \ol V  }
\]
	The most-right square is commutative, by definition of $\alpha'$. On the other hand, by a straightforward computation, we can  prove that both `triangles' are commutative as well. By Lemma \ref{le:Psi} (2) we get
	\begin{align*}
	\wh{\varphi_i}& = \big[\psi_{T^{n-i},V}^{-1}(\alpha'_i\ot \ol V)\psi_{R,V}\big]^{\wh{\ }}=(T^{n-i}\otK V\otK m_S)(T^{n-i}\otK\mt_{V}\otK S)(\wh{\alpha'_i}\otK \wh{\ol V})\\
	&= (T^{n-i}\otK\mt_{V})( \alpha_{R,i}\otK V).
	\end{align*}
	Note that for the ultimate equation we used the fact that $\wh{\ol V}:V\to \ol V$ maps $v$ to $v\otK 1$. For $x\in X\subseteq T^{n+1}$ and $s\in S$, we deduce the relation:
	\[
	\big[(\psi_{T}^{n-i+1})^{-1}(\alpha_i\ot \ol V)\psi_{X}\big](x\otK s)=\varphi_i(x\otK s)= \wh \varphi_i(x)s=[(T^{n-i}\otK\mt_{V})( \alpha_{R,i}\otK V)](x)s.
	\]
	For $x$ and $s$ as above, proceeding in a similar  way,  we show that the equation below  holds as well.
	\[
	\big[(\psi_{T}^{n-i+1})^{-1}(\ol V\ot \alpha_i)\psi_{X}\big](x\otK s)=(V \otK  \alpha_{R,i})(x)s.
	\]	
	Thus $\vartheta_{X,i}: X \to T^{n+1}\otK S$, the component of degree $-i$ of $\vartheta_X$, satisfies the following equality  of functions defined on $X$
	\[
	\vartheta_{X,i}= V\otK\alpha_{R,i}-(T^{n-i}\otK \mt_V)( \alpha_{R,i}\otK V).
	\]
	Obviously, the inclusion $\vartheta_1(\ol X)\subseteq \ol R$ holds if and only if $\vartheta_{X,1}(X)\subseteq R\otK S$. In conclusion, $\ol \alpha_1$ satisfies \eqref{jpr0} if and only if $\alpha_{R,1}:R\to T^{n-1}\ot S$ verifies the condition:
	\begin{equation*}
	\big[V\otK \alpha_{R,1}-(T\otK \mt_V) (\alpha_{R,1}\otK V)\big](X)\subseteq R\otK S. \tag{$\cJ_0''$}\label{jsc0}
	\end{equation*}
	Assuming that \eqref{jpr0} holds, $\vartheta_1$  can be viewed as a map from $\ol X$ to $\ol R$. Therefore, \eqref{jpri} holds for some $i<n$ if and only if $[((\psi_T^{n-i+1})^{-1}\alpha_i\psi_R)(\psi_R^{-1}\vartheta_1\psi_X)]^{\,\wh{ }}= -[(\psi_T^{n-i+1})^{-1}\vartheta_{i+1}\psi_X]^{\,\wh{\ }}$. Furthermore, the letter identity is equivalent to $ (\alpha'_i\vartheta'_1)^{\wh{\ }}=-\vartheta_{X,i+1}$. Applying Lemma \ref{le:Psi} (3), we conclude that \eqref{jpri} is equivalent to the following identity  of functions defined on $X$:
	\begin{equation*}
	(T\underset{ \mK}{\otimes} m_S) ( \alpha_{R,i}\underset{ \mK}{\otimes} S)[ V\underset{ \mK}{\otimes}\alpha_{R,1} -(T^{n-1}\underset{ \mK}{\otimes}\mt_V)  (\alpha_{R,1}\underset{ \mK}{\otimes} V)]=-[V\underset{ \mK}{\otimes}\alpha_{R,i+1} -(T^{n-i-1}\underset{ \mK}{\otimes} \mt_V)  (\alpha_{R,i+1}\underset{ \mK}{\otimes} V)]. \tag{$\cJ''_i$\label{jsci}}
	\end{equation*} 	
	Proceeding in a similar way we can show that \eqref{jprn} is equivalent to the following identity of functions defined on $X$:
	\begin{equation*}
	(T\otK m_S) ( \alpha_{R,n}\otK S)[ V\otK\alpha_{R,1} -(T^{n-1}\otK \mt_V)  (\alpha_{R,1}\otK V)]=0.\tag{$\cJ''_n$}\label{jscn}
	\end{equation*}	
\end{fact}
As a direct application of Theorem \ref{te:PBW-pure} we obtain the following result.
\begin{theorem}\label{te:PBW-winterspoon}
	Keeping the notation and assumptions from \ref{fa:notation_WW}, the $S$-ring $\ol U$ is a $PBW$-deformation of $\ol A$ if and only if there exists a morphism of filtered $\mK$-bimodules $\alpha_R:R\to T\otK S$ such that $\alpha_R(R)S=P$ and either $\Tor_3^A(\mK,\mK)=0$ or \eqref{jsc0}--\eqref{jscn} holds.
\end{theorem}

A similar result was proved for certain smash products in \cite[Theorem 0.4]{WW}. It can be seen as  a particular case of Theorem \ref{te:PBW-winterspoon}, since any smash product as in the aforementioned paper is a twisted tensor product with invertible twisting map.

Indeed, let $H$ be  a Hopf algebra over a field $\mk$. We assume that $H$ acts on a graded $\mk$-algebra such that the components of $A$ are $H$-submodules. Hence one can define \textit{the smash product}, which is a graded algebra structure on $A\otk H$ with unit $1\ot 1$. The multiplication is given by:
\[
	(a\otk h)\cdot (a'\otk h')=\sum a(h_1\cdot a')\otk h_2h',
\]
where $\Delta(h)=\sum h_1\otk h_2$ is the Sweedler notation  for the comultiplication of $H$. The smash product is denoted by $A\# H$. For more details on the graded smash products the reader is referred to \cite[\S1.1]{WW}.

Let  $\mt:H\otk A\to A\otk H$ denote the map $h\otk a\mapsto \sum (h_1\cdot a)\otk h_2$. It is easy to see that $\mt$ is a graded twisting map such that $A\# H\cong A\ott H$. Moreover, if  the antipode  $\boldsymbol S$ of $H$ is bijective, then $\mt$ is invertible and its inverse $\mt^{-1}$ is defined by $a\otk h\mapsto\sum h_2\otk\boldsymbol S^{-1}(h_1)\cdot a$. In conclusion, we are able to apply Theorem \ref{te:PBW-winterspoon} for the smash products $A\# H$ as in \cite{WW}.

\subsection{The Poincar\'e-Birkhoff-Witt Theorem for special classes of rings.}\hfill\vspace*{1.3ex}

	In the previous applications we  imposed almost no restriction on the base ring $S$. Instead, we asked that the graded $S$-ring $A$ to fulfill certain conditions.  In the next applications we will restrict ourselves to the case when $S$ belongs to some special classes of rings. In this way some of the restrictions that we imposed to $A$ can be relaxed.
\begin{fact}[Rings of weak dimension $0$.]	
	Let $S$ be an algebra over a field $\mk$.  When  we work with $\mk$-algebras, by an $S$-bimodule we will always mean a left module over $S^e=S\otk S^{op}$, the enveloping algebra of $S$.
	
	We now consider the case when the weak dimension of $S$ is zero,  that is $S$ is $S^e$-flat. As it is shown in \cite{Ha}, there is a strong relationship between the $\mk$-algebras of weak dimension zero and  \textit{von Neumann regular algebras}. Recall that, by definition, $S$ is regular von Neumann if and only if the equation $sxs=s$ has a solution, for every $s\in S$. It is well-known that $S$ is regular von Neumann if and only if the weak global dimension of $S$ is zero (all left $S$-modules  are flat or, equivalently, all right $S$-modules are flat). For simplicity, we will call them just \textit{regular}.
	
	Turning back to algebras of weak dimension zero, by \cite[Theorem 1]{Ha}, the enveloping algebra of such an algebra is regular and conversely. Moreover, if $S^e$ is regular then $S$ is regular as well, since the weak global dimension of $S$ is less than or equal to the weak dimension of $S$, cf. \cite[Lemma 2]{Ha}.
	
	 Throughout the remaining part of this subsection we fix a $\mk$-algebra $S$ of weak dimension zero.  Thus, left $S^e$-module and left $S$-modules are all flat.
	
	It is well-known that a module over an arbitrary ring is projective and finitely generated if and only if it is flat and finitely presented.

Moreover, it is easy to see that a quotient of a finitely presented module (still over an arbitrary ring) by a finitely generated submodule is finitely presented. By the foregoing, for regular rings, a module is projective and finitely generated if and only if it is finitely presented. Thus, in this case, it follows that any quotient of a projective noetherian module is also projective noetherian. Consequently, every submodule of a projective noetherian module is always a direct summand and, as such, it is projective and noetherian.
	
	We can apply the above results both for $S$ and $S^e$, since they are regular in the present setting. If $M$ is a projective noetherian $S^e$-module and $N$ is an $S^e$-submodule, then  $M/N$ is $S^e$-projective and $N$ is a direct summand. Note that $N$ is noetherian (finitely generated) as an $S^e$-module, provided that $N$ is noetherian (finitely generated) as a left $S$-module. For simplicity, we will say that such an $N$ is left noetherian  (left finitely generated). 	A graded module over an arbitrary ring will be called \textit{componentwise  noetherian} if its homogeneous components are all noetherian. Note that a componentwise  noetherian module is not  noetherian in general, excepting the case when it has only a finite number of non trivial components.
	
	Analogously, if $M$ is a projective $S^e$-module, then we will say for short that  it is $S^e$-projective. On the other hand, $M$ is said to be left $S$-projective, if it is projective as a left $S$-module. Note that any projective $S^e$-module is both left and right projective. The converse does not hold in general. For example, $S$ is always left and right projective, but it is $S^e$-projective if and only of $S$ is separable.
	
	Let us now fix two projective $S^e$-modules  $X_i$, where $i=1,2$.  The tensor-hom adjunction formula \eqref{eq:adjunction} yields us the relation:
$$
	\Hom_{S,S}(X_1\ot X_2,-)\cong\Hom_{S,S}(X_2,\Hom_{S,-}(X_1,-)).
$$
	Since $X_2$ is a projective $S^e$-module and $X_1$ is right projective, we conclude that $X_1\ot X_2$ is  $S^e$-projective. Therefore, if $V$ is a $S^e$-projective, then so is $V^{\ot n}$ for $n>0$.
	
	On the other hand, if $X_1$ is  left noetherian  and $X_2$ is left finitely generated, then we claim that $X_1\ot X_2$ is left noetherian. Indeed,  $X_2$ is a quotient module of some  $S^n$. Thus $X_1\ot X_2$ is left noetherian, being a quotient of  $X_1^n$. Since the latter module is left noetherian we conclude the proof of our claim. In particular, if $V$ is a left noetherian $S^e$-module, then so is $V^{\ot n}$ for any $n>0$.\vspace{1ex}
	
	Let $A$ be a strongly graded connected $S$-ring, where $S$ is a regular algebra. By $\pi_A:T_A\to A$  we denote, as usual, the canonical $S$-ring morphism from the tensor $S$-ring $T_A$ of $A^1$ to $A$. Let $I_A$ denote the kernel of $\pi_A$.  We assume that $A^1$ is $S^e$-projective and left noetherian.
	
	By the foregoing remarks, $T_A^+$ is  projective and  componentwise noetherian (both as  $S$-module and  $S^e$-module). It follows that $A^+$ is projective and componentwise noetherian (also as $S$-module and  $S^e$-module), being a quotient of  $T_A^+$. Since $I_A\subseteq T_A^+$, it follows that $I_A$ is a projective and componentwise noetherian $S^e$-module. In conclusion, $\wt I_A=A^1I_A+I_AA^1$ is a direct summand of $I_A$ as an $S^e$-module. By Proposition \ref{pr:minimal} (1), if $R$ is a graded $S^e$-submodule complement of $\wt I_A$ in $I_A$, then $R$ is a bimodule of relations for $A$ and $R$ is $S^e$-projective. Thus, $R$ is projective as  left and right $S$-module.
	
	We claim that $S$ has a minimal resolution as a left $A$-module. We will prove by induction that we can construct an exact sequence:
	\begin{equation}\label{eq:La-minimal}
	A \ot V_n\xrightarrow{d_n} A \ot V_{n-1}\xrightarrow{d_{n-1}}\cdots  \xrightarrow{d_3} A \ot V_2\xrightarrow{d_2} A \ot A^1\xrightarrow{d_1} A  \xrightarrow{d_0}S\rightarrow 0,
	\end{equation}
	where all $V_i$ and $K_i:=\ker(d_i)$ are projective and componentwise noetherian, and $\ker(d_i)\subseteq A^+\ot V_i$, for all $i=1,\dots, n$.
	
	For $n=0$, we can take $d_0$ to be the projection onto $A^0$, since we have noticed that $K_0=A^+$ is left projective and componentwise noetherian. Let us assume that we have  constructed \eqref{eq:La-minimal} up to degree $n$.
	
	By induction hypothesis, $K_n$  is  projective and componentwise noetherian. It follows that $K_n/(A^+K_n)$ is projective, so $A^+K_n$  is a direct summand of $K_n$. Let $V_{n+1}$ be a graded submodule complement of $A^+K_n$ in $K_n$. Clearly $V_{n+1}$ is projective and componentwise noetherian. By Lemma \ref{le:minimal} there is a graded morphism $d_{n+1}:A\ot V_{n+1}\to A\ot V_n$ such that, by adding it to \eqref{eq:La-minimal}, we get an exact sequence and $\ker(d_{n+1})\subseteq A^+\ot V_{n+1}$.
	
	By Lemma \ref{le:minimal}, the kernel $K_{n+1}$ of $d_{n+1}$ is projective. It remains to prove that $K_{n+1}$ is componentwise noetherian.
	The left $S$-module $A^p\ot V^{m-p}_k$ is noetherian,  since $A^p$ and  $V_k^{m-p}$ are  left noetherian and  finitely generated, respectively. Thus   $(A\ot V_{n+1})^k=V_{n+1}^k\oplus[\oplus_{p=1}^k(A^p\ot V^{k-p}_{n+1})] $  is noetherian.  It follows that $A\ot V_{n+1}$ is componentwise noetherian. Since $K_{n+1}$ is a graded submodule of $A\ot V_{n+1}$, we conclude that it is componentwise noetherian.
	
	Clearly, \eqref{eq:La-minimal}  is the part in degree up to $n$ of  a minimal resolution of $S$. \vspace*{1ex}
	
	Let $P$ be an $S^e$-submodule of $T$ such that $R\subseteq R_P$. Obviously,   $R^{n}_P$ is $S^e$-projective, being a submodule of $T^n$ which is $S^e$-projective and left noetherian. Thus $\gr P\cong R_P$ is projective as well. We conclude that $P^{\leq n}$ is a direct summand of $P^{\leq n+1}$ as an $S^e$-module, for all $n$.  As a direct application of Theorem \ref{te:PBW}, we get the following result.	
\end{fact}

\begin{theorem}\label{te:wd0}
	Let $S$ be a $\mk$-algebra of weak dimension zero. Let $A$ be a strongly graded connected $S$-ring such that $A^1$ is left noetherian. Then there exists a bimodule of relations $R$ for $A$ which is $S^e$-projective and left noetherian. If $P$ denotes a  subbimodule of $T_A$ such that $R\subseteq R_P$, then $U(P)$ is a $PBW$-deformation of $A$ if and only if $R_P\subseteq \ig{R}$, there exists a morphism of filtered $S$-bimodules $\alpha:R\to T_A$ such that $\alpha(R)\subseteq P\subseteq \ig{\alpha(R)}$ and either $c(A)=-1$ or $\alpha(R)$ satisfies the Jacobi conditions $(\cJ_1)-(\cJ_{c(A)})$.
\end{theorem}

\begin{fact}[Rings of dimension $0$.]
By definition, the dimension of a $\mk$-algebra $S$ is zero if and only if $S$ is $S^e$-projective, that is $S$ is separable. 	By \cite[Theorem 9.2.11]{We}, a separable algebra is finite dimensional as a $\mk$-linear space. In particular, such an algebra is left noetherian. A separable algebra $S$ is semisimple too. Indeed, $S$ is flat as an $S^e$-module, so $S$ is regular and noetherian. By \cite[Theorem 4.2.2]{We}  it follows that $S$ is semisimple. Note also that any $S^e$-module is projective.

Let $A$ be a strongly graded connected $S$-ring $A$, where $S$ is a separable algebra over a field  $\mk$. In view of the above properties of separable algebras, it follows that there exists a bimodule of relations for $A$, say $R$, since $\wt I_A$ is  a direct summand of $I_A$ as an $S^e$-module. Obviously,  $A^1$ and $R$ are left projective, and $R_P$ is $S^e$-projective, for any $S^e$-module $P$. Therefore, as another application of Theorem  \ref{te:PBW}, we get the following.	
\end{fact}

\begin{theorem}\label{PBW_separable}
	Let $S$ be a separable $\mk$-algebra and let $A$ denote a strongly graded connected $S$-ring. There exists a bimodule of relations $R$ for $A$.  If $P$ is a subbimodule of $T_A$ such that $R\subseteq R_P$, then $U(P)$ is a $PBW$-deformation of $A$ if and only if $\ig{R}=\ig{R_P}$, there exists a morphism of filtered $S$-bimodules $\alpha:R\to T_A$ such that $\alpha(R)\subseteq P\subseteq \ig{\alpha(R)}$ and either $c(A)=-1$ or $\alpha(R)$ satisfies the Jacobi conditions $(\cJ_1)-(\cJ_{c(A)})$.
\end{theorem}

In the particular case when $S=\mk$ and the bimodule $P$ is as an Example \ref{ex:CS}, the above theorem is precisely \cite[Theorem 4.2]{CS}.

\begin{fact}[Multi-Koszul algebras.]\label{ex:herscovich}
	Multi-Koszul algebras were introduced in \cite{Herscovich} and generalize usual Koszul-algebras. Let us  briefly recall an important characterization of multi-Koszul algebras, which can be interpreted as their definition.
	
	Let $\Bbbk$ be a field. We fix  a finite-dimensional vector space  $V$ and we take $I$  to be a homogeneous ideal of the tensor algebra $T_\mk(V)$. Let us denote the quotient $T_\mk(V)/I$ by ${A}$. We choose a a bimodule of relations $R$ for $I$, which exists because we are working with algebras over a field. This is equivalent to require that $R$ is a space of relations of ${A} $ in the sense of \cite[\S3]{Herscovich}.
	
	We assume further that $R$ is finite-dimensional, i.e. there exists a finite set $\mathfrak{S}$ whose elements are integers $s\geq 2$, such that $R=\bigoplus_{s\in \mathfrak{S}}R^s$ and $R^s\subseteq V^{\otimes s}$.
	
	In view of \cite[Proposition 3.12]{Herscovich}, $A$ is multi-Koszul if and only if, for every $i\in \mathbb{N}$, we have
$$
	\Tor_{i}^{\A}(\Bbbk,\Bbbk)\cong \bigoplus_{s\in \mathfrak{S}}\bigcap\limits_{j=0}^{n_s(i)-s}V^{\otimes j}\otimes R^s\otimes V^{\otimes (n_s(i)-s-j)},
$$
	as graded vector spaces, where  $n_s(2j) := sj$ and  $n_s(2j + 1) := sj + 1$, cf. \cite[p. 202]{Herscovich}.
	
	In particular, if $A$ is multi-Koszul, for $s\in \mathfrak{S}$, we have
	\begin{equation*}\textstyle
		\Tor_{3,s+1}^{A}(\Bbbk,\Bbbk)\cong \left(R^s\otimes V\right) \bgi \left(V\otimes R^s\right)
	\end{equation*}
	while $	\Tor_{3,s+1}^{A}(\Bbbk,\Bbbk)=0$, for $s\notin \mathfrak{S}$. Because we are working with algebras over fields, the homological complexity of $A$ can be computed using the formula $c(A)=\sup\{n-1\mid 	\Tor_{3,n}^{A}(\Bbbk,\Bbbk)\neq 0\}$. Therefore, $c(\A)\leq N$,  where  $N:=\max\mathfrak{S}$.
	
	In particular, if $A$ is $N$-Koszul, i.e. $A$ is multi-Koszul and  $\mathfrak{S}=\{N\}$, then $c(A)\in\{-1,N\}$ and both values are possible. The $N$-Koszul algebras were introduced in  \cite{B2} and they generalize ordinary Koszul algebras, which coincides with $2$-Koszul algebras.
	
	For multi-Koszul algebras we get the following version of Poincar\'e-Birkhoff-Witt Theorem.
\end{fact}

\begin{theorem}\label{te:PBW_multi_Koszul}
	Let $A$ a be an $\mathfrak{S}$-multi-Koszul algebra over a field $\mk$ and let $R$ denote a  space (bimodule)  of  relations for $A$.  Let $P$ be a subspace of $T_\mk(A^1)$ such that $R\subseteq R_P$. The $\mk$-algebra $U(P)$ is a $PBW$-deformation of $A$ if and only if $\ig{R}=\ig{R_P}$, there exists a morphism of filtered linear spaces $\alpha:R\to T_\mk(A^1)$ such that $\alpha(R)\subseteq P\subseteq \ig{\alpha(R)}$ and either $c(A)=-1$ or $P'$ satisfies the Jacobi conditions $(\cJ_1)-(\cJ_{N})$, where $N=\max\mathfrak{S}$.
\end{theorem}
	
\begin{proof}
	We may assume that $c(A)\neq -1$. Recall that $c(A)\leq N$, so $P$ satisfies $(\cJ_1)-(\cJ_{c(A)})$  if $(\cJ_1)-(\cJ_{N})$ hold true. On the other hand, under the assumption that the Jacobi conditions $(\cJ_1)-(\cJ_{c(A)})$ are fulfilled, it follows that $(\cJ_n)$ is true for all $n$, see the proof of Theorem \ref{th:PBW}.
\end{proof}	
	
\subsection{Changing the presentation by generators and relations.}\hfill\vspace*{1.3ex}\label{fa:notation_tensor}

	Let $S$ be a separable $\Bbbk$-algebra. We fix a surjective graded homomorphism of  strongly graded connected $S$-rings $\varfun{\psi}{\CT}{T}$ so that its  component of degree $0$ is an isomorphism. We also assume that $\psi^1$ is bijective. Thus, if we let $\CK:=\ker(\psi)$, we get $\CK^0=\CK^1=0$.
	
	As in Section \ref{sec:1}, we fix an $S$-subbimodule $P$ of $T$ such that $P\bgi T^{\leq 1}=0$. We define  $R:=R_P$. Note that $R^0=R^1=0$. Let $I:=\ig{R}$ and $J=\ig{{P}}$. We set $\CI:=\psi^{-1}(I)$ and  $\CJ=\psi^{-1}(J)$. We now define the strongly graded connected $S$-rings: $A:=T/I$,  $\CA:=\CT/\CI$, $U=T/J$ and $\CU:=\CT/\CJ$. For the corresponding projection maps we will use the notation $\pi_A$, $\pi_\CA$, $\pi_U$ and $\pi_\CU$, respectively.
	
	Note that $\psi$ induces an isomorphism  of graded $S$-rings 	${\psi_A}:\CA\to A$ and an isomorphism of filtered $S$-rings  ${\psi_U}:\CU\to U$. By construction, $\psi_A\pi_\CA=\pi _A\psi$ and $\psi_U\pi_\CU=\pi _U\psi$.  Also by construction and the above assumptions, the ideals $I$ and $\CI$ are trivial in degree $0$ and $1$.
	
	Since $S$ is separable, there is a graded $S$-bilinear section $\varsigma:T\to \CT$. By the same hypothesis, we can pick a bimodule of relations $\CK_0$ of $T=\CT/\CK$. Now we can define $\CR:=\varsigma(R)+ \CK_0$ and $\CP:=\varsigma(P)+\CK_0$. Let us prove that $\CR$ and $\CP$ generate the ideals $\CI$ and $\CJ$, respectively. Indeed, since $\varsigma$ is a section of $\psi$, for every subbimodule $X\subseteq T$ we have the relation:
	\begin{equation}\label{eq:psiinv}
	\psi^{-1}(X)={\varsigma}(X)+ \CK
	\end{equation}
	and the sum is direct, that is ${\varsigma}(X)\bgi  \CK=0$. On the other hand, because $R$ and $\CK_0$ generate $I$ and $\CK$, using \eqref{eq:psiinv} we get:
	\[
	\CI=\psi^{-1}(I)=\psi^{-1}(\ig{R})=\ig{\psi^{-1}(R)}=\ig{{\varsigma}(R)+\CK}=\ig{{\varsigma}(R)+\CK_0}.
	\]	
	The fact that $\CP$ generates $\CJ$ can be proved in a similar way.
	
	By Lemma \ref{lem:grsum}, we have that $R_{\varsigma(P)}=\varsigma (R_P)$. Since $R_{\CK_0}=\CK_0$, we get $R_{\varsigma(P)}\cap R_{\CK_0}\subseteq\im(\varsigma)\cap \CK=0.$ By the same lemma we deduce that $R_{\varsigma(P)}+ R_{\CK_0}=R_{\varsigma(P)+\CK_0}$ which means $\CR=R_\CP$.
	%If  $p^n_\CT :\CT\to \CT^n$ denotes the projection, then $\CR^n:=p ^n_\CT(\CP^{\leq n})$.
	%The relation follows by the computation:
	%\[
	%	\CR^n= \varsigma^n(R^n)+ \CK_0^n =\varsigma^np_T^n(P^{\leq n})+ \CK_0^n = p_\CT^n\varsigma(P^{\leq n})+\CK_0^n =  p_\CT^n(\varsigma(P)^{\leq n}+\CK_0^{\leq n})=p_\CT^n(\CP^{\leq n}).
	%\]
	%	In the above sequence of relations, the first three equalities are obvious. The fourth equation follows by the identity $\varsigma(P^{\leq n})=\varsigma(P)^{\leq n}$ which, in turn, holds as $\varsigma$ is an injective graded morphism. The last identity is also easy to prove, taking into account the definition of $\CP$ and the fact that, for two graded $S$-subbimodules $X$ and $Y$ of $\CT$ such that $X\bgi Y=0$, we have  $(X+ Y)^{\leq n}=X^{\leq n}+ Y^{\leq n}$.
	
	By Lemma \ref{le:Phi_functorial} we get that the following commutative diagram:
	\begin{equation*}
		\xymatrix{
			\CA \ar[r]^-{\Phi_{\CP}} \ar@{->}[d]_-{\psi_A} & \mathrm{gr}\CU \ar@{->}[d]^-{\mathrm{gr} {\psi_U}} \\
			A \ar[r]_-{\Phi_P} & \mathrm{gr}U
		}
	\end{equation*}
	Since ${\psi_A}$ and ${\psi_U}$ are isomorphisms, we conclude that $P\subseteq T$ is of PBW-type if and only if $\CP\subseteq \CT$ is so.
	
	As in \S\ref{fa:P_z}, let $D:={T[z] }/{\langle P^*\rangle}$ and $\CD:=\CT[z]/{\langle \CP^*\rangle}$ be the corresponding central extensions of $A$ and $\CA$, respectively. 	Let $\psi_z:\CT[z]\rightarrow T[z] $ denote the unique morphism of graded $S$-rings that lifts $\psi$ and maps $z$ to $z$. Note that  $\mathrm{ev}_1  \psi_z=\psi  \mathrm{ev}_1$.
	
	To state and prove the last application of our main results, we need one more fact, namely that  $\psi_z$ induces an isomorphism of graded $S$-rings ${\psi}_D:\CD\rightarrow D$. Of course, in order to do that, it is enough to show that $\psi_z(\ig{\CP^*})=\ig{P^*}$ and that the kernel of $\psi_z$ is included into $\ig{\CP^*}$.
	
	We first note that, for an element $t=\sum_{i=0}^nt_i$ in $\CT$ with $t_n\neq 0$, we have
	\begin{equation}\label{eq:psizh}
	\psi_z(t^*)=\psi_z\bigg(\sum_{i=0}^nt_iz^{n-i}\bigg)=\sum_{i=0}^n\psi\left(t_i\right)z^{n-i}=z^{n-m}\psi(t)^*.
	\end{equation}
	In the above relation, the number $m\leq n$ is chosen such that  $t_{i}\in\ker(\psi)$ for all $i\geq m$, but $\psi(t_{m-1})\neq 0$. Therefore, $\psi_z(\CP^*)\subseteq \ig{P^*}$. Thus we get the inclusion  $\psi_z\left(\ig{\CP^*}\right)\subseteq \ig{P^*}$.
	
	To prove the other inclusion, we pick $x\in P$. Then $x^*=(\psi\varsigma(x))^*=\psi_z(\varsigma(x)^*)$, where for the last identity we used \eqref{eq:psizh}.  It follows that $x^*\in\psi_z(\CP^*)$. In conclusion, $\psi_z(\ig{\CP^*})=\ig{P^*}$.
	
	It remains to show that $\CK_z:=\ker(\psi_z)$ is included into the ideal generated by $\CP^*$. Let  $y=\sum_{i=0}^ny_iz^{n-i}$ be an element in $\CK_z^n$. Then each $y_i$ belongs to $\CK\subseteq \ig{\CK_0}$ and $\CK_0=\CK_0^*\subseteq\CP^*$ so $y\in \ig{\CP^*}$.
	%for all $0\leq i\leq n$. Since $\CK_0$ generates $\CK$, there are some elements $a_i^j\in\CT^{p_j}$, $b_i^j\in \CT^{r_j}$ and $x_i^j\in \CK^{q_j}$ such that $	y_i=\sum_{j=1}^{n_i} a_i^j x_i^j b_i^j$ and $p_j+q_j+r_j=i$. Thus,
	%	\begin{equation*}
	%		y_i= \sum_{j=1}^{n_i} a_i^j (x_i^j)^* b_i^j\in\ig{\CP^*},
	%	\end{equation*}
	%	so $y\in\ig{\CP^*}$.
	The proof of the existence of the isomorphism $\psi_D:\CD\rightarrow D$ is complete now.
	
	Since $S$ is separable, it is semisimple too. Thus there exists a minimal projective resolution $A\ot V_*$ of $S$ as left $A$-module.

\begin{theorem}\label{te:PBW_change_presentation}
	We keep the above notation and we assume further that $\CR=\varsigma(R)+ \CK_0$ is a bimodule of relations for $\CA$. The bimodule  $P$ is of PBW-type if and only if $H^1(^1M_*)=0$ and either $c(A)=-1$ or $P$ satisfies the Jacobi conditions $(\cJ_1)-(\cJ_{c(A)})$, where $(^1M_*, \delta_*)$ is the complex associated to the central extension $\CD$. If $\CT=T_S(T^1)$ and $\psi$ is the canonical map which lifts $\id_{T^{\leq 1}}$, then the condition $H^1(^1M_*)=0$ can be dropped.
\end{theorem}

\begin{proof}
	We can suppose that $c(A)\neq -1$. Since $\psi_A$ is an isomorphism of connected $S$-rings,  $c(\CA)=c(A)$. We know that $P\subseteq T$ is of $PBW$-type if and only if $\CP\subseteq \CT$ is of $PBW$-type. According to Theorem \ref{th:PBW}, the latter bimodule is of $PBW$-type if and only if $H^1(^1M_*)=0$ and either $c(A)=-1$ or $\CP$ satisfies the Jacobi conditions $(\cJ_1)-(\cJ_{c_{V_3}})$. By \ref{fa:notation_PBW} we can choose the resolution in such a way that $c_{V_3}=c(A)$.
	
	By Theorem \ref{Th:JacAnn}, this is equivalent to the fact that $z$ is $c(A)$-regular in $\CD$. We have seen that $\psi_D:\CD\to D$ is an isomorphism and, by construction, $\psi_D(z)=z$. Thus $z$ is $c(A)$-regular in $\CD$ if and only if $z$ is $c(A)$-regular in $D$. By applying Theorem \ref{Th:JacAnn} once again, this is equivalent to require that $P$ satisfies the Jacobi conditions $(\cJ_1)-(\cJ_{c(A)})$.
	
	Let $\CT=T_S(T^1)$. Since $S$ is separable, there exists a morphism of filtered bimodules $\alpha:\CR\to \CT$ such that $\CP=\alpha(\CR)$. We conclude the proof using Theorem \ref{co:PBW} (2).
\end{proof}

\begin{example}
	The bimodule $\CR=\varsigma(R)+\CK_0$ is not a bimodule of relations for $\CA$ in general. To see that, we reconsider the Example \ref{ex:X^3}. Thus, $T=\Bbbk[X]/\ig{X^3}$ and $A=T/\ig{x^2}\cong \Bbbk[X]/\ig{X^2}$, where $x$ is the class of $X$ in $T$. Let $\CT:= \Bbbk[X]$. Of course, with the notation of the subsection \ref{fa:notation_tensor}, we have $\CK=\ig{X^3}$. By definition, $R$  and $\CK_0$ are the one dimensional subspaces generated by $x^2$ and $X^3$, respectively. Clearly, they are spaces of relations for $A$ and $T$. The unique $\mk$-linear map $\zeta:T\to\CT$, given by  $x^i\mapsto X^i$ for $i\leq 2$, is a section of $\psi:\CT\to T$.
	
	Note that $X^2\in\CI$ and $X^3\in \CK_0$, so $X^3$ is also an element of $\CK_0\cap\wt{\mathcal{I}}$. Thus $\CR$ cannot be  a space of relations for $\CA$, see the next remark.
\end{example}

\begin{remark}
	The relation $\CK_0\bgi \wt{\CI}=0$ is a necessary condition for $\CR$ being a bimodule of relations for $\CA$, as $\CK_0\subseteq \CR$. Suppose that $R$ is a bimodule of relations for $A$ and let us check that the condition above is also sufficient. We pick  $x\in \CR\bgi \wt \CI$. Since $\psi( \CR)=R$, $\psi(\CT^1)=T^1$ and $\psi(\CI)=I$ it follows that  $\psi(x)$ is an element in $R\cap (IT^1+T^1I)$. As $R$ is a bimodule of relations for $A$, we get $\psi(x)=0$,  so $x\in \CK\bgi \CR=\CK_0$. Thus $x\in \CK_0\bgi \wt{\CI}$. This proves that $\CR\bgi \wt \CI=\CK_0\bgi \wt{\CI}$ whence the conclusion follows.
	
	Assume further that there exists some $n$ such that $\CK_0$ is $n$-pure and $R\subseteq T^{\geq n}$. Since $\CR=\varsigma(R)\oplus \CK_0$, it follows that $\CI\subseteq\CT^{\geq n}$, so $\wt \CI\subseteq \CT^{\geq n+1}$. Thus $\CK_0\cap \wt \CI=0$.
	
	For instance,  if $T$ is quadratic, that is $\CK_0\subseteq \CT^2$, then the condition $\CK_0\bgi \wt \CI=0$ is automatically verified for any $R \subseteq T^{\geq 2}$. Hence, taking $\CT:=T_S(T^1)$ in the preceding theorem we obtain the following result.
\end{remark}

\begin{theorem}\label{te:PBW_quadratic}
	Let $S$ be a separable $\mk$-algebra and let $A$ denote a graded quotient of a quadratic $S$-ring $T$. We assume that  $R\subseteq T^{\geq 2}$ is a bimodule of relations for $A$ and that $R_P=R$, for some subbimodule $P\subseteq T$. Then $P$ is of PBW-type if and only if either $c(A)=-1$ or  $P$ satisfies $(\cJ_1)-(\cJ_{c(A)})$.
\end{theorem}	

Note that the polynomial ring $T:=S[X_1,\dots,X_n]$ can be seen as a quadratic $S$-ring $T=T_S(S^n)/\ig{\CK_0}$. If $\{e_i\}_{i\leq n}$ denotes the canonical basis of the left $S$-module $S^n$, then $\CK_0$ is the  $S$-bimodule spanned by $\{e_i\ot e_j-e_j\ot e_i\}_{1\leq j<i\leq n}$. Therefore, we also have the following.

\begin{corollary}\label{co:PBW_quadratic}
	Let $S$ be a separable $\mk$-algebra and let $A$ denote a graded quotient of the polynomial ring $T:=S[X_1,\dots,X_n]$. We assume that  $R\subseteq T^{\geq 2}$ is a bimodule of relations for $A$ and that $R_P=R$, for some subbimodule $P\subseteq T$. Then $P$ is of PBW-type if and only if either $c(A)=-1$ or  $P$ satisfies $(\cJ_1)-(\cJ_{c(A)})$.
\end{corollary}

\subsection{Examples.}\label{ssec:examples}\hfill\vspace*{1.3ex}

We end this paper by presenting some examples. We start by estimating the homological complexity of some classes of strongly graded connected $S$-rings.

\begin{fact}[An upper bound for $c(A)$.]\label{co:3n(A)}
	Let $A$ be a strongly graded connected $S$-ring. Assume that there exists a bimodule of relations $R\subseteq T_A$ for $A$ and that $A^1$ and $R$ are projective as left $S$-modules. By \S\ref{fa:special_resolution}, we know that $S$ has a resolution $(A\ot V_*,d_*)$ in $\mathfrak{P}$ such that $V_2=R$, $S\ot_A d_3=0$ and $c(A)=c_{V_3}$.
	
	Let $c_A$ be the complexity of $A$ as a graded $S$-module,  $c_A=\sup\{n-1\mid A^n\neq 0\}$. We will show that  $c(A)\leq c_A+2$. Note that $c_A$ is finite, provided that $A$ is a finitely generated left or right $S$-module. Indeed, in this case, we can choose a finite set of homogeneous elements $a_1,\dots,a_n$ of $A$ as a left $S$-module. If we assume that $a_n$ is the element of highest degree, say $d$, then $c_A=d-1$.	

	We have $R\oplus \A^+K_1=K_1$, where $K_1:=\ker(d_1)$. Let us pick $k\geq c_A+3$. Thus  $R^k =0$, since $K_1^k \subseteq\left(\A\otimes A^1\right)^k=A^{k-1}\ot A^1=0$. Therefore,
\begin{equation*}
	K_2^k\subseteq\left(A\otimes R\right)^k  =  \sum_{p=0}^{c_A+2}A^{k-p}\otimes R^p
	=\sum_{p=0}^{c_A)+2}A^{k-c_A-3}A^{c_A+3-p}\otimes R^p
	= \A^{k-c_A-3}\left(A\otimes R\right)^{c_A+3}.
\end{equation*}
	Since $d_2$ is a graded map,  $d_2^{c_A+3}:\left(A\otimes R\right)^{c_A+3}\to \left(A\otimes A^1\right)^{c_A+3}$ is zero. Thus $K_2^{c_A+3}$ coincides with $\left(A\otimes R\right)^{c_A+3}$. Henceforth,  for $k\geq c_A+4$, we have $K_2^k\subseteq \A^{k-c_A-3}K_2^{c_A+3}$. Therefore, $K_2^k=A^1K_2^{k-1}$ and, as noticed in \S\ref{fa:special_resolution}, we get $\tor^A_{3,k}(S,S)=0$. Thus $c(A)\leq c_A+2$.

	Let $n$ be a positive integer. The $\Bbbk$-algebra $A:=\Bbbk[X]/\ig{X^n}$  is $n$-Koszul and $c_A=n-2$. On the other hand, it is well-known that the homological complexity of an $n$-Koszul algebra is $n$, so the upper bound of $c(A)$ is reached.
\end{fact}

\begin{fact}[Split central extensions.] \label{rem:trivcomplex}
	Let $D$ be a central extension of a strongly graded connected $S$-ring $A$, as in Section \ref{sec:3}. We assume that $A^+$ is projective as an $S$-bimodule, so  the projection $\pi:D\to A$  has an $S$-bilinear section $\sigma:=\id_S\oplus\sigma^+$, where $\sigma^+$ is a section of $\pi^+:D^+\to A^+$. We fix a resolution as in \eqref{eq:special_resolution}. Hence, the maps $f_n$ satisfying \eqref{eq:partial_n}  can be defined in an algorithmic way as follows.
	
	Since, for $a,b\in A$, the element $\sigma(ab)-\sigma(a)\sigma(b)$ is in $\ker(\pi)=zD$, we can consider the diagram:
\begin{equation*}
	\xymatrix @C=25pt @R=20pt{& A \otimes A \ar[d]^-{\sigma  m_A - m_D  (\sigma \otimes\sigma)} \ar@{.>}[dl]_-{\omega} \\
		D(-1) \ar[r]^-{z\cdot } & zD \ar[r] & 0 	}
\end{equation*}
	Note that $A^+\ot A^+$ is projective as an object in the category of graded $(S,S)$-bimodules. Thus there is a graded $S$-bilinear map $\varfun{\omega^+}{A^+ \otimes A^+}{D(-1)}$ such that the restriction of $\sigma  m_A - m_D  (\sigma \otimes\sigma)$ to $A^+\ot A^+$ is $z\cdot\omega^+$. This map can be extended to an $S$-bilinear map which vanishes on $S\ot A+A\ot S$ and that makes commutative the diagram above.
	
	We use the notation introduced in \S\ref{ssec:u-v}. Then $\wh{\partial}_n=(\sigma\ot V_{n-1})\wh{d_n}$. In view of relation \eqref{eq:f_hat} we get:
\begin{align*}
(\partial _{n-1}\partial _{n})^{\wh{}}& =(m_D\ot V_{n-2})(D\ot \wh{\partial}_{n-1})\wh{\partial}_n\\&=(m_D\ot V_{n-2})\big[D\ot(\sigma \ot V_{n-2})\wh{d}_{n-1}\big](\sigma \ot V_{n-1})\wh{d}_{n}\\
	&=\big[m_D(\sigma\ot\sigma)\ot V_{n-2}\big](A\ot\wh{d}_{n-1})\wh{d}_n	\\
	&=(\sigma\ot V_{n-2})( m_A\ot V_{n-2})(A\ot\wh{d}_{n-1})\wh{d}_n-(z\cdot\omega\ot V_{n-2})(A\ot\wh{d}_{n-1})\wh{d}_n\\
	&=(\sigma\ot V_{n-2})[d _{n-1}d_{n}]^{\;\wh{}}-(z\cdot\omega\ot V_{n-2})(A\ot\wh{d}_{n-1})\wh{d}_n\\
	& =-(z\cdot\omega\ot V_{n-2})(A\ot\wh{d}_{n-1})\wh{d}_n.
\end{align*}
	Note that the last relation follows by the fact that the maps $d_*$ are the differentials of a complex.
	Now we can define $f_n$ in a unique unique way such that
\begin{equation}\label{eq:deffn}
	\wh{f}_{n} :=(-1)^n (\omega\ot V_{n-2})(A\ot\wh{d}_{n-1})\wh{d}_n.
\end{equation}
	Clearly $f_n$ satisfies the relation \eqref{eq:partial_n},  since $[ \partial _{n-1}\partial _{n}]^{\;\wh{}}=[(-1)^{n-1}z\cdot f_n]^{\;\wh{}}$.
	
	For proving that $(^2M_\ast, \delta_\ast)$ is a complex, it is sufficient to check that  \eqref{eq:f_n} holds for $n=2$. Using a section $\sigma$ and the corresponding map $\omega$ as above, this can be done proceeding as follows. We choose the map $f_n$ as in  \eqref{eq:deffn}, for all $n\geq 2$. Thus,
\begin{align*}
(\partial_{1} f_{3})^{\wh{}} =(m_D\otimes V_{1})(D\otimes \wh{\partial}_1) \wh{f}_3 	& =  -(m_D\otimes S)\big[D\otimes (\sigma\otimes S) \wh{d}_1)\big](\omega\otimes V_1)(A\otimes \wh{d}_{n}) \wh{d}_3\\
	& = -\big[m_D(\omega\otimes \sigma )\otimes S\big](A\otimes A\otimes \wh{d}_1)(A\otimes \wh{d}_2) \wh{d}_3.
\end{align*}
	In a similar way, we get:
\begin{align*}
	(f_2\partial_{3})^{\wh{}}&=(m_D\otimes A)(D\otimes \wh{f}_2) \wh{\partial}_3 = (m_D\otimes S)\big[D\otimes (\omega\otimes S)(A\otimes \wh{d}_1)\wh{d}_n)\big](\sigma\otimes V_2)\wh{d}_3 \\
	& = \big[m_D(\sigma\otimes \omega)\otimes V_1\big](A\otimes A\otimes \wh{d}_1)(A\otimes \wh{d}_2)\wh{d}_3.
\end{align*}
	It follows that $(^2M_*,\delta_*)$ is a complex if and only if
\begin{equation}\label{eq:complex}
	\big[m_D
	(\sigma\otimes \omega-\omega\otimes \sigma)\otimes S\big]
	\varrho=0,
\end{equation}
	where $\varrho:= (A\otimes A\otimes \wh{d}_1)(A\otimes \wh{d}_2)\wh{d}_3$. On the other hand, if we assume that
\begin{equation}\label{eq:complex'}
	m_D (\sigma\otimes \omega-\omega\otimes \sigma)=\omega (m_A\otimes A-A\otimes m_A),
\end{equation}
	then we get
\begin{align*}
	\big(m_D
	(\sigma\otimes \omega-\omega\otimes \sigma)\otimes S\big)\varrho&=
	\big(\omega (m_A\otimes A-A\otimes m_A)\otimes S\big)\varrho \\
	&=(\omega \otimes S)\big[(A\otimes \wh{d}_1)(d_2d_3)^{\wh{}}-\big(A\otimes (d_1d_2)^{\wh{}}\;{}\big)\wh{d}_3\big]=0.
\end{align*}
	In conclusion, the relation \eqref{eq:complex'} is a sufficient condition for $(M_*,\delta_*)$ to be a complex. For instance, if $\pi$ has a section which is a morphism of graded $S$-rings, then we can choose $\omega=0$, so \eqref{eq:complex'} holds and, consequently, $(M_*,\delta_*)$ is a complex.
	
	As a final remark we mention that $\im(\varrho)\subseteq (A^+)^{\ot 3}$,  provided that the image of  $ d_n$ is included into $A^+\ot V_{n-1}$ for $n\leq 3$.  In this case, for proving that $(M^2_*,\delta_*)$ is a complex, it is sufficient to check that  \eqref{eq:complex'} holds as a relations between two maps defined on $(A^+)^{\ot 3}$.

\end{fact}\vspace*{1ex}

Let us conclude this section with a couple of examples to show that the vanishing condition on $\h_1(^1M_\ast)$ cannot be dropped in this general framework.

\begin{example}[Trivial central extensions]\label{ex:trivial_extension}
	Let $A$ denote a strongly graded connected $S$-ring and let  $V$ be a strongly graded $A$-bimodule (by definition, $V$ is called \textit{strongly graded} if  $V^{n+1}=A^1V^n+V^nA^1$ for all $n>0$).  On  the graded   $A$-bimodule $D=A\oplus V$ we define  the product:
	\[
	(a,v)\cdot (a',v')=(aa', av'+va'),
	\]
	for all $a,a'\in A$ and $v,v'\in V$. Therefore, $D$ is the \textit{graded trivial extension of $A$ with kernel $V$}. We regard $D$ as an $S$-ring via the restriction of scalars. Clearly, $D$ is strongly graded as $A$ and $V$ are so, and $D$ is connected because $V^0=0$.
	
	We assume that there exists $z_V\in V^1$ such that $az_V=z_Va$, for all $a\in A$. Thus the couple $z:=(0,z_V) $ is a central element of degree $1$ in $D$. The ideal generated by $z$ in $D$ coincides with $Az_V=z_VA$, so $D$ is a central extension of $\wt A:=A\oplus (V/Az_V)$. Henceforth, if the canonical projection $V\to V/(Az_V)$ has an $A$-bilinear  section $\sigma$, then  $\wt \sigma:\wt A\to D$ is a morphism of graded $S$-rings  and a section of $\wt \pi:D\to \wt A$.
	
	In view of Remark \ref{rem:trivcomplex}, for the above central extension, we can take $\omega$ to be the trivial, so $(^2M_\ast)$ is a complex. However,  $z$ is regular in $D$ if and only if $\h_1(^1M_\ast)=0$, cf. Theorem \ref{th:mainth}.
	
	Let us consider the particular case when $A$ is itself a trivial extension of $S$ with kernel $W$, where $W$ is a graded $S$-bimodule such that $W=W^1$. We take $V:=A(-1)$ and $z_V=1\in A(-1)^1$. Obviously, $Az_V=z_VA=V$. Thus $V/Az_V=0$, so $\wt A=A$ and the canonical projection from $V$ to $V/(Az_V)$ trivially has an $A$-bilinear  section. On the other hand, in this case, we have $z^2=0$, by the definition of multiplication in $D$. Hence $z$ is not regular. By Theorem \ref{th:mainth}, the module $\h_1(^1M_\ast)$ does not vanish.
\end{example}

\begin{example}\label{ex:X^3}
	We now take $T:=\Bbbk[X]/\ig{X^3}$. Let  $x$ and $R$ denote the class of $X$ in $T$ and the linear space $\Bbbk x^2$, respectively. Clearly, $R$ is a minimal subspace of generators $\ig{x^2}\subseteq T$. Equivalently, $R$ is a space of relations for this algebra, as $S=\Bbbk$ is separable as an algebra over itself. Let $A:=T/\ig{x^2}$ and $D:=T[z]/\ig{x^2+z^2}$.
	Note that
	\[
	A\cong \Bbbk[X]/\ig{X^2}\quad\text{and}\quad
	D=\Bbbk[X,Z]/\ig{X^3,X^2+Z^2}.
	\]
	Note also that in $D$ we have $xz^2=- x^3=0$, so $z$ is not regular.
	
	Furthermore, the unique linear map $\sigma:A\to D$ satisfying the relations $\sigma(1)=1$ and $\sigma(x)=x$ is a section of the canonical projection $\pi:D\to A$ that maps $z$ to $0$. It is easy to see that $\sigma m_A-m_D(\sigma\otimes \sigma)$ maps  the element $x\otimes x$, which forms a basis for $A^+\otimes A^+$, to $-x^2=z^2$. Hence the map $\omega$ constructed in \S\ref{rem:trivcomplex} may be chosen such that
	$ \omega(x\otimes x)=z$ and it vanishes on $S\ot A+A\ot S$. Since
	\begin{equation*}
	\sigma(x)\omega(x\otimes x)-\omega(x\otimes x)\sigma(x)=0=\omega(x^2\otimes x-x\otimes x^2),
	\end{equation*}
	we get that \eqref{eq:complex'} holds on $\left(A^+\right)^{\otimes 3}$. According to \S\ref{rem:trivcomplex}, we conclude that $(^2M_\ast,\delta_*)$ is a complex.
	
	Remark that $D$ is the central extension associated  to $\alpha:R\to T$ as in \S\ref{fa:extension}, where $\alpha$ is the unique linear map such that $\alpha(x^2)=x^2+1$. In conclusion, the condition $\h_1(^1M_\ast)=0$   cannot be dropped in Theorem  \ref{th:mainth}, if $T$ is not a tensor $S$-ring,  even for central extensions associated to filtered maps.
	
	The same central extension yields us another counterexample. We first note that $A$ and $D$ are isomorphic to $A':=\mk[X]/\ig{X^2,X^3}$ and $D':=\mk[X,Z]/\ig{ X^3, X^2+Z^2}$, respectively. We denote the corresponding canonical  isomorphisms by $\eta_A:A\to A'$ and $\eta_D:D\to D'$.  We can now regard $D'$ as the central extension of $A'$ associated to the filtered map $\alpha':R' \to \mk[X]$, where $R'$ is spanned by $\{X^2,X^3\}$,  and  $\alpha'$ is the unique filtered morphism such that $\alpha'(X^3)=X^3$ and $\alpha'(X^2)=X^2+1$. Clearly, $D'$ is a split central extension of $A'$ with respect to the $S$-bimodule section $\sigma':=\eta_D\sigma\eta_A^{-1}$  of the projection $\pi':D'\to A'$.  The map $\omega':A'\ot A'\to D'(-1)$, associated to $\sigma'$ as in \S\ref{rem:trivcomplex} may be chosen to be $\eta_D\omega(\eta_A^{-1}\ot\eta_A^{-1}) $. Thus, it satisfies the relation \eqref{eq:complex'}, since the above $\omega$ does it. Proceeding as above, we can show that   $(^2M'_*,\delta'_*)$  is a complex. Nevertheless,  $H_1(M'_*)\neq 0$. This conclusion does not contradict Theorem \ref{prop:H10} which cannot be applied as $R'$ is not a space (bimodule) of relations for $A'$.\vspace*{1ex}
	
	\noindent\textbf{Acknowledgments.}  This article was written while A. Ardizzoni and P. Saracco were members of the "National Group for Algebraic and Geometric Structures, and their Applications" (GNSAGA-INdAM). P. Saracco was partially supported by the Erasmus program while he was visiting the University of Bucharest in 2016. He would like to heartily thank the members of the Faculty of Mathematics and Computer Science of the University of Bucharest for their warm hospitality and their friendship during his stay. D. \c{S}tefan was partially supported by INdAM, while he was visiting professor at University of Torino in 2015.
	
%	The same central extension yields us another counterexample. We first note that $A$ and $D$ admits the following presentations: $A:=\mk[X]/\ig{X^2,X^3}$ and $D:=\mk[X,Z]/\ig{ X^3, X^2+Z^2}$. We can now regard $D$ as the central extension of $A$ associated to the filtered map $\alpha:R' \to \mk[X]$, where $R'$ is spanned by $\{X^2,X^3\}$,  and  $\alpha$ is the unique filtered morphism such that $\alpha(X^3)=X^3$ and $\alpha(X^2)=X^2$. Since we work with the same rings as above, $(^2M_*,\delta_*)$ is a complex and $H_1(M_*)\neq 0$. The explanation of this fact is that $R$ is not a space (bimodule) of relations for $A$. \rc{[Dragos, if $D$ is associated to $\alpha$, the reader could think we are implicitly taking as a starting resolution of $S$ the one in which $V_2=R'$ In this case we cannot say that $(^2M_*,\delta_*)$ is the same as before where $R$ is different. Maybe it is better to point out we are constructing $(^2M_*,\delta_*)$ starting from the resolution having $V_2=R$ as before. ]}
	\end{example}

\end{document}